\documentclass[11pt]{amsart}
\usepackage[utf8]{inputenc}
\usepackage{amssymb}
\usepackage{graphicx}
\usepackage{url}
\usepackage{enumerate}
\usepackage{enumitem}
\usepackage{verbatim}%This gives comment 
\usepackage{color}
\usepackage{tikz}
\usetikzlibrary{arrows.meta,bending}
\usetikzlibrary{decorations.pathmorphing}
\usetikzlibrary{decorations.markings}
\usepackage{array}
\usepackage{tikz-cd}
\usepackage{amsmath}
\usepackage{geometry}
\usepackage{amsthm}
\usepackage[colorlinks=true,linkcolor=purple,citecolor=violet]{hyperref}
\usepackage{mathtools}

\usepackage{libertine}
\usepackage{libertinust1math}
\usepackage[T1]{fontenc}

\usepackage{lipsum}
\newcommand\blfootnote[1]{%
  \begingroup
  \renewcommand\thefootnote{}\footnote{#1}%
  \addtocounter{footnote}{-1}%
  \endgroup
}

\usepackage{mathtools}

\newtheorem{theorem}{Theorem}[section]
\newtheorem{lemma}[theorem]{Lemma}
\newtheorem{corollary}[theorem]{Corollary}
\newtheorem{proposition}[theorem]{Proposition}
\newtheorem{conjecture}[theorem]{Conjecture}
\newtheorem{remark}[theorem]{Remark}
\newtheorem*{claim}{Claim}
\theoremstyle{definition}
\newtheorem{definition}[theorem]{Definition}

\newtheorem{example}[theorem]{Example}

\newtheorem*{definitionproposition}{Definition-Proposition}

\newtheorem*{acknowledgement}{Acknowledgement}
\newtheorem*{notation}{Notations}
\newtheorem*{idea}{Main ideas and organization of the paper}

\title{Classification results for conformally K\"ahler Gravitational instantons}
\author{Mingyang Li}
\date{}

\newcommand{\Addresses}{{% additional braces for segregating \footnotesize
  \bigskip
  \footnotesize

  \textsc{{Department of Mathematics, University of California, Berkeley,
    CA 94720, USA.}}\par\nopagebreak
  \textit{E-mail address}: \href{mailto:mingyang_li@berkeley.edu}{\texttt{mingyang\_li@berkeley.edu}}

}}

\begin{document}

\begin{abstract} We investigate the asymptotic geometry of %prove the existence and classification of all asymptotic models of 
Hermitian non-K\"ahler Ricci-flat metrics with finite $\int|Rm|^2$ at infinity. Specifically, we prove:
\begin{itemize}
\item %Any asymptotic cone at infinity is one of the following: $\mathbb{C}^2/\Gamma$, $\mathbb{R}^3/\mathbb{Z}_k$, the half-plane $\mathbb{H}$, ray $[0,\infty)$. 
Any such metric is asymptotic to an ALE, ALF-$A$, AF, skewed special Kasner, $\text{ALH}^*$ model at infinity.

\item
Any Hermitian non-K\"ahler gravitational instanton with non-Euclidean volume growth is one of the following: the Kerr family, the Chen-Teo family, the Taub-bolt space, the reversed Taub-NUT space. This particularly confirms a conjecture by Aksteiner-Andersson. It includes the well-known Kerr family from general relativity.

\item
All Hermitian non-K\"ahler gravitational instantons can be compactified to log del Pezzo surfaces. This explains a curious relation to compact Hermitian non-K\"ahler Einstein 4-manifolds. 
\end{itemize}
For a 4-dimensional Ricci-flat metric, being Hermitian non-K\"ahler is equivalent to being \textit{non-trivially conformally K\"ahler}.
\end{abstract}

\maketitle

\tableofcontents

\section{Introduction}

\blfootnote{This work is partially supported by NSF grant DMS-2004261, DMS-2304692, and Simons Collaboration Grant on Special Holonomy in Geometry, Analysis, and Physics (488633, S.S.). Part of this work is finished during the author's visit to IASM, Zhejiang University.}

A \textit{gravitational instanton} $(M,h)$ is defined as a (non-flat) complete non-compact Ricci-flat 4-manifold with finite energy 
$$\int_M|Rm_h|_h^2dvol_h<\infty.$$
In this paper, we investigate the asymptotic structure of Hermitian non-K\"ahler Ricci-flat metrics at infinity, along with Hermitian non-K\"ahler gravitational instantons that collapse at infinity. We will provide a complete classification of the asymptotic structures of Hermitian non-K\"ahler Ricci-flat metrics at infinity, under the condition of finite $\int |Rm|^2$, as well as a classification of Hermitian non-K\"ahler gravitational instantons.

We begin by discussing the naturality of the condition that gravitational instantons are \textit{Hermitian non-K\"ahler}, a condition analogous to being hyperk\"ahler. For 4-dimensional Ricci-flat metrics, the condition of being Hermitian non-K\"ahler is equivalent to being \textit{non-trivially conformally K\"ahler}.
In general, oriented gravitational instantons can be broadly classified into three types, based on the number of distinct eigenvalues of the self-dual Weyl curvature $W^+$. This classification, originally introduced by Derdzi\'{n}ski \cite{derdzinski}, also applies to 4-dimensional Einstein metrics.

\begin{definitionproposition}
By treating the self-dual Weyl curvature $W^+$ as an endomorphism $W^+:\Lambda^+\to\Lambda^+$ of the bundle of self-dual two-forms, 4-dimensional Ricci-flat metrics can be classified into the following three types:
\begin{itemize}
\item Type I: If $W^+ \equiv 0$, then the metric is anti-self-dual.
\item Type II: If $W^+$ has exactly two distinct eigenvalues, treated as an endomorphism $W^+:\Lambda^+\to\Lambda^+$ everywhere, then up to passing to a double cover\footnote{It was pointed out to the author by Claude LeBrun that in the definition one should include passing to a double cover. The author is very grateful for that.}, there exists a compatible complex structure $J$ such that $(M,h,J)$ is Hermitian and the conformal metric $g=\lambda^{2/3}h$ is K\"ahler, where $\lambda\coloneqq 2\sqrt{6}|W^+|_h$. 
\item Type III: If $W^+$ generically has three distinct eigenvalues, then $(M,h)$ with the given orientation is not locally conformally K\"ahler.
\end{itemize}
\end{definitionproposition}

Type I and Type II are special as the Ricci-flat metric has extra geometric structures:
\begin{itemize}
  \item Gravitational instantons of Type I are equivalent to being \textit{anti-self-dual} and are, in particular, \textit{locally hyperk\"ahler}.
  \item Gravitational instantons of Type II are equivalent to, up to passing to a double cover, being \textit{Hermitian non-K\"ahler}. We will actually see there is no need to pass to double cover in the situation of gravitational instantons, therefore being Type II is equivalent to being Hermitian non-K\"ahler. They are moreover \textit{conformally K\"ahler}. The K\"ahler metric $g = \lambda^{2/3} h$ is, in fact, \textit{extremal K\"ahler} in the sense of Calabi, with scalar curvature $s_g = \lambda^{1/3}$ and extremal vector field $\mathcal{K} = J \nabla_g \lambda^{1/3} = -J \nabla_h \lambda^{-1/3}$.
\end{itemize}

As the first main result of our paper, we classify all asymptotic cones as well as asymptotic models of Hermitian non-K\"ahler Ricci-flat metrics with finite $\int|Rm|^2$. 

\begin{theorem}\label{thm:main classification ends nonrealizable}
  Any Hermitian non-Kähler Ricci-flat metric on an end that is complete at infinity and has finite $\int |Rm|^2$ is asymptotic, with a polynomial rate, to one of the following asymptotic models at infinity: ALE, ALF-$A$, AF, skewed special Kasner, $\text{ALH}^*$. The corresponding asymptotic cones are $\mathbb{R}^4/\Gamma$, $\mathbb{R}^3$, $\mathbb{R}^3/\mathbb{Z}_k$ or the half-plane $\mathbb{H}$, $[0,\infty)$, $[0,\infty)$.
\end{theorem}

A detailed introduction to these asymptotic models will be provided in Section \ref{subsec:define asymptotic models}. Note that this also gives a classification of the asymptotic geometry of Type II Ricci-flat metrics with $\int|Rm|^2<\infty$. Any such Type II Ricci-flat end can only be asymptotic to a further $\mathbb{Z}_2$-quotient of the asymptotic models above. 
Here, for $\mathbb{R}^4/\Gamma$, the group $\Gamma\subset SO(4)$ can only be certain special subgroups, as discussed in more detail in Section \ref{subsec:define asymptotic models}. For $\mathbb{R}^3/\mathbb{Z}_k$, $\mathbb{Z}_k$ acts as rotations along a fixed axis on $\mathbb{R}^3$. For $\text{ALH}^*$ model it turns out to have a rate faster than any polynomial rate. As we will see, neither skewed special Kasner models nor $\text{ALH}^*$ models can be filled in as Hermitian non-K\"ahler gravitational instantons. Interestingly, however, there exist Type III gravitational instantons that are asymptotic to special Kasner model \cite{yamada}. This seems to be the only known class of Type III gravitational instantons so far, though they have infinite fundamental group and are not asymptotically flat. Specifying to the situation of gravitational instantons, we have

\begin{theorem}\label{thm:asymptotic geometry}
  Any Hermitian non-K\"ahler gravitational instanton is asymptotic to a model at infinity. Specifically there are the following situations:
  \begin{itemize}
    \item The gravitational instanton is ALE, with asymptotic cone $\mathbb{R}^4/\Gamma$.
    \item The gravitational instanton is ALF-$A$, with asymptotic cone $\mathbb{R}^3$.
    \item The gravitational instanton is AF, with asymptotic cone $\mathbb{R}^3/\mathbb{Z}_k$ or $\mathbb{H}$.
  \end{itemize}
\end{theorem}
Note that the skewed special Kasner models and the $\text{ALH}^*$ models cannot be filled in by Theorem \ref{thm:asymptotic geometry}. Theorem \ref{thm:main classification ends nonrealizable}-\ref{thm:asymptotic geometry} are surprising as it is in \textit{sharp contrast} with the hyperk\"ahler case. By the recent work of \cite{sz}, any hyperk\"ahler gravitational instanton must be asymptotic to an ALE, ALF, ALG, ALH, $\mathrm{ALG}^*$, $\mathrm{ALH}^*$ model \cite{sz}, which a prior all could be asymptotic models of Hermitian non-K\"ahler gravitational instantons. Our theorem above rules out most of them for Hermitian non-K\"ahler gravitational instantons.

\

We now turn to the classification of Hermitian non-K\"ahler gravitational instantons. Due to lots of previous work, an entire classification for hyperk\"ahler gravitational instantons is completed \cite{ah88,biquardminerbe,ch05,ck99,minerbe2010,minerbe,cc17,cc,cc21,ck02,cv21,cjl22,hei12,hsvz22,kro89a,kro89b,CVZ,leelin}. For Hermitian non-K\"ahler gravitational instantons, known examples and their topological types are:
\begin{itemize}
  \item ALE gravitational instantons: 
  \begin{itemize}
    \item The Eguchi-Hanson metric on $T^*\mathbb{P}^1$ with reversed orientation.
  \end{itemize}
\end{itemize}
\begin{itemize}
\item AF gravitational instantons:
\begin{itemize}
\item Kerr metrics on $\mathbb{R}^2\times S^2$;
\item Chen-Teo metrics on $\mathbb{P}^2\setminus S^1$.
\end{itemize}
\item ALF gravitational instantons:
\begin{itemize}
\item the Taub-NUT metric on $\mathbb{C}^2$ with reversed orientation;
\item the Taub-bolt metric on $\mathbb{P}^2\setminus\{point\}$ with both orientations.
\end{itemize}
\end{itemize}
There is a complex structure on the usual Taub-NUT space that is compatible with the metric, whose orientation is opposite to the standard hyperk\"ahler orientation. This is what we mean by \textit{reversed Taub-NUT}. The situation is similar for \textit{reversed Eguchi-Hanson}. These are conventional ways to interpret the topology of these gravitational instantons, as can be found in previous literatures. A more familiar example is the Schwarzschild metric, which is one of the metrics in the Kerr family. These examples all can be explicitly written down. See \cite{gibbons}, \cite{conjecture}, and \cite{chenteo} for example. They are all toric, in the sense that there is an effective $T^2$-action that acts holomorphically and isometrically. 
Aksteiner and Andersson \cite{conjecture} were the first to observe that Chen-Teo spaces are Hermitian non-K\"ahler, and they conjectured that the ALF/AF gravitational instantons in the above list constitute the complete set of Hermitian non-K\"ahler ALF/AF gravitational instantons. 
Toric Hermitian non-K\"ahler ALF/AF gravitational instantons have already been classified by Biquard-Gauduchon \cite{biquardgauduchon}, with an additional assumption on the decay rate $\delta_0$ of the gravitational instanton. 
Chen-Teo spaces are also studied in details in \cite{biquardgauduchonb}.
\begin{conjecture}[Aksteiner-Andersson]
If $(M,h)$ is a Hermitian non-K\"ahler ALF/AF gravitational instanton, it must be one of the Kerr metrics, the Chen-Teo metrics, the Taub-bolt metric with both orientations, or the reversed Taub-NUT metric.
\end{conjecture}

In this paper, as the second main result, we shall prove
\begin{theorem}\label{thm:main}
Any Hermitian non-K\"ahler ALF/AF gravitational $(M,h)$ must be holomorphically isometrically toric, in the sense that there is an effective $T^2$-action that acts holomorphically and isometrically on $(M,h)$.
\end{theorem}

Therefore, combined with the previous result by Biquard-Gauduchon \cite{biquardgauduchon}, our classification of asymptotic models Theorem \ref{thm:asymptotic geometry}, and our preivous work on the ALE case \cite{me}, we obtain the following classification result of Hermitian non-K\"ahler gravitational instantons. This particularly gives an affirmative answer to the aforementioned conjecture, and classifies all Hermitian non-K\"ahler gravitational instantons that collapse at infinity explicitly. 
\begin{theorem}
Any Hermitian non-K\"ahler gravitational instanton $(M,h)$ must be in one of the following situations:
\begin{itemize}
  \item ALE, in which case it corresponds to a special Bach-flat K\"ahler orbifold in the terminology of \cite{me}.
  \item ALF, in which case it is either the reversed Taub-NUT metric, or the Taub-bolt metric with both orientations.
  \item AF, in which case it is one of the Kerr metrics or the Chen-Teo metrics. 
\end{itemize}
\end{theorem}

Note that it is direct to check all gravitational instantons above have no further $\mathbb{Z}_2$-quotients, so they already constitute a complete list of Type II gravitational instantons.
The conjecture in \cite{conjecture} was formulated based on a comparison to the compact case. A complete classification for compact Hermitian non-K\"ahler Einstein 4-manifolds was obtained by LeBrun through a series of works \cite{lebrun95,cxx,lebrun12}, where all metrics turn out to be toric. 
Our key observation lies in a compactification phenomenon for  Hermitian non-K\"ahler ALF/AF gravitational instantons. The compactified surfaces enjoy certain positivity properties, which is similar to what happens in the non-collapsed case of Hermitian non-K\"ahler gravitational instantons \cite{me}. An intriguing outcome that we will establish while proving our second result is the following theorem.

\begin{theorem}\label{thm:partialcompactify}
Any Hermitian non-K\"ahler ALF or rational AF gravitational instanton $(M,h)$ can be naturally compactified as a toric log del Pezzo surface $\overline{M}$ (with orbifold singularities in the rational AF case) by introducing a  holomorphic sphere $D$ at infinity.
\end{theorem}
The compactification is natural in the sense that, in these situations, the extremal vector field $\mathcal{K}$ induces a holomorphic $S^1$-action, which can be extended to the compactified surface, and the added holomorphic sphere $D$ is a fixed curve under the extended $S^1$-action. It is also compatible with the K\"ahler metric $g$ in the sense that, $g$ will be shown to be asymptotic to a cusp bundle over sphere, and the compactification is done by compactifying each cusp.  There are also Hermitian non-K\"ahler \textit{irrational} AF gravitational instantons that can be compactified. However, there is no natural way to do so, since there is no way to compactify by adding a fixed curve. 
From our compactification perspective, the most natural way to comprehend the underlying complex structure of Taub-bolt, Kerr, and Chen-Teo spaces is to consider them as metrics on $\mathcal{O}_{\mathbb{P}^1}(1)$ (or $\mathcal{O}_{\mathbb{P}^1}(-1)$ if one reverses the orientation), $\mathbb{C}^1\times\mathbb{P}^1$, and $Bl_p(\mathbb{C}^1\times\mathbb{P}^1)$, respectively.

Our result also establishes a notable dichotomy between Hermitian non-K\"ahler ALF/AF gravitational instantons and compact Hermitian non-K\"ahler Einstein 4-manifolds classified by LeBrun.
LeBrun in \cite{lebrun95} proved that compact Hermitian non-K\"ahler Einstein metrics only reside on del Pezzo surfaces with continuous holomorphic symmetry. And it was shown by LeBrun \cite{lebrun12} that the Page metric \cite{page} and the Chen-LeBrun-Weber metric \cite{cxx} are all the compact Hermitian Einstein 4-dimensional metrics. In the non-compact situation, our result demonstrates that Hermitian non-K\"ahler ALF/AF gravitational instantons can all be compactified to toric log del Pezzo surfaces by adding one rational curve. A similar compactification phenomenon occurs for Hermitian non-K\"ahler ALE gravitational instantons, as demonstrated in our prior work \cite{me}. %In \cite{me}, we have already proven that Hermitian non-K\"ahler ALE gravitational instantons have their ends biholomorphic to  $B^*/\Gamma$, where $B^*\subset \mathbb{C}^2$ is the standard punctured ball and $\Gamma\subset U(2)$, and they all can be compactified to log del Pezzo surfaces. We also proved the non-existence of Hermitian non-K\"ahler ALE gravitational instantons except the reversed Eguchi-Hanson in cases that $\Gamma\subset SU(2)$. 
As a consequence of Theorem \ref{thm:asymptotic geometry}, Theorem \ref{thm:partialcompactify}, and \cite{me}, we have the following corollary.
\begin{corollary}[Algebracity]\label{cor:compactify yau conjecture}
  For any given Hermitian non-K\"ahler gravitational instanton $(M,h)$, there is an algebraic surface $X$ and a divisor $D$ such that $M$ is biholomorphic to $X\setminus D$. 
\end{corollary}
Corollary \ref{cor:compactify yau conjecture} should be compared to a compactification conjecture proposed by S-T. Yau \cite{problemyau}. Yau proposed:
\begin{quote}
  Let $M$ be a complete K\"ahler manifold with zero Ricci curvature. Prove that $M$ is a Zariski open set of some compact K\"ahler manifold. If this is true, we shall have algebraic means to classify these manifolds.
\end{quote}
Our corollary shows that any Type II gravitational instanton is a Zariski open subset of an algebraic surface. In \cite{sz}, combining with previous works, it has been proved that for any hyperk\"ahler gravitational instanton, there is one complex structure on it, making it is algebraic in the sense of Corollary \ref{cor:compactify yau conjecture}.

Finally, we would like to mention the recent work by Aksteiner et al. \cite{and23}, showing that when an ALF gravitational instanton with $S^1$-symmetry has the same topology as the Kerr or Taub-bolt spaces, then it must be Kerr or Taub-bolt. They also showed that when it has same topology as Chen-Teo spaces, then it must be Hermitian. There is the very recent work concerning the stability of Hermitian non-K\"ahler ALF gravitational instantons by Biquard-Ozuch \cite{biquardozuch}, and the work by Biquard-Gauduchon-LeBrun \cite{claude} studying the deformation of Hermitian non-K\"ahler ALF gravitational instantons.

\begin{idea}

  \

\begin{itemize}

\item Collapsing geometry and classification of ends:

The classification of asymptotic cones and ends is derived from a detailed study based on Cheeger-Fukaya-Gromov's theory \cite{cfg} of collapsing Riemannian manifolds with bounded curvature, combined with the special structure of Hermitian non-K\"ahler Ricci-flat metrics. In the process of taking an asymptotic cone $(M,h_j,p)\to(M_\infty,d_\infty,p_\infty)$ of a Hermitian non-K\"ahler Ricci-flat metric $(M,h)$, it is crucial to capture the information about the K\"ahler metric $g=\lambda^{2/3}h$ at the same time. 
%The way to do so is to rescale the conformal factor $\lambda^{1/3}$ suitably when we are taking asymptotic cones. We will show that the rescaled conformal factors converge smoothly in certain sense to a limit rescaled conformal factor on the asymptotic cone. 
According to Cheeger-Fukaya-Gromov's theory, for sufficiently large $j$ and each small ball $(B_j,h_j) \subset (M,h_j,p)$ away from $p$, $B_j$ has a singular fibration structure fibered by infranilmanifolds, with $B_j$ collapsing to the base. The sequence of universal covers $(\widetilde{B_j}, \widetilde{h_j})$ of $(B_j, h_j)$ is non-collapsed. The conformally K\"ahler property is preserved in each non-collapsed limit of $(\widetilde{B_j}, \widetilde{h_j})$, providing additional structure to each local of the asymptotic cone. Analyzing this special structure leads to the classification of the asymptotic cones. The improvement to the classification of asymptotic models requires a detailed study on the $SU(\infty)$ Toda equation, which appears in the local ansatz for Hermitian non-K\"ahler Ricci-flat metrics. Solutions to this equation play a role analogous to that of harmonic functions in the Gibbons-Hawking ansatz. This is done in Section \ref{sec:collapsegeometry}.

\ 

\item Classification of gravitational instantons, compactification, and positivity:

Given our classification of asymptotic models, the remaining task is to classify Hermitian non-K\"ahler ALF/AF gravitational instantons. The skewed special Kasner models and the $\text{ALH}^*$ models will be excluded due to the negativity of the scalar curvature of the corresponding K\"ahler metrics $g$.
We divide into two cases here:
\begin{itemize}
\item The extremal vector field $\mathcal{K}$ induces an $S^1$-action, which later turns out to induce a holomorphic $\mathbb{C}^*$-action together with $J\mathcal{K}$. In this case we will show that $M$ can be naturally compactified to an algebraic surface $\overline{M}$ by adding a divisor $D=\mathbb{P}^1$ at infinity. 
The observation by LeBrun in \cite{lebrun95} shows that there is a hermitian metric $e^{-2\log s_g}g_{-K}$ on $-K_M$ with positive curvature.
One can show that this hermitian metric extends to a singular hermitian metric on $-(K_{\overline{M}}+D)$. Careful treatments will be done to control the regularity and positivity of the extended hermitian metric. This will imply that $-(K_{\overline{M}}+D)$ is ample.
Further consideration gives that the compactified surfaces all have to be holomorphically toric. This will be discussed in Section \ref{sec:compactification}.
\item The extremal vector field $\mathcal{K}$ induces an $\mathbb{R}$-action only, which later turns out to give rise to a torus action. In this case $(M,h)$ automatically is holomorphically isometrically toric.  Using the analysis we did in Section \ref{sec:collapsegeometry}, we already can apply the method in \cite{biquardgauduchon} and conclude that $(M,h)$ can only be Kerr or Chen-Teo metrics. We will see that $M$ still can be compactified to an algebraic surface $\overline{M}$, but there is no natural way to do so.
\end{itemize}

\ 

\item Calabi-type theorem:

By our previous argument it remains to classify the case where $\mathcal{K}$ induces an $S^1$-action. For compact extremal K\"ahler manifolds, it is a theorem of Calabi \cite{calabi} that the isometry group is a maximal compact subgroup of the automorphism group. We will prove a Calabi-type theorem in this non-compact situation, which will improve holomorphically toric to holomorphically isometrically toric.  This  allows us to apply the method in \cite{biquardgauduchon} and ultimately classify all Hermitian non-K\"ahler ALF/AF gravitational instantons. See Section \ref{sec:matsushima} for details.
\end{itemize}

\begin{comment}
In \cite{biquardgauduchon}, they have a definition for a gravitational instanton being ALF that seems to be more general, which will be denoted by $\text{ALF}_{\text{\textbf{BG}}}$ in our paper. In the appendix we will actually show that the ALF/AF we defined above are the only possible situations under the $\text{ALF}_{\text{\textbf{BG}}}$ definition.  Petrunin-Tuschmann proposed a conjecture in \cite{petrunin} that will fail once there were an $\text{ALF}_{\text{\textbf{BG}}}$ end that is not ALF/AF in our sense: 
\begin{conjecture}[Petrunin-Tuschmann]
For a 4-dimensional simply-connected end with curvature decay $o(\rho^{-2})$, which has a unique asymptotic cone that is a metric cone, it can only have $\mathbb{R}^4$ or $\mathbb{R}^3$ as its asymptotic cone. 
\end{conjecture}
The non-existence of $\text{ALF}_{\text{\textbf{BG}}}$ ends except ALF/AF ends that we will prove in the appendix provides more evidence that supports the conjecture by Petrunin-Tuschmann.
\end{comment}
\end{idea}

\begin{notation}
Throughout this paper, we adopt the following notations:
\begin{itemize}
\item The pair $(M,h)$ refers to Hermitian non-K\"ahler gravitational instantons, with $\lambda\coloneqq2\sqrt{6}|W^+|_h$ as the norm of the self-dual Weyl curvature. The metric $g\coloneqq\lambda^{2/3}h$ is the associated extremal K\"ahler metric. We denote the scalar curvature $\lambda^{1/3}$ of $g$ by $s_g$. 
We put the metric in the lower index to denote their corresponding Levi-Civita connection, or curvature tensor, etc. Use $\rho$ to denote the distance function to a base point under $h$.
\item
We use $\mathcal{K}$ to denote the \textit{extremal vector field} $\mathcal{K}\coloneqq -J\nabla_h\lambda^{-1/3}=J\nabla_g\lambda^{1/3}$. It is real holomorphic and Killing with respect to both the metrics $h$ and $g$. 
By \textit{holomorphic vector field}, we mean $(1,0)$-vector fields that are holomorphic. \textit{Holomorphic extremal vector field} refers to the holomorphic vector field $\mathcal{K}-iJ\mathcal{K}$ that is associated to $\mathcal{K}$.

\item If $T$ is a tensor on $(M,h)$, then we say that $T=O'(\rho^{-\tau})$ if, for any integer $k\geq0$, 
$$|\nabla^k_hT|_h= O(\rho^{-\tau-k}),$$ 
as $\rho$ approaches infinity. 
\item We use $H_k$ to denote the Hirzebruch surface $\mathbb{P}(\mathcal{O}\oplus\mathcal{O}(k))$, where $k\geq0$. We use $C_0$ to denote the curve inside with self-intersection $k$ while $C_\infty$ to denote the curve at infinity with self-intersection $-k$.
\end{itemize}
\end{notation}

\begin{acknowledgement}
I am very grateful to my advisor Song Sun for his constant support and lots of inspiring discussions. I would like to thank Ruobing Zhang for pointing out a lemma in \cite{petrunin}. I am thankful to Olivier Biquard for his interest in this work and a useful conversation, and Lars Andersson for his valuable comments. I am indebted to Claude LeBrun for very helpful discussions.
\end{acknowledgement}

\section{Preliminaries}

\subsection{Gravitational instantons in the Hermitian non-K\"ahler setting}

For a Type II Ricci-flat metric $(M,h)$, the Weyl curvature $W^+:\Lambda^+\to\Lambda^+$ has exactly one repeated eigenvalue and one non-repeated eigenvalue. It was observed by Derdzi\'nski \cite{derdzinski} that if we take $\lambda^{1/3}=(2\sqrt{6}|W^+|_{h})^{1/3}$ as the conformal factor and define $g=\lambda^{2/3}h$, then eigen-two-forms with constant size that correspond to the non-repeated eigenvalue of $W^+$ are parallel under the conformal metric $g$. Up to passing to a double cover, one can pick a globally defined eigen-two-form $\omega$ with size one, which makes the conformal metric $g$ K\"ahler. The complex structure $J$ can be determined using $g,\omega$.
It is a direct calculation to check the scalar curvature of $g$ satisfies $|s_g|=\lambda^{1/3}$.

Now moving to a Type II gravitational instanton. We always pass to a double cover (if necessary), to ensure there is the globally defined two-form $\omega$ and complex structure $J$. From the conformal change of scalar curvature $6\Delta_h(s_g^{-1})+s_g^{-4}=0$, an easy consideration at the maximum/minimum of $s_g$ based on the maximum principle shows that $s_g=\lambda^{1/3}>0$.
Notice that $\lambda^{1/3}$ is positive everywhere as $W^+$ cannot vanish at any point. The following is essentially due to Derdz\'inski \cite{derdzinski}.
\begin{proposition}
Any Type II gravitational instanton $(M,h)$, up to passing to a double cover, is Hermitian non-K\"ahler. Moreover, on the cover the conformal metric $g=\lambda^{2/3}h$ is K\"ahler. The scalar curvature of $g$ is exactly the conformal factor $s_g=\lambda^{1/3}$.
\end{proposition}
Conversely, when a Ricci-flat 4-dimensional metric is Hermitian non-K\"ahler, then it is the result of Goldberg-Sachs that the metric is of Type II.
%Consequently, a gravitational instanton being of Type II is equivalent to being Hermitian non-K\"ahler. 
%\begin{proposition}
%A gravitational instanton $(M,h)$ is Type II if and only if it is Hermitian non-K\"ahler.
%\end{proposition}

For a Hermitian non-K\"ahler gravitational instanton $(M,h)$, the K\"ahler metric $g$ is Bach-flat \cite{besse} as it is conformal to an Einstein metric, which implies that $g$ is extremal K\"ahler \cite{lebrun20}.  Since the scalar curvature of $g$ is $s_g=\lambda^{1/3}$, the extremal vector field is $\mathcal{K}=J\nabla_gs_g=-J\nabla_h\lambda^{-1/3}$. The vector field $\mathcal{K}$ is Killing under $g$, hence it preserves the conformal factor $s_g=\lambda^{1/3}$, therefore it is Killing under the Ricci-flat metric $h$ as well. 
\begin{proposition}
The K\"ahler metric $g$ is extremal K\"ahler. The extremal vector field $\mathcal{K}=J\nabla_gs_g=-J\nabla_h\lambda^{-1/3}$ is Killing with respect to both the metrics $g$ and $h$.
\end{proposition}

\subsection{Gromov-Hausdorff convergence and asymptotic cones}
Denote $\mathcal{MET}_{finite}$ as the set of isometry classes of all metric spaces with finite diameter, and $\mathcal{MET}$ as the set of isometry classes of all pointed complete length spaces $(M,d,p)$ such that every closed ball in $M$ is compact.
\begin{definition}[Gromov-Hausdorff distance]
Let $(X,d_X)$ and $(Y,d_Y)$ be two metric spaces with finite diameter in $\mathcal{MET}_{finite}$. The Gromov-Hausdorff distance $d_{GH}((X,d_X),(Y,d_Y))$ is defined to be the infimum of all $\epsilon>0$ such that there is a metric $d$ on $X\sqcup Y$ extending $d_X$ and $d_Y$ on $X$ and $Y$, such that $X\subset B_\epsilon(Y)$ and $Y\subset B_\epsilon(X)$.
\end{definition}
There is also the pointed version of the Gromov-Hausdorff distance.
\begin{definition}[pointed Gromov-Hausdorff distance]
Let $(X,d_X,p_X)$ and $(Y,d_Y,p_Y)$ be two pointed complete length spaces in $\mathcal{MET}$. The pointed Gromov-Hausdorff distance $d_{GH}((X,d_X,p_X),(Y,d_Y,p_Y))$ is defined to be  $\min\{\epsilon_0,\frac12\}$, where $\epsilon_0\geq0$ is the infinimum of all $\epsilon\in[0,\infty]$ such that there is a metric $d$ on $X\sqcup Y$ extending $d_X$ and $d_Y$ on $X$ and $Y$, such that $d(p_X,p_Y)\leq\epsilon$, $B_{1/\epsilon}(p_X)\subset B_\epsilon(Y)$, and $B_{1/\epsilon}(p_Y)\subset B_\epsilon(X)$. 
\end{definition}

It is a classical result that both $(\mathcal{MET}_{finite},d_{GH})$ and $(\mathcal{MET},d_{GH})$  are complete metric spaces. The \textit{Gromov-Hausdorff convergence} and \textit{pointed Gromov-Haudorff convergence} are defined to be the convergence induced by the Gromov-Hausdorff distance and pointed Gromov-Hausdorff distance respectively. There is also the notion of \textit{equivariant (pointed) Gromov-Hausdorff convergence} when the sequence of metric spaces are equipped with effective isometric group actions. These all can be found in standard reference, for example \cite{rong}. When we merely consider smooth Riemannian manifolds, there is the notion of \textit{Cheeger-Gromov $C^{k,\alpha}$ convergence}.
\begin{definition}[$C^{k,\alpha}$ convergence]
Suppose $(M_j,h_j,p_j)$ is a sequence of pointed $n$-dimensional Riemannian manifolds. We say $(M_j,h_j,p_j)$ converges in $C^{k,\alpha}$-topology to a limit pointed $n$-dimensional Riemannian manifold $(M_\infty,h_\infty,p_\infty)$, if for some sequence $\epsilon_j\to0$, there are $C^{k+1}$ diffeomorphisms $f_j:B_{1/\epsilon_j}(p_j,h_j)\to B_{1/\epsilon_j}(p_\infty,h_\infty)$, such that $f_j(p_j)=p_\infty$, and $(f_j)_*h_j$  converges  to $h_\infty$ as $C^{k,\alpha}$ metrics on any compact subset of $M_\infty$.
\end{definition}

For any sequence of pointed complete $n$-dimensional Riemannian manifolds $(M_j,h_j,p_j)$ with lower Ricci bound $Ric_{h_j}\geq kh_j$ for some $k\in\mathbb{R}$, the Gromov compactness theorem guarantees that by passing to some subsequence, $(M_j,h_j,p_j)$ converges to a complete length space $(M_\infty,d_\infty,p_\infty)$ in the pointed Gromov-Hausdorff topology.
Specializing to the situation of gravitational instantons, given a gravitational instanton $(M,h)$ with base point $p$, for any sequence $r_j\to\infty$, we can consider the rescaled sequence $(M,r_j^{-2}h,p)$. As $i\to\infty$, after passing to subsequences $(M,r_j^{-2}h,p)$ converges to a limit complete metric space $(M_\infty,d_\infty,p_\infty)$ under the pointed Gromov-Hausdorff topology. Any such limit space will be called a \textit{asymptotic cone} of the gravitational instanton $(M,h)$. In the following, we will usually use $(M,h_j,p)$ to denote the rescaled sequence $(M,r_j^{-2}h,p)$.

\subsection{Regular region and smooth region}

\begin{definition}
For a limit space $(M_\infty,d_\infty,p_\infty)$ of a sequence of pointed complete $n$-dimensional Riemannian manifolds $(M_j,h_j,p_j)$ with lower Ricci bound $Ric_{h_j}\geq kh_j$:
\begin{itemize}
\item a point $q_\infty\in M_\infty$ is defined to be \textit{regular} if there is a constant $\delta>0$ and a sequence of $q_j\in M_j$ converging to $q_\infty$, such that $\sup_{B_\delta(q_j)}|Rm_{h_j}|\leq C$ for all $j$;
\item a regular point $q_\infty\in M_\infty$ is defined to be \textit{smooth} if the length space $(M_\infty,d_\infty)$ is a smooth Riemannian manifold in a neighborhood of $q_\infty$.
\end{itemize}
The \textit{regular region} $\mathcal{R}$ is the set of all regular points and the \textit{smooth region} $\mathcal{G}$ is the set of all smooth points.
\end{definition}

Both the regular region and the smooth region are open.
It is a direct application of the $\epsilon$-regularity for 4-dimensional Einstein metrics proved by Cheeger-Tian \cite{cheegertian} that any gravitational instanton $(M,h)$ automatically satisfies the bound
\begin{equation}
|\nabla_h^kRm_h|_h\leq C\rho^{-2-k}.
\end{equation} 
Therefore, in the setting of asymptotic cones of gravitational instantons, the rescaled spaces $(M,h_j,p)$ always have bounded curvature in any definite size ball with definite positive distance to the base point $p$. Henceforth the regular region of an asymptotic cone $(M_\infty,d_\infty,p_\infty)$ of a gravitational instanton is exactly $\mathcal{R}=M_\infty\setminus\{p_\infty\}$. But in general, points in $\mathcal{R}$ not necessarily have to be smooth.
\begin{example}[Gromov's skew rotation]\label{exam:gromovrotation}
This example originates from Section 8.9 in \cite{gromov} and was discussed in previous works \cite{minerbe2010} and \cite{liyu}.
Consider the flat space $\mathfrak{M}_{\mathfrak{a},\mathfrak{b}}$ that  is realized as the quotient of the Euclidean space $\mathbb{R}^3\times\mathbb{R}^1$ by the skew rotation $(\rho,\alpha,\beta,t)\sim(\rho,\alpha+\mathfrak{a},\beta,t+\mathfrak{b})$.
When $\mathfrak{a}/2\pi$ is irrational, $\mathfrak{M}_{\mathfrak{a},\mathfrak{b}}$ has the upper-half-plane $\mathbb{R}^1_+\times\mathbb{R}^1$ as its asymptotic cone. The asymptotic cone consists of regular points entirely, however it does have non-smooth points, namely points in the boundary. 
\end{example}

\subsection{Renormalized limit measure}
\label{subsec:measure}

During the process of taking asymptotic cones of a gravitational instanton $(M,h)$, if the rescaled sequence $(M,h_j,p)\coloneqq(M,r_j^{-2}h,p)$ converges to an asymptotic cone $(M_\infty,d_\infty,p_\infty)$, we can consider the \textit{renormalized measure} defined as 
\begin{equation}\label{eq:volume}
d\nu_j\coloneqq\frac{dvol_{h_j}}{\mathrm{Vol}_{h_j}(B_1(p))}
\end{equation}
on $(M,h_j,p)$. By the work of Cheeger-Colding, passing to a further subsequence if necessary, there exists a Radon measure $d\nu_\infty$ on the limit space $(M_\infty,d_\infty,p_\infty)$, such that we have the \textit{measured Gromov-Hausdorff convergence} $(M,h_j,d\nu_j,p)\xrightarrow{mGH}(M_\infty,d_\infty,d\nu_\infty,p_\infty)$. In particular, for any converging sequence of points $q_j\to q_\infty$ and all $R>0$, we have $d\nu_j(B_R(q_j))\to d\nu_{\infty}(B_R(q_\infty))$.

On the smooth region $\mathcal{G}$, it was shown by Fukaya \cite{fukaya} that the limit measure $d\nu_\infty$ can be expressed using a smooth density function $\chi$
\begin{equation}
d\nu_\infty=\chi dvol_{h_\infty}.
\end{equation}
The density function up to a constant multiple can be determined as follows. For any $q_\infty\in\mathcal{G}$, we find $q_j\in(M,h_j,p)$ such that $q_j\to q_\infty$. Select $\delta>0$, such that the sequence of  universal covers $(\widetilde{B_{\delta}(q_j)},\widetilde{h_j},G_j,\widetilde{q_j})$ is non-collapsed and converges to a limit $(\widetilde{B_\infty},\widetilde{h_\infty},G_\infty,\widetilde{q_\infty})$, 
where $G_j$ is the fundamental group of the cover $\pi_j:\widetilde{B_{\delta}(q_j)}\to{B_{\delta}(q_j)}$ that acts isometrically on  the cover.
The convergence can be made smooth so that the sequence converges in the equivariant $C^\infty$-Cheeger-Gromov sense because of the standard regularity theory for Einstein metrics.
The group $G_\infty$ acts isometrically on $(\widetilde{B_\infty},\widetilde{h_\infty})$, and it is a finite extension of its identity component, which is a nilpotent Lie group, essentially due to Fukaya (or see Theorem 2.1 in \cite{naber}). The ball $B_\delta(q_\infty)$ around $q_\infty$ is isometric to the quotient $\widetilde{B_\infty}/G_\infty$. Fibers of $\pi_\infty$ over points in $B_\delta(q_\infty)$ can be locally identified with an open neighborhood in $G_\infty$. Up to a multiplicative constant, the density function $\chi$ is given by the ratio between the volume form on the fiber induced by the metric $\widetilde{h_\infty}$, and a fixed left-invariant volume form on $G_\infty$.

\subsection{Local ansatzs for Hermitian non-K\"ahler Ricci-flat metrics}
\label{subsec:localansatz}

Just like the Gibbons-Hawking ansatz for hyperk\"ahler metrics with a Hamiltonian Killing field, we also have a local ansatz for complex 2-dimensional Hermitian non-K\"ahler Ricci-flat metrics, essentially due to LeBrun \cite{lebrun}. Since they all automatically have a Killing field, namely the extremal vector field $\mathcal{K}$, which is Hamiltonian under the K\"ahler metric $g$, the ansatz can be applied to any complex 2-dimensional Hermitian non-K\"ahler Ricci-flat metrics locally.
The following theorem is taken from Proposition 3.3 in \cite{biquardgauduchon}. Consider a complex 2-dimensional K\"ahler metric $g$ with a Hamiltonian Killing field $\partial_{t}$. By taking the symplectic reduction with respect to the Hamiltonian Killing field, the metric locally can always be written as
\begin{equation}\label{eq:metricglocallocal}
g=Wd\xi^2+W^{-1}\eta^2+We^v(dx^2+dy^2).
\end{equation}
Here $\eta$ is a connection one-form such that $\eta(\partial_{t})=1$, $\xi$ is the moment map for $\partial_{t}$, and $We^v(dx^2+dy^2)$ is the induced metric on the symplectic reduction.
The complex structure $J$ is given by $d\xi\to -W^{-1}\eta,dx\to-dy$, and the K\"ahler form $\omega$ is given by $d\xi\wedge\eta+We^vdx\wedge dy$.
Integrability condition of the complex structure $J$ is
\begin{align}
d\eta=(We^v)_\xi dx\wedge dy+W_x dy\wedge d\xi+W_y d\xi\wedge dx.\label{eq:deta}
\end{align}
We also have
\begin{align}
(We^v)_{\xi\xi}+W_{xx}+W_{yy}=0\label{eq:compatibility},
\end{align}
which is exactly the equation $d(d\eta)=0$.
The scalar curvature of such a metric can be calculated as
\begin{equation}\label{eq:scalarcurvature}
s_g=-\frac{(e^v)_{\xi\xi}+v_{xx}+v_{yy}}{We^v}.
\end{equation}
The local ansatz for Hermitian non-K\"ahler Ricci-flat metric is:
\begin{theorem}[Local ansatz]\label{thm:localansatz}
For the K\"ahler metric $g$ given by equation (\ref{eq:metricglocallocal}),
the conformal metric $h=\xi^{-2}g$ is Ricci-flat if and only if
\begin{equation}\label{eq:W}
W=k^2\left(\frac{12}{\xi^3}-\frac{6v_{\xi}}{\xi^2}\right)
\end{equation}
and $v$ satisfies the twisted $SU(\infty)$ Toda equation
\begin{equation}\label{eq:twistedtoda}
(e^v)_{\xi\xi}+v_{xx}+v_{yy}=-k^{-2}\xi We^v
\end{equation}
for some positive constant ${k}$.
Any complex 2-dimensional Hermitian non-K\"ahler Ricci-flat metric $h$ is canonically conformally related to such a K\"ahler metric $g=\xi^2h=\lambda^{2/3}h$.
\end{theorem}
The constant can be assumed to be one by scaling $\partial_{t}$ properly. When $k=1$, we have $s_g=\xi$ and the Hamiltonian Killing field $\partial_{t}$ is exactly the vector field $\mathcal{K}$ associated to the Hermitian non-K\"ahler Ricci-flat metric $h=\xi^{-2}g$.
One can perform the transformation
\begin{align}
  \varrho=\frac{1}{\xi}, && V=\xi^2W, && u=v-4\log\xi  \label{eq:ansatz transformation}
\end{align}
and the above ansatz becomes
\begin{theorem}[Local ansatz, second version]\label{thm:local ansatz version two}
  After preforming the transformations \eqref{eq:ansatz transformation}, equations in Theorem \ref{thm:localansatz} with $k=1$ become
  \begin{equation}\label{eq:ansatz two V}
    V=-12\varrho+6\varrho^2u_{\varrho}
  \end{equation}
  and the usual $SU(\infty)$ Toda equation
  \begin{equation}\label{eq:ansatz two Toda}
    (e^u)_{\varrho\varrho}+u_{xx}+u_{yy}=0.
  \end{equation}
  The metric $g$ and $h$ become
  \begin{equation}
    h=\varrho^2g=V(d\varrho^2+e^u(dx^2+dy^2))+V^{-1}\eta^2.
  \end{equation}
\end{theorem}

\subsection{Asymptotic models}
\label{subsec:define asymptotic models}

In this subsection, we introduce the asymptotic models---ALE, ALF-$A$, AF, skewed special Kasner, and $\text{ALH}^*$---that appear in Theorem \ref{thm:main classification ends nonrealizable}.

\begin{definition}\label{def:ale}
  A gravitational instanton $(M,h)$ is ALE (asymptotically locally Euclidean) with rate $\delta_0>0$, if:
  \begin{itemize}
    \item Outside a compact set $K$, the manifold $M\setminus K$ diffeomorphically is $(\mathbb{R}^4/\Gamma)\setminus\overline{B_N(0)}$ with a diffeomorphism $\Phi$, where $\Gamma\subset SO(4)$ is a finite subgroup acting freely on $S^3$.
    \item The metric $h$ satisfies
    $$\left|\nabla_h^k\left(h-\Phi^*g_{\mathbb{R}^4/\Gamma}\right)\right|_h=O(\rho^{-\delta_0-k})$$
    for any $k\geq0$,
    where $\rho$ denotes the distance function to the origin on $\mathbb{R}^4/\Gamma$.
    \end{itemize}
\end{definition}
Note that in \cite{me}, it has already been proved that for a Hermitian non-K\"ahler ALE gravitational instanton, it is asymptotic to $\mathbb{R}^4/\Gamma$, with only very special $\Gamma$. For a Euclidean quotient $\mathbb{R}^4/\Gamma$, the conformal metric $g=\frac{1}{\rho^{4}_{\Gamma}}g_{\mathbb{R}^4/\Gamma}$ is again flat, where $\rho_{\Gamma}$ is the distance function on $\mathbb{R}^4/\Gamma$ from the origin. Only those $\mathbb{R}^4/\Gamma$, where $(\mathbb{R}^4/\Gamma)\setminus {0}$ equipped with the conformal flat metric $\frac{1}{\rho_{\Gamma}^4}g_{\mathbb{R}^4/\Gamma}$ is isometric to a $(\mathbb{C}^2/\Upsilon)\setminus\{0\}$ with standard the flat metric, where $\Upsilon\subset U(2)$, can serve as asymptotic models for Hermitian non-K\"ahler ALE gravitational instantons. 
In other words, the asymptotic model is $(\mathbb{C}^2/\Upsilon)\setminus\{0\}$ with $\Upsilon\subset U(2)$, but equipped with the flat metric $\frac{1}{\rho_{\Upsilon}^4}g_{\mathbb{C}^2/\Upsilon}$. Notably, for a Hermitian non-K\"ahler ALE gravitational instanton the complex structure is asymptotic to the standard complex structure of $\mathbb{C}^2/\Upsilon$.

\begin{definition}\label{def:alf}
A gravitational instanton $(M,h)$ is ALF (asymptotically locally flat) with rate $\delta_0>0$, if: 
\begin{itemize}
\item Outside a compact set $K$, the manifold $M\setminus K$ diffeomorphically is $(N,\infty)\times L$, where $L$ is the total space of an $S^1$-fibration over $S^2$ or $\mathbb{RP}^2$. So $M\setminus K$ diffeomorphically is the total space of an $S^1$-fibration $\pi$ over $\mathbb{R}^3$ or $\mathbb{R}^3/\{\pm id\}$ minus a compact ball $\overline{B_{N}(0)}$.
\item The one-form $\gamma$ is a connection one-form for the $S^1$-fibration $L$.
\item The metric $h$ satisfies
$$\left|\nabla_h^k\left(h-\pi^*g_{\mathbb{R}^3}-\gamma^2\right)\right|_h=O(\rho^{-\delta_0-k})$$
for any $k\geq0$,
where $\rho$ denotes the standard distance function on the base $\mathbb{R}^3$ or  $\mathbb{R}^3/\{\pm id\}$.
\end{itemize}
We denote the generator of the $S^1$-action on $L$ by $\partial_t$.
\end{definition}
We say that an ALF gravitational instanton is \textit{ALF-$A$} if the base of $\pi$ is $\mathbb{R}^3\setminus\overline{B_N(0)}$. If the base of the fibration $\pi$ is $\left(\mathbb{R}^3/\{\pm id\}\right)\setminus \overline{B_N(0)}$, then we say the gravitational instanton is \textit{ALF-$D$}. There are more asymptotic models when the topology of the end is trivial, namely diffeomorphic to $(\mathbb{R}^3\setminus \overline{B_N(0)})\times S^1$. 
\begin{definition}\label{def:af}
A gravitational instanton $(M,h)$ is AF (asymptotically flat) with rate $\delta_0>0$ if: 
\begin{itemize}
\item Outside a compact subset $K\subset M$, there exists a diffeomorphism $\Phi:M\setminus K\to \mathfrak{M}_{\mathfrak{a},\mathfrak{b}}\setminus \overline{B_N(0)}$ from $M\setminus K$ to a model space $\mathfrak{M}_{\mathfrak{a},\mathfrak{b}}$ outside a compact ball $\overline{B_N(0)}$  with $0\leq\mathfrak{a}<2\pi$ and $\mathfrak{b}\neq0$.
\item The model space $(\mathfrak{M}_{\mathfrak{a},\mathfrak{b}},h_{\mathfrak{a},\mathfrak{b}})$ is the quotient of $(\mathbb{R}^3\times\mathbb{R}^1,h_{\mathbb{R}^3\times\mathbb{R}^1})$ by the following relation:
$$(\rho,\alpha,\beta,t)\sim(\rho,\alpha+\mathfrak{a},\beta,t+\mathfrak{b}).$$ 
Here we are using the spherical coordinate $(\rho,\alpha,\beta)$ for $\mathbb{R}^3$ with $0\leq\alpha<2\pi$, $0\leq\beta\leq\pi$, and $\rho$ is the distance function on $\mathbb{R}^3$. 
\item The metric $h$ is asymptotic to the quotient metric $h_{\mathfrak{a},\mathfrak{b}}$ on the model space:
$$|\nabla_{h}^k(h-\Phi^*h_{\mathfrak{a},\mathfrak{b}})|_{h}=O(\rho^{-\delta_0-k})$$
for any $k\geq0$.
\end{itemize}
For the model space the quotient map will be denoted by $\pi:\mathbb{R}^3\times\mathbb{R}^1\to\mathfrak{M}_{\mathfrak{a},\mathfrak{b}}$.
An AF gravitational instanton is said to be $\mathrm{AF}_{\mathfrak{a},\mathfrak{b}}$ if it is asymptotic to the model space $(\mathfrak{M}_{\mathfrak{a},\mathfrak{b}},h_{\mathfrak{a},\mathfrak{b}})$.
\end{definition}

When $\mathfrak{a}=0$, the above definition reduces to a special case of ALF-$A$. The Kerr family \cite{gibbons} and Chen-Teo family \cite{chenteo} of gravitational instantons are examples of AF gravitational instantons, which are not hyperk\"ahler. We call a gravitational instanton \textit{rational} AF, if it is $\text{AF}_{\mathfrak{a},\mathfrak{b}}$ with $\mathfrak{a}/2\pi=m/n$ rational, and \textit{irrational} AF, if it is $\text{AF}_{\mathfrak{a},\mathfrak{b}}$ with $\mathfrak{a}/2\pi$ irrational.
There are indeed more asymptotic models that cannot be filled in as Hermitian non-K\"ahler gravitational instantons. They are asymptotic models in the sense that there are Hermitian non-K\"ahler Ricci-flat metrics asymptotic to them at infinity while being incomplete inside, and they cannot be realized in the sense that there is no Hermitian non-K\"ahler gravitational instanton asymptotic to them. 

\begin{definition}\label{def:Kasner}
  Over $(-\infty,-N)\times\mathbb{R}^3$ where $\varrho\in(-\infty,-N)$ and $(x,y,t)$ parametrizes $\mathbb{R}^3$, define 
  \begin{align}
    h_{\text{Kasner}}&=-6\varrho\left(d\varrho^2-\varrho(dx^2+dy^2)\right)-\frac{1}{6\varrho}dt^2,\label{eq:Kasner0}\\
    h_{\text{ALH}^*}&=-12\varrho\left(d\varrho^2+dx^2+dy^2\right)-\frac{1}{12\varrho}(dt-12xdy)^2.\label{eq:ALH*}
  \end{align}
\end{definition}
For any lattice $\Lambda \subset \mathbb{R}^3$, we further define $h_{\text{Kasner},\Lambda}$ as the quotient metric of $h_{\text{Kasner}}$ on $(-\infty, -N) \times (\mathbb{R}^3 / \Lambda)$. The asymptotic models $h_{\text{Kasner},\Lambda}$ will be referred to as \textit{skewed special Kasner models}. The name Kasner is from general relativity. There is a family of Kasner metrics, with the one mentioned above $h_{\text{Kasner}}$ being Type II. For clarity, we refer to this specific Kasner metric as \textit{special Kasner}. The later metric $h_{\text{ALH}^*}$ is invariant under the Heisenberg group action by identifying the Heisenberg group $\mathcal{H}\simeq\mathbb{R}^3$ as 
\begin{equation}
  \mathcal{H}=\left\{\left(\begin{matrix}1 & x & \frac{t}{12}\\ 0 & 1 & y\\ 0 & 0 & 1\end{matrix}\right)\mid x,y,t\in\mathbb{R}\right\}.
\end{equation}
For any cocompact lattice $\Gamma \subset \mathcal{H}$, consider the quotient metric $h_{\text{ALH}^*-\Gamma}$ on $(-\infty, N) \times (\Gamma \setminus \mathcal{H})$. These asymptotic models will be referred to as \textit{$\text{ALH}^*$ models}.

The skewed special Kasner models and $\text{ALH}^*$ models are Hermitian non-K\"ahler Ricci-flat, as they correspond to the solutions $u = \log(-\varrho)$ and $u = 0$ to \eqref{eq:ansatz two Toda}. Note that the asymptotic models $h_{\text{ALH}^*-\Gamma}$ are precisely the $\text{ALH}^*$ asymptotic models for hyperk\"ahler gravitational instantons. Here, like the reversed Eguchi-Hanson and the reversed Taub-NUT, under the orientation opposite to their hyperk\"ahler orientation, $\text{ALH}^*$ models are Hermitian non-K\"ahler. It is crucial to notice that one cannot take further finite quotient as it breaks the $S^1$-action generated by $\partial_t$, where $\mathcal{K}=-\partial_t$ is the extremal vector field. That is, the fiber over $(-\infty,-N)$ is a nilpotent manifold and cannot be an infranilmanifold. Conformal extremal K\"ahler metrics are
\begin{align}
  g_{\text{Kasner},\Lambda}&=\frac{1}{\varrho^2}h_{\text{Kasner},\Lambda}=\frac{1}{\varrho^2}\left(-6\varrho d\varrho^2-\frac{1}{6\varrho}dt^2\right)+6(dx^2+dy^2),\\
  g_{\text{ALH}^*-\Gamma}&=\frac{1}{\varrho^2}h_{\text{Kasner},\Lambda}=\frac{1}{\varrho^2}\left(-12\varrho d\varrho^2-\frac{1}{12\varrho}(dt-12xdy)^2\right)-\frac{12}{\varrho}(dx^2+dy^2).
\end{align}
Scalar curvatures of these K\"ahler metrics are both $s_g=\varrho<0$, which explains why there should not be Hermitian non-K\"ahler gravitational instantons asymptotic to them.

\section{Collapsing geometry and classification of ends}
\label{sec:collapsegeometry}

In this section we classify all possible asymptotic cones as well as ends. Consider a $d$-dimensional asymptotic cone $(M_\infty,d_\infty,p_\infty)$ of $(M,h,p)$, which is realized as the limit of some sequence $(M,h_j,p)$ with $h_j=r_j^{-2}h$ and $r_j\to\infty$. In the Hermitian non-K\"ahler case, the smooth region $\mathcal{G}$ might be smaller than the regular region $\mathcal{R}$. In the following we study the geometric structures on the regular region.

Suppose $q_\infty$ is a point in the regular region $\mathcal{R}$. Then we can find a sequence of points $q_j\in (M,h_j,p)$ converging to $q_\infty$. There exists a small $\delta>0$, such that $B_\infty\coloneqq B_\delta(q_\infty)\subset\mathcal{R}$, and the sequence of universal covers $\widetilde{B_{j}}$ of $B_j\coloneqq B_\delta(q_j)$ is non-collapsed, which is a consequence of Cheeger-Fukaya-Gromov's theory. Denote the deck transformation group of the universal covering map $\pi_j:\widetilde{B_j}\to B_j$ by $G_j$. Then we have the following equivariant convergence picture:

$$
\begin{tikzcd}
{(\widetilde{B_j},\widetilde{h_j},G_j,\widetilde{q_j})} \arrow[d, "\pi_j"'] \arrow[rrrr, "\text{$equivariant$ $GH$-$convergence$}"] &  &  &  & {(\widetilde{B_\infty},\widetilde{h_\infty},G_\infty,\widetilde{q_\infty})} \arrow[d, "\pi_\infty"] \\
{(B_{j},h_j)} \arrow[rrrr, "GH-convergence"']                                                                                         &  &  &  & {(B_\infty,h_\infty).}                                                                              
\end{tikzcd}
$$
Here, $G_\infty$ is a subgroup of the isometry group $\mathrm{Iso}(\widetilde{B_\infty},\widetilde{h_\infty})$ of $\widetilde{B_\infty}$ and $B_\infty=\widetilde{B_\infty}/G_\infty$.
Moreover, $G_\infty$ is a finite extension of its identity component $G_\infty^0$, which is a nilpotent Lie group, and it acts properly and faithfully on $\widetilde{B_\infty}$ (Theorem 2.1 in \cite{naber}). 
Since $\widetilde{h_j}$ is Ricci-flat, the standard regularity theory about non-collapsed limit of Einstein metrics shows that the pointed convergence $(\widetilde{B_j},\widetilde{h_j},\widetilde{q_j})\to (\widetilde{B_\infty},\widetilde{h_\infty},\widetilde{q_\infty})$ actually holds in  $C^\infty$ topology.  This picture will be implicitly used frequently. We shall study the geometry of each non-collapsed $(\widetilde{B_\infty},\widetilde{h_\infty},G_\infty)$, which for simplicity we shall refer to as the \textit{local geometry of the asymptotic cone}.

\subsection{Limit rescaled conformal factor}\label{subsec:limit rescaled conformal factor}

To capture the information about the K\"ahler metric $g$ in the process of taking asymptotic cones of $(M,h,p)$, we need to rescale the conformal factor $\lambda^{1/3}$ and take its limit at the same time. Since the scalar curvature of $g$ is $s_g=\lambda^{1/3}$ and the scalar curvature of $h$ is $s_h=0$, by computing the scalar curvature $s_g$ using the conformal change, we have the equation
\begin{equation}\label{eq:conformal change of scalar}
6\Delta_h\lambda^{1/3}+\lambda^{4/3}=0.
\end{equation}
As above we take $q_\infty\in\mathcal{R}\subset M_\infty$ and pick a sequence of points $q_j\in (M,h_j)$ with $q_j\to q_\infty$. Renormalize $\lambda^{1/3}$ as $\lambda^{1/3}_j\coloneqq \lambda^{-1/3}(q_j)\lambda^{1/3}$. It will turn out later that the choice of $q_\infty\in\mathcal{R}$ does not essentially affect the rescaled limit. 
After the renormalization, the equation that $\lambda^{1/3}_j$ satisfies is
\begin{equation}\label{eq:normalizedequation}
6\Delta_{h_j}\lambda^{1/3}_j+r_j^{2}\lambda(q_j)\lambda^{4/3}_j=0.
\end{equation}
By applying the Cheng-Yau gradient estimate to $\lambda_j^{1/3}$ on the rescaled Ricci-flat Riemannian manifold $(M,h_j)$, we have the following proposition.
\begin{proposition}\label{prop:chengyau}
On any ball $B_{a}(x_j,h_j)\subset (M,h_j)$ with radius $a$, we have the inequality
\begin{equation}\label{eq:gradient}
\frac{|\nabla_{h_j}\lambda_j^{1/3}|_{h_j}}{\lambda_j^{1/3}}\leq C\frac{a^2}{a^2-d_j(\cdot,x_j)^2}\max\left\{r_j^2\sup_{B_a(x_j,h_j)}\lambda,\ \frac{1}{a}\right\}.
\end{equation}
Here $d_j(\cdot,x_j)$ refers to the distance function to $x_j$ under the rescaled metric $h_j=r_j^{-2}h$ and $C$ is a universal constant.
\end{proposition}
\begin{proof}
We have the following inequalities on $B_a(x_j,h_j)$
\begin{align*}
|\nabla_{h_j}(\Delta_{h_j}\lambda_j^{1/3})|_{h_j}&=\frac23r_j^{2}\lambda(q_j)\lambda_j|\nabla_{h_j}\lambda_j^{1/3}|_{h_j}\leq \frac23r_j^{2}\lambda(q_j)\left(\sup_{B_a(x_j,h_j)}\lambda_j\right) |\nabla_{h_j}\lambda_j^{1/3}|_{h_j}\\
&=\frac23r_j^{2}\left(\sup_{B_a(x_j,h_j)}\lambda\right) |\nabla_{h_j}\lambda_j^{1/3}|_{h_j},\\
\Delta_{h_j}\lambda_j^{1/3}&=-\frac16r_j^{2}\lambda(q_j)\lambda^{4/3}_j<0.
\end{align*}
Notice that the rescaled metric $h_j=r_j^{-2}h$ is again Ricci-flat.
It is a direct application of Theorem 6 in \cite{chengyau} to the strictly positive function $\lambda^{1/3}_j$ to get (\ref{eq:gradient}).
\end{proof}

Hence, if we take $B_a(x_j,h_j)$ as the ball $B_a(q_j,h_j)$ with fixed definite radius $a$ so that the singularity point $p_\infty\notin B_a(q_\infty,h_\infty)$, we will have $r_j^{2}\sup_{B_a(q_j,h_j)}\lambda\leq C$ by the curvature decay for $j$ large. Therefore, by taking $j$ large enough, inequality (\ref{eq:gradient}) becomes
\begin{equation}\label{eq:gradientilarge}
{|\nabla_{h_j}\lambda_j^{1/3}|_{h_j}}\leq C{\lambda_j^{1/3}}
\end{equation}
with $C$ depending on $a$.
This gives us the following Harnack inequality on a single ball.
\begin{proposition}\label{prop:localharnack}
For $j>j_0$ with $j_0$ depending on the choice of $q_\infty$ and $a$, on $B_{a/2}(q_j,h_j)$, there is a constant $C$ depending on $a$ such that
\begin{equation}
\sup_{B_{a/2}(q_j,h_j)}\lambda^{1/3}_j\leq C\inf_{B_{a/2}(q_j,h_j)}\lambda^{1/3}_j.
\end{equation}
\end{proposition}

With the Harnack inequality, we can conclude that the renormalized conformal factor $\lambda^{1/3}_j$ is bounded from both above and below on $B_{a/2}(q_j,h_j)$, since $\lambda^{1/3}_j(q_j)=1$. Harnack inequality over other balls that are not centered at $q_j$ with definite radius and definite positive distance to the base point $p$ can be proved similarly, and we can therefore conclude the following global Harnack inequality simply  by  patching balls together.
\begin{proposition}\label{prop:harnack}
For any compact set $K\subset\subset\mathcal{R}\subset M_\infty$ and any sequence of $K_j\subset (M,h_j)$ that converges to ${K}$, the following Harnack inequality for $\lambda^{1/3}_j$ on $K_j$ uniformly holds when $j$ is large enough
\begin{equation}\label{eq:harnacklarge}
\sup_{K_j}\lambda^{1/3}_j\leq C_K\inf_{K_j}\lambda^{1/3}_j
\end{equation}
\end{proposition}

Therefore, for any such sequence of $K_j$,  we have uniform upper and lower bounds on $\lambda^{1/3}_j$ that do not depend on $j$. Now the standard elliptic theory can be applied to the equation (\ref{eq:normalizedequation}) to derive bounds on derivatives of $\lambda^{1/3}_j$. Passing to certain subsequence of $(M_j,h_j,p)$, we can make sense of $\lambda^{1/3}_j$ smoothly converging to a limit function on the regular region $\mathcal{R}$ as follows.

\begin{definition}
For a rescaled sequence $(M,h_j,p)$ that converges to an asymptotic cone $(M_\infty,d_\infty,p_\infty)$, we say a sequence of smooth functions $f_j$ on $(M,h_j,p)$ converges smoothly to a limit function $f_\infty$ on the regular region $\mathcal{R}$, if for any $B_\delta(q_\infty)\subset\subset\mathcal{R}$ and any sequence of balls $B_\delta(q_j)\subset(M,h_j)$ converging to $B_\delta(q_\infty)$ such that the sequence of universal covers  $\widetilde{B_\delta(q_j)}$ is non-collapsed and converges to $\widetilde{B_\infty}$, the pullback sequence of functions $\pi_j^*f_j$ converges smoothly to $\pi_\infty^*f_\infty$.
\end{definition}

\begin{proposition}
By passing to a suitable subsequence, $\lambda^{1/3}_j$ converges smoothly to a limit function $\lambda^{1/3}_\infty$ on the regular region $\mathcal{R}$ of the asymptotic cone $M_\infty$. The limit rescaled conformal factor $\lambda^{1/3}_\infty$ is nowhere vanishing.
\end{proposition}
\begin{proof}
For any ball $B_\delta(q_j)\subset(M,h_j,p)$, with definite size and definite positive distance to the base point $p$ under the rescaled metric $h_j$, the standard Schauder theory can be applied to the equation (\ref{eq:normalizedequation}) on the cover $\widetilde{B_\delta(q_j)}$. It is now a simple application of the diagonalization procedure to ensure convergence. The non-vanishing property follows from the Harnack inequality.
\end{proof}

We call the limit $\lambda^{1/3}_\infty$ as a  \textit{limit rescaled conformal factor}. The choice of the point $q_\infty\in\mathcal{R}$ to renormalize $\lambda^{1/3}$ only affects $\lambda^{1/3}_\infty$ up to multiplication by a constant, but clearly $\lambda^{1/3}_\infty$ depends on the specific asymptotic cone we take. Even if we fix an asymptotic cone, we may still need to pass to a subsequence to take a limit $\lambda^{1/3}_\infty$. Note that it is possible that $\lambda^{1/3}_\infty$ is trivial in the sense that $\lambda^{1/3}_\infty\equiv1$, however we will later prove it cannot happen. As a consequence of the smooth convergence of the rescaled conformal factor, we have K\"ahler metric $g_j:=\lambda_j^{2/3}h_j$ pulled back to each $\widetilde{B_j}$  converge smoothly to a limit K\"ahler metric $\widetilde{g_\infty}=\lambda_\infty^{2/3}\widetilde{h_\infty}$ on $\widetilde{B_\infty}$.

\subsection{Analysis of asymptotic cones}
\label{subsec:analysis tangent cone}

We investigate collapsed asymptotic cones of Hermitian non-K\"ahler gravitational instantons. All Hermitian non-K\"ahler gravitational instantons in this subsection are assumed to be collapsed at infinity.
\begin{proposition}
  The extremal vector field $\mathcal{K}$ has no zeros outside a large compact set $K$. Therefore the moment map $s_g=\lambda^{1/3}$ of $\mathcal{K}$ has no critical points outside $K$.
\end{proposition}
\begin{proof}
  First notice that there cannot be a connected component of the fixed point set of $\mathcal{K}$ that connects to infinity. Otherwise as $\lambda^{1/3}\to0$, we would have $\lambda^{1/3}\equiv0$ along this connected component. So each component of fixed point set must be compact. If a fixed point is non-isolated, as $\mathcal{K}$ is real holomorphic under $g$, the fixed point must be contained in a fixed holomorphic curve. Over each fixed curve $s_g$ is a constant. Next we show there cannot be a sequence of fixed holomorphic curves that go to infinity. If such a sequence existed, the generic flows of $-\nabla_g s_g$ would connect one fixed curve to another. This leads to a contradiction, as by \eqref{eq:conformal change of scalar}, each fixed curve should be a local maximum of $s_g$. However, along the flow of $-\nabla_gs_g$, the scalar curvature $s_g$ decreases.

  Now we prove there cannot be isolated fixed points going to infinity. Prove by contradiction. If $\mathcal{K}$ had isolated zeros $q_j$ with $q_j\to\infty$, then set $r_j=d_h(q_j,p)$ and suppose $(M,h_j,p)$ converges to an asymptotic cone $(M_\infty,d_\infty,p_\infty)$. Further assume $\delta$ is small enough such that the universal cover $\widetilde{B_j}$ of $B_j=B_{\delta}(q_j)\subset (M,h_j)$ is non-collapsed. Then the extremal vector field pulled back to $\widetilde{B_j}$, which we still denote by $\mathcal{K}$, converges to a limit Killing field, since it is a non-trivial Killing field with isolated zeros. The fixed point $\widetilde{q_\infty}$ is still an isolated fixed point of the limit Killing field. The identity component of the collapsing group $G_\infty^0$ must preserve $\widetilde{q_\infty}$, and this leads to a contradiction, as $\widetilde{q_\infty}$ would have the entire $G_\infty^0$ as its isotropy. This is not possible since there cannot be a point with the entire $G_\infty^0$ as isotropy (see for example Theorem 2.1 in \cite{naber}).
\end{proof}

Since $\lambda^{1/3}$ has no critical points on $M \setminus K$, the symplectic reduction can be applied over the entire end $M \setminus K$. A key feature of Hermitian non-K\"ahler gravitational instantons is that the ansatz from Theorems \ref{thm:localansatz} and \ref{thm:local ansatz version two} can be applied globally to the end. We now consider the following two cases.
\begin{enumerate}
  \item[($\clubsuit$)] $\mathcal{K}$ induces an $S^1$-action.
  \item[($\spadesuit$)] $\mathcal{K}$ only induces an $\mathbb{R}$-action. 
\end{enumerate}
We deal with ($\clubsuit$) and ($\spadesuit$) in Section \ref{subsubsec:induce S^1-action} and \ref{subsubsec:induce R-action} respectively.

\subsubsection{$\mathcal{K}$ induces an $S^1$-action}
\label{subsubsec:induce S^1-action}

Suppose $\mathcal{K}$ induces an $S^1$-action. The moment map over the end for $\mathcal{K}$ is $\xi=\lambda^{1/3}$ and each level hypersurface $L_{\xi_0}=\{\xi=\xi_0\}$ is compact as $\xi$ is decaying to zero. Let $\Sigma_{\xi_0}$ be the surface (orbifold if the $S^1$-action is only semi-free) obtained by employing the symplectic reduction at level $\xi_0$. Clearly all $\Sigma_{\xi_0}$ are biholomorphic (after forgetting the orbifold structure) by the flow of $\nabla_gs_g$, so we can drop the lower index and denote the surface by $\Sigma$. Diffeomorphically the end is $M\setminus K=(N_0,\infty)\times L$, where $L$ is the link of the end that is an $S^1$-bundle over $\Sigma$ (orbifold bundle when $\Sigma$ is an orbifold). Fixing a holomorphic coordinate $x+iy$ on $U\subset\Sigma$, the metric can be written as
\begin{equation}\label{eq:ansatz applied to write metric}
  h=\varrho^2g=V(d\varrho^2+e^u(dx^2+dy^2))+V^{-1}\eta^2.
\end{equation}
Here we are applying the ansatz Theorem \ref{thm:local ansatz version two}, so $\varrho=\xi^{-1}=\lambda^{-1/3}>0$. Note that $V^{-1}=|\mathcal{K}|^2_h$ is a globally well-defined function, but $u$ depends on the particular choice of $x+iy$. We have the following important observation.
\begin{lemma}\label{lem:geometric Toda equation}
  Let $K_{\Sigma}$ be the curvature of the metric $e^u(dx^2+dy^2)$ on $\Sigma$. Then the $SU(\infty)$ Toda equation $(e^u)_{\varrho\varrho}+u_{xx}+u_{yy}=0$ is equivalent to 
  $$\frac{1}{2}(e^u)_{\varrho\varrho}e^{-u}=K_\Sigma.$$
\end{lemma}
\begin{proof}
  Direct calculation.
\end{proof}

By Gauss-Bonnet, for the metric $e^u(dx^2+dy^2)$ on $\Sigma$ we have $\int_{\Sigma}K_{\Sigma}dvol_{\Sigma}=\frac{1}{2}\int_{\Sigma}(e^u)_{\varrho\varrho}dxdy=2\pi\chi(\Sigma)$ and
\begin{equation}\label{eq:integral e^u}
  \int_{\Sigma}e^udxdy=2\pi\chi(\Sigma)\varrho^2+a\varrho+b
\end{equation}
for some constants $a,b$. It is also direct to compute that 
\begin{align}
  \int_{\Sigma} Ve^u dxdy&=\int_{\Sigma} (-12\varrho+6\varrho^2 u_{\varrho})e^u dxdy=\int_{\Sigma} \varrho^4(\frac{6e^u}{\varrho^2})_{\varrho} dxdy\notag\\
  &=-6a\varrho^2-12b\varrho.\label{eq:integral Ve^u}
\end{align}
These imply $\chi(\Sigma)>0$ and $\Sigma$ holomorphically is the sphere $S^2$ since $\varrho>0$. Finally we can calculate the volume
\begin{lemma}\label{lem:volume lemma S^1}
  $\mathrm{Vol}(\{D_0\leq \varrho\leq D_1\})=-4\pi a(D_1^3-D_0^3)-12\pi b(D_1^2-D_0^2)$.
\end{lemma}
\begin{proof}
  Note that $\mathrm{Vol}(\{D_0\leq \varrho\leq D_1\})=\int_{D_0}^{D_1}\int_{0}^{2\pi}(\int_{\Sigma}Ve^u dxdy)\eta d\varrho$.
\end{proof}

Now consider an asymptotic cone $(M,h_j,p)\to(M_\infty,d_\infty,p_\infty)$ where the rescaled conformal factor converges $\lambda_j^{1/3}\to\lambda_\infty^{1/3}$, and denote $\varrho_j:=\lambda_j^{-1/3}$ which converges to its limit $\varrho_\infty:=\lambda_\infty^{-1/3}$. 
\begin{lemma}\label{lem:rescaled conformal cannot be trivial}
  The limit rescaled conformal factor $\lambda^{1/3}_\infty$ cannot be identically one.
\end{lemma}
\begin{proof}
  Consider annuli $A_j=B_2(p)\setminus \overline{B_1(p)}\subset (M,h_j,p)$ that converges to the corresponding annulus in the asymptotic cone $A_\infty$. Suppose $\lambda^{1/3}$ is renormalized at $q_j\in A_j$. If $\varrho_\infty=\lambda^{-1/3}_\infty\equiv1$, then $\varrho(q_j)=\lambda^{-1/3}(q_j)$ goes to infinity while $\sup_{A_j}\varrho-\inf_{A_j}\varrho\ll \lambda^{-1/3}(q_j)$. By Lemma \ref{lem:volume lemma S^1}, one concludes the renormalized volume (see Section \ref{subsec:measure})
  $$\frac{\mathrm{Vol}_{h_j}(A_j)}{\mathrm{Vol}_{h_j}(B_1)}\to0,$$
  which is a contradiction.
\end{proof}

Hence the limit rescaled conformal factor is always non-trivial. It follows that in each non-collapsed limit $\widetilde{B_\infty}$ the vector field $\nabla_{\widetilde{h_\infty}}\lambda^{-1/3}_{\infty}$ is non-trivial as well. Let $\mathcal{K}_j:=-J\nabla_{h_j}\lambda_{j}^{-1/3}$ which after pulling back to $\widetilde{B_j}$ converges to $\mathcal{K}_\infty:=-J\nabla_{\widetilde{h_\infty}}\lambda^{-1/3}_\infty$ on $\widetilde{B_\infty}$ and 
\begin{align}
V_j:=|\nabla_{h_j}\lambda^{-1/3}_j|_{h_j}^{-2}=\lambda^{-2/3}(q_j)r_j^{-2}V\to V_\infty:=|\nabla_{\widetilde{h_\infty}}\lambda^{-1/3}_\infty|^{-2}_{\widetilde{h_\infty}}.
\end{align}
Back to the local ansatz applied to the end. After passing to $h_j$ the metric becomes
\begin{align}\label{eq:rescaled metric with ansatz}
h_j=r_j^{-2}V(d\varrho^2+e^u(dx^2+dy^2))+r_j^{-2}V^{-1}\eta^2=V_j(d\varrho_j^2+\lambda^{2/3}(q_j)e^u(dx^2+dy^2))+r_j^{-4}\lambda^{-2/3}(q_j)V_j^{-1}\eta^2.
\end{align}

\begin{lemma}\label{lem:tangent cone 3-dim S^1}
  The function $u_j:=u+\log\lambda^{2/3}(q_j)$ converges smoothly to a limit $u_\infty$. Any collapsed asymptotic cone must be 3-dimensional where the collapsing is along the $S^1$-action by $\mathcal{K}$. 
\end{lemma}
\begin{proof}
  As $\varrho_j \to \varrho_\infty$ smoothly, on each $\widetilde{B_j}$, the metric obtained by performing symplectic reduction on $\widetilde{g_j}$ with the Hamiltonian Killing field $\mathcal{K}_j$ at different levels of $\lambda_j^{-1/3}$ converges to the metric obtained by performing symplectic reduction on $\widetilde{g_\infty}$ with $\mathcal{K}_\infty$ at the corresponding level of $\lambda_\infty^{-1/3}$. This particularly shows there are constants $C_j$ such that $u+C_j$ converges. To see $u+\log\lambda^{2/3}(q_j)$ converges, first notice that $u+\log\lambda^{2/3}(q_j)\not\to\infty$, since otherwise the diameter of symplectic reductions would be unbounded and contradict with the convergence to the asymptotic cone. It suffices to rule out the case where $u + \log \lambda^{2/3}(q_j) \to -\infty$. If this were true, the symplectic reduction $\Sigma$ would collapse in $\widetilde{B_\infty}$, which is impossible, as $\chi(\Sigma) > 0$ would contradict the smooth convergence $\widetilde{B_j} \to \widetilde{B_\infty}$. According to Cheeger-Fukaya-Gromov's theory, collapsed fibers must be infranilmanifolds. 
  The smooth convergence of $u+\log\lambda^{2/3}(q_j)$ shows the only collapsed direction is $\mathcal{K}$.
\end{proof}

Now we use the equation $V=-12\varrho+6\varrho^2u_\varrho$. The rescaled $V_j$ satisfies
$$V_j=\lambda^{-1}(q_j)r_j^{-2}(-12\varrho_j+6\varrho_j^2u_{\varrho_j}).$$
Although $\varrho_\infty$ may have critical points, we can always work away from these, and our subsequent argument will ultimately show that $\varrho_\infty$ has no critical points. There are two situations to consider, owing to the smooth convergence $V_j \to V_\infty$, with the important observation that $V_\infty \not\equiv 0$.
\begin{enumerate}
  \item[(1)] $\lambda^{-1}(q_j)r_j^{-2}\to constant$ and $-12\varrho_\infty+6\varrho_\infty^2 \partial_{\varrho_\infty}u_{\infty}$ is a non-trivial function.
  \item[(2)] $\lambda^{-1}(q_j)r_j^{-2}\to \infty$ and $-12\varrho_\infty+6\varrho_\infty^2 \partial_{\varrho_\infty}u_{\infty}\equiv0$.
\end{enumerate}
In (1) the local geometry of the asymptotic cone is Type II in the sense that the Ricci-flat metric $\widetilde{h_\infty}$ on each $\widetilde{B_\infty}$ is Type II. As for (2) the local geometry is Type I. Luckily we have

\begin{lemma}\label{lem:equation for u_jnfty}
  The case (1) above cannot happen. Hence $\lambda^{-1}(q_j)r_j^{-2}\to \infty$ and $-12\varrho_\infty+6\varrho_\infty^2 \partial_{\varrho_\infty}u_{\infty}\equiv0$.
\end{lemma}
\begin{proof}
  If we were in (1), then there would be a simple contradiction in the volume of the annuli $A_j$. Lemma \ref{lem:tangent cone 3-dim S^1} gives the volume $\mathrm{Vol}_{h_j}(A_j)\sim r_j^{-2}\lambda^{-1/3}(q_j)\sim r_j^{-4/3}$. However, Lemma \ref{lem:volume lemma S^1} implies $\mathrm{Vol}_{h_j}(A_j)\sim r_j^{-4}(-4\pi ar_j^{2}-12\pi br_j^{4/3})$, which is a contradiction. 
\end{proof}

From $-12\varrho_\infty+6\varrho_\infty^2 \partial_{\varrho_\infty}u_{\infty}\equiv0$ we conclude $u_\infty=2\log\varrho_\infty+\varpi(x,y)$. One can pass the $SU(\infty)$ Toda equation to the limit as well which gives 
$$\partial^2_{\varrho_\infty}(e^{u_\infty})+\partial^2_{x}u_\infty+\partial^2_{y}u_\infty=0.$$
Substituting $u_\infty=2\log\varrho_\infty+\varpi(x,y)$ to the above we see the curvature of $e^{\varpi(x,y)}(dx^2+dy^2)$ is one. Therefore the metric $e^{u_j-2\log\varrho_j}(dx^2+dy^2)$ on $\Sigma$ is converging to the round metric $g_{S^2}$ on $S^2$, or an orbifold metric on $S^2$ with curvature one. Consequently, any collapsed asymptotic cone is 3-dimensional, equipped with the metric 
\begin{equation}\label{eq:metric on tangent cone}
  h_\infty=V_\infty(d\varrho_\infty^2+\varrho_\infty^2 g_{S^2/\mathbb{Z}_k}).
\end{equation}
Notice that here $\mathbb{Z}_k$ acts on $S^2$ as rotations. So from now on we replace the distance function $d_\infty$ in the tuple $(M_\infty,d_\infty,p_\infty)$ by the metric $h_\infty$.

\begin{lemma}\label{lem:structure of tangent cone for S^1}
  When $\mathcal{K}$ induces an $S^1$-action, the local geometry of the asymptotic cone $(M_\infty,h_\infty,p_\infty)$ is flat, in the sense that each non-collapsed limit $(\widetilde{B_\infty},\widetilde{h_\infty})$ is flat. Any such Hermitian non-K\"ahler gravitational instanton has unique asymptotic cone $\mathbb{R}^3/\mathbb{Z}_k$, with $\mathbb{Z}_k$ being the isotropy group of the $S^1$-action.
\end{lemma}
\begin{proof}
  By \eqref{eq:rescaled metric with ansatz} and Lemma \ref{lem:tangent cone 3-dim S^1}, the metric $\widetilde{h_\infty}=V_\infty(d\varrho_\infty^2+\varrho^2_\infty g_{S^2})+V_\infty^{-1}\eta_\infty^2$, where the group $G_\infty=\mathbb{R}$.  The key ingredient is that, since the local geometry of the asymptotic cone is Type I, we have $W^+_{\widetilde{h_\infty}}\equiv0$.  Consequently, the conformal metric $\widetilde{g_\infty} = \frac{1}{\varrho_\infty^2} \widetilde{h_\infty}$ becomes K\"ahler scalar-flat, with the complex structure defined by $J_\infty: d\varrho_\infty \to V_\infty^{-1} \eta_\infty$ and the standard complex structure on $S^2/\mathbb{Z}_k$. The ansatz from Theorem \ref{thm:localansatz} simplifies when $\widetilde{g_\infty}$ is K\"ahler scalar-flat. One can write 
  $$\widetilde{g_\infty}={W_\infty}d\xi_\infty^2+W_\infty^{-1}\eta^2+{W_\infty}e^{{v_\infty}}(dx^2+dy^2)$$
  where $\partial^2_{\xi_\infty}(e^{{v}_\infty})+\partial^2_{x}{v}_{\infty}+\partial^2_{y}{v}_{\infty}=0$ and $\partial^2_{\xi_\infty}({W}_\infty e^{{v}_\infty})+\partial^2_{x}{W}_{\infty}+\partial^2_{y}{W}_{\infty}=0$ by the LeBrun ansatz for K\"ahler scalar-flat metrics with $\mathbb{R}$-symmetry, and $W_\infty,\xi_\infty,v_\infty$ are related to $V_\infty,\varrho_\infty,u_\infty$ as in \eqref{eq:ansatz transformation}. Computing the change of scalar curvature from $\widetilde{g_\infty}$ to $\widetilde{h_\infty}=\xi_\infty^{-2}\widetilde{g_\infty}$ we conclude
  $$\partial_{\xi_\infty}{v_\infty}=\frac{2}{\xi_\infty}.$$
  This forces ${v}_\infty=2\log\xi_\infty+\varpi(x,y)$, which coincides with $u_\infty=2\log\varrho_\infty+\varpi(x,y)$. 

  Now we turn to the equation 
  \begin{equation}\label{eq:linearization of Toda}
    \partial_{\xi_\infty}^2(W_\infty e^{v_\infty})+\partial_x^2W_\infty+\partial_y^2W_\infty=0.
  \end{equation}
  With $v_\infty=2\log\xi_\infty+\varpi(x,y)$ where $e^{\varpi}(dx^2+dy^2)$ is the round metric $g_{S^2}$, \eqref{eq:linearization of Toda} can be rewritten as 
  \begin{equation}\label{eq:linearization of Toda separating}
    \partial_{\xi_\infty}^2(W_\infty\xi_\infty^2)+\Delta_{S^2}W_\infty=0.
  \end{equation}
  Write $W_\infty=\sum w_\lambda f_\lambda$, where $f_\lambda$ is an eigenfunction with eigenvalue $\lambda$ for $\Delta_{S^2}$, $\{f_\lambda\}$ forms an orthogonal basis for $L^2(S^2)$, and $w_\lambda$ are functions of $\xi_\infty$. By separating the variables, \eqref{eq:linearization of Toda separating} reduces to ODEs
  $$\partial_{\xi_\infty}^2(w_\lambda \xi_\infty^{2})+\lambda w_\lambda=0,$$
  where solutions are $w_\lambda=c_{\lambda_1}\xi_{\infty}^{\lambda_1}+c_{\lambda_2}\xi_{\infty}^{\lambda_2}$. Here $c_{\lambda_1},c_{\lambda_2}$ are arbitrary constants, and for each eigenvalue $\lambda$, the numbers $\lambda_1,\lambda_2$ are the two roots for the following equation of $\mu$ 
  $$(\mu+2)(\mu+1)+\lambda=0.$$
  Note that the spherical eigenvalues are $\lambda=-k(k+1)$ with $k\in\mathbb{Z}_{\geq0}$. Translating to the metric $h_\infty$, we have 
  \begin{equation}\label{eq:V_infty separating}
    V_\infty=\sum (c_{\lambda_1}\varrho_\infty^{-\lambda_1-2}+c_{\lambda_2}\varrho_\infty^{-\lambda_2-2})f_\lambda
  \end{equation}
  and
  $$h_\infty=V_\infty(d\varrho^2_\infty+\varrho_\infty^2g_{S^2/\mathbb{Z}_k})+V_\infty^{-1}\eta^2_\infty.$$

  Finally we claim $V_\infty=c$ for a constant $c>0$. That is, in \eqref{eq:V_infty separating}, only the $\lambda=0$ term appears. Moreover, when $\lambda=0$ where we have $\lambda_1=-1,\lambda_2=-2$, the constants $c_{\lambda_1}=0,c_{\lambda_2}=c$. To see this, consider the ratio between the integral formulas \eqref{eq:integral e^u} and \eqref{eq:integral Ve^u} 
  \begin{equation}\label{eq:ratio between integrals}
    \frac{\int_\Sigma Ve^udxdy}{\int_\Sigma e^udxdy}=\frac{-6a\varrho^2-12b\varrho}{2\pi\chi(\Sigma)\varrho^2+a\varrho+b}.
  \end{equation}
  The growth order of \eqref{eq:ratio between integrals} clearly implies only the constant term and the $\varrho^{-1}$ term could appear in \eqref{eq:V_infty separating}. This precisely means only the $\lambda=0$ term appears. If $a<0$, then $V_\infty=c$. However, if $a=0$, we obtain $V_\infty=c\varrho^{{-1}}$. If the latter case was true, then  we would have $V\sim\frac{1}{\varrho}$. By \eqref{eq:ansatz applied to write metric} and Lemma \ref{lem:volume lemma S^1}, the gravitational instanton would have maximal volume growth, which contradicts the assumption that the gravitational instanton is collapsed.
\end{proof}

As an obvious corollary, setting $\rho_\infty$ to be the distance function to the cone vertex in the asymptotic cone we have

\begin{corollary}
  In the asymptotic cone $V_\infty\equiv constant$ and $\varrho_\infty$ is a constant multiple of the distance function $\rho_\infty$.
\end{corollary}

\subsubsection{$\mathcal{K}$ only induces an $\mathbb{R}$-action}
\label{subsubsec:induce R-action}

We now repeat the arguments from Section \ref{subsubsec:induce S^1-action} for the case where $\mathcal{K}$ induces only an $\mathbb{R}$-action. Since $\mathcal{K}$ cannot flow points to infinity, we can take the closure of the $\mathbb{R}$-action in the isometry group, which results in an isometric $T^k$-action that is also Hamiltonian. The end remains diffeomorphic to $(N_0, \infty) \times L$, with the link $L$ admitting the $T^k$-action. We can again apply the symplectic reduction for $\mathcal{K}$, but only over suitably small open subsets of $L$ where the foliation induced by the $\mathbb{R}$-action is trivialized. Locally, the metric can still be written as in \eqref{eq:ansatz applied to write metric}.

\begin{lemma}
  The closure cannot be a $T^3$-action. Hence the closure is a $T^2$-action.
\end{lemma}
\begin{proof}
  If the closure were $T^3$, then the link $L$ would exactly be the $T^3$-orbit and in \eqref{eq:ansatz applied to write metric} the functions $u,V$ are both $T^3$-invariant. The $SU(\infty)$ Toda equation reduces to $(e^u)_{\varrho\varrho}=0$. By combining this with $V = -12\varrho + 6\varrho^2 u{\varrho}$, one can explicitly write down the metric $h$ and verify that no smooth, complete Hermitian non-K\"ahler gravitational instanton can exist in this setting.
\end{proof}

Hence the gravitational instanton is always toric in this setting. The most important lemma in Section \ref{subsubsec:induce S^1-action} is Lemma \ref{lem:volume lemma S^1}, which is a global consequence of the $SU(\infty)$ Toda equation. However, the proof of Lemma \ref{lem:volume lemma S^1} does not carry over here as we cannot do symplectic reduction globally. In the following we prove Lemma \ref{lem:volume lemma S^1} in the $\mathbb{R}$-action setting using toric geometry instead. First, we write down the toric metric locally in symplectic coordinate and compare it to Theorem \ref{thm:local ansatz version two}. Pick $\partial_{\phi_1}$ as the Hamiltonian Killing field $\mathcal{K}$ and $\partial_{\phi_2}$ as another Hamiltonian Killing field in $\mathrm{Lie}(T^2)$. Let $x_\alpha$ be the moment map of $\partial_{\phi_\alpha}$, hence $x_1=\xi$ is the moment map for $\mathcal{K}$. The K\"ahler metric $g=\xi^2h$ is given by
\begin{equation}\label{eq:toric Kahler}
  g=g_{\alpha\beta}dx_\alpha dx_\beta+g^{\alpha\beta}d\phi_\alpha d\phi_\beta
\end{equation}
where $g_{\alpha\beta}=\mathrm{Hess}f(x_1,x_2)$ and $(g^{\alpha\beta})=(g_{\alpha\beta})^{-1}$ with $f$ being the symplectic potential function. The Hermitian non-K\"ahler Ricci-flat metric is
\begin{equation}\label{eq:toric Ricci flat}
  h=\frac{1}{x_1^2}(g_{\alpha\beta}dx_\alpha dx_\beta+g^{\alpha\beta}d\phi_\alpha d\phi_\beta).
\end{equation}
Relate \eqref{eq:toric Ricci flat} to the ansatz Theorem \ref{thm:local ansatz version two} as follows. The metric on the level hypersurface of $x_1$ is $\frac{1}{x_2^2}(g_{22}dx_2^2+g^{\alpha\beta}d\phi_\alpha d\phi_\beta)$, and after further taking the quotient by $\partial_{\phi_1}$ is 
\begin{equation}\label{eq:metric on quotient toric ansatz}
\frac{1}{x_1^2}\left(g_{22}dx_2^2+\left(g^{22}-\frac{(g^{12})^2}{g^{11}}\right)d\phi_2^2\right).
\end{equation}
The function $V^{-1}=|\partial_{\phi_1}|_h^2=\frac{1}{x_1^2}g^{11}$, so comparing the metric \eqref{eq:metric on quotient toric ansatz} with $Ve^{u}(dx^2+dy^2)$ we have
$$e^u(dx^2+dy^2)=\frac{1}{x_1^4}(g^{11}g_{22}dx_2^2+\mathrm{det}(g^{\alpha\beta})d\phi_2^2).$$
Recall that $x+iy$ is any holomorphic coordinate we choose on the local quotient. Hence, we can simply pick $y$ as the parameter $\phi_2$. Since $x$ is the harmonic conjugate of $y$, parameters $x,x_2$ are related by 
\begin{equation}\label{eq:relation x x_2}
  \frac{dx}{dx_2}=\sqrt{\frac{g^{11}g_{22}}{\mathrm{det}(g^{\alpha\beta})}}.
\end{equation}
The function $e^u=\frac{1}{x_1^4}\mathrm{det}(g^{\alpha\beta})$. Furthermore, as $y={\phi_2}$ parametrizes the other symmetry the $SU(\infty)$ Toda equation reduces to
$$(e^u)_{\varrho\varrho}+u_{xx}=0.$$
Back to the setting of end. Link $L$ being a 3-manifold admitting a $T^2$-action, diffeomorphically it must be a lens space with the standard $T^2$-action, $S^2\times S^1$ with the standard $T^2$-action, or $T^3$ with a standard $T^2$-action. For the lens space case we pass to the universal cover $S^3$ to simplify the situation. Except the last case, on a suitably far away $T^2$-invariant end $M\setminus K$, there are two connected components of points that have isotropy. We proceed to calculate the integral $\int_{L}e^udxdy\eta$. Note that although $x,y$ are only local coordinates and the function $u$ depends on the choice of $x,y$, the 3-form $e^udxdy\eta$ indeed is globally well-defined on $L$.

If the link $L$ is $S^1\times T^2$, then without loss of generality we pick $\partial_{\phi_2}\in\mathrm{Lie}(T^2)$ to be an integral primitive direction that corresponds to an $S^1\subset T^2$. Fix $\varrho=\varrho_0$ and write $L$ as $[0,1]\times T^2/_\sim$, so that the moment map $x_2$ is well-defined over $[0,1]\times T^2$ and parametrizes $[0,1]$. Now that $x$  related to $x_2$ via \eqref{eq:relation x x_2} also parametrizes $[0,1]$, we assume the range of $x$ is $[a_1,a_2]$. Then
\begin{align*}
  \int_{L}(e^u)_{\varrho\varrho}dxdy\eta=-4\pi^2\int_{a_1}^{a_2}u_{xx}dx=-4\pi^2(u_x(a_2)-u_x(a_1))=0
\end{align*}
since $a_1,a_2$ correspond to the same point in $S^1$. As a result we conclude 
$$\int_Le^udxdy\eta=a\varrho+b$$
for some constant $a,b$. Using $V=-12\varrho+6\varrho^2u_\varrho$ we also have 
$$\int_{L}Ve^udxdy\eta=-6a\varrho^2-12b\varrho.$$
Recall $\varrho>0$. This leads to a contradiction, as both $\int_Le^udxdy\eta$ and $\int_{L}Ve^udxdy\eta$ are positive. Therefore this case cannot happen.

So we focus on the case that $L$ is $S^3$ or $S^2\times S^1$, and in both cases the $T^2$-action is the standard one. In a sufficiently far away $T^2$-invariant end $M\setminus K$, there are two components of points with isotropy, which correspond to the two component of boundaries in the momentum polytope $\Delta_{M\setminus K}$ of the end. Label these two boundaries as $P_1,P_2$. For simplicity we pick $\partial_{\phi_2}$ such that $x_2$ is the affine function on the polytope that is the distance function to $P_1$. The other moment map for $\mathcal{K}=\partial_{\phi_1}$ is $x_1=\varrho^{-1}$. Fix $\varrho=\varrho_0$, then this level hypersurface $L_0\simeq L$ at $\varrho_0$ is the preimage of the segment $x_1=\varrho^{-1}_0$ in the polytope $\Delta_{M\setminus K}$. The other moment map $x_2$ can be used to parametrize $x_1=\varrho^{-1}_0$ in the polytope $\Delta_{M\setminus K}$ and let us suppose the range of $x$ is $[0,a_2]$, where $x=0$ corresponds to the boundary $P_1$ and $a_2$ corresponds to $P_2$. From the reduced $SU(\infty)$ Toda equation,
$$\int_L(e^u)_{\varrho\varrho}dxdy\eta=-4\pi^2(u_x(a_2)-u_x(0)).$$
We now use the boundary behavior of the potential function $f$ to calculate the value of above. Near $P_1$, the smoothness condition of the toric K\"ahler metric $g$ is 
\begin{equation}\label{eq:smooth condition toric}
f(x_1,x_2)=\frac{1}{2}x_2\log x_2+f_S(x_1,x_2)
\end{equation}
where $f_S$ is a smooth function across the boundary $P_1$. The determinant of $\mathrm{Hess}f$ satisfies
$$\mathrm{det}(\mathrm{Hess}f)=\mathrm{det}(g_{\alpha\beta})=\frac{1}{\delta_S(x_1,x_2)\cdot x_2}$$
where $\delta_S$ is a smooth strictly positive function across $P_1$. Via \eqref{eq:relation x x_2} and $u=\log(\frac{1}{x_1^4}\mathrm{det}(g^{\alpha\beta}))$, near $P_1$ we have
\begin{align}
  u_x&=u_{x_2}\frac{dx_2}{dx}=\frac{1}{\mathrm{det}(g^{\alpha\beta})}\partial_{x_2}\mathrm{det}(g^{\alpha\beta})\sqrt{\frac{\mathrm{det}(g^{\alpha\beta})}{g^{11}g_{22}}}\notag\\
  &=\partial_{x_2}(\delta_S\cdot x_2)\frac{1}{g^{11}}\notag\\
  &\sim\frac{\delta_S}{g^{11}}.
\end{align}
A more detailed computation based on \eqref{eq:smooth condition toric} shows $u_x\sim\frac{\delta_S}{g^{11}}\sim2$ near $P_1$. One can repeat the above calculation  near the other boundary $P_2$ to conclude $u_x(a_2)=-2$. We therefore have the following lemma parallel to \eqref{eq:integral e^u} and \eqref{eq:integral Ve^u}.
\begin{lemma}
  There are constants $a,b$ such that
  \begin{equation}
    \int_Le^udxdy\eta=16\pi^2(\frac{1}{2}\varrho^2+a\varrho+b)
  \end{equation}
  and 
  \begin{equation}
    \int_LVe^udxdyd\eta=-96\pi^2a\varrho^2-192\pi^2b\varrho.
  \end{equation}
\end{lemma}
So the volume of the gravitational instanton can be explicitly calculated in terms of $\varrho$, which is parallel to Lemma \ref{lem:volume lemma S^1}.
\begin{lemma}\label{lem:volume T2}
  $\mathrm{Vol}(D_0\leq\varrho\leq D_1)=-32\pi^2a(D_1^3-D_0^3)-96\pi^2b(D_0^2-D_1^2)$.
\end{lemma}

The rest part of the proof goes almost the same, as Lemma \ref{lem:rescaled conformal cannot be trivial}-\ref{lem:equation for u_jnfty} are all essentially local. A small but important difference lies in that the collapsing now is along the $\mathbb{R}$-action, therefore the entire $T^2$-action is collapsed. The asymptotic cone is 2-dimensional, and it is easy to see that the asymptotic cone is the half-plane $\mathbb{H}$ by quotient $T^2$. We summarize as the following lemma, comparing to Lemma \ref{lem:structure of tangent cone for S^1}.
\begin{lemma}
  When $\mathcal{K}$ only induces an $\mathbb{R}$-action, the local geometry of the asymptotic cone $(M_\infty,h_\infty,p_\infty)$ is flat. Any such Hermitian non-K\"ahler gravitational instanton has unique asymptotic cone $\mathbb{H}$.
\end{lemma}

Similarly, as a corollary we have
\begin{corollary}
  In the asymptotic cone $V_\infty\equiv constant$ and $\varrho_\infty$ is a constant multiple of $\rho_\infty$.
\end{corollary}

\ 

In our previous work \cite{me} it was proved that in the non-collapsed situation the asymptotic cone can only be $\mathbb{C}^2/\Gamma$ with $\Gamma\subset U(2)$. Hence, including the non-collapsed situation, till now we have proved
\begin{theorem}[Classification of asymptotic cones]
  Any Hermitian non-K\"ahler gravitational instanton has unique asymptotic cone. The asymptotic cone can only be
  \begin{itemize}
    \item $\mathbb{R}^4/\Gamma$ in the non-collapsed situation.
    \item $\mathbb{R}^3/\mathbb{Z}_k$ or $\mathbb{H}$ in the collapsed situation.
  \end{itemize}
\end{theorem}
In the former case the collapsing happens along $S^1$-orbits while in the later case the collapsing happens along $T^2$-orbits. An immediate consequence is the curvature of any Hermitian non-K\"ahler gravitational instanton is decaying as $o(\rho^{-2})$, where $\rho$ is the distance function to a chosen base point. Any 4-manifold with only one end (it can have more ends but we stay with one end for simplicity) is said to be $\mathcal{AF}$ if its curvature decays as $o(\rho^{-2})$ and has unique asymptotic cone that is a metric cone. It was proved by Petrunin-Tuschmann \cite{petrunin} that for an $\mathcal{AF}$ 4-manifold whose end is simply-connected, asymptotic cone can only be $\mathbb{R}^4,\mathbb{R}^3$, or the half-plane $\mathbb{H}$, and they further conjectured $\mathbb{H}$ cannot be realized.  Recently, in our work \cite{ls} the conjecture is confirmed. 

\begin{theorem}[\cite{ls}]
  For an $\mathcal{AF}$ 4-manifold that is simply-connected at infinity, its asymptotic cone can only be $\mathbb{R}^4$ or $\mathbb{R}^3$.
\end{theorem}

A particular consequence is, if the asymptotic cone of a Hermitian non-K\"ahler gravitational instanton is $\mathbb{R}^3/\mathbb{Z}_k$ with non-trivial $\mathbb{Z}_k$ or $\mathbb{H}$, then the end of the gravitational instanton is non-simply-connected. Therefore, from our previous argument, we have the link $L$ in these cases diffeomorphically can only be $S^2\times S^1$, because it cannot be lens spaces, whose finite cover will be simply-connected. The action induced by $\mathcal{K}$ is the rotation on $S^1$ skewed by a multiple of the standard rotation on $S^2$. When it is a rational multiple, $\mathcal{K}$ induces an $S^1$-action, and we end with asymptotic cone $\mathbb{R}^3/\mathbb{Z}_k$. When it is an irrational multiple, $\mathcal{K}$ induces an $\mathbb{R}$-action, and asymptotic cone is $\mathbb{H}$. In the later case although the asymptotic cone is flat, the renormalized volume density function is proportional to the distance function to $\partial\mathbb{H}$. Finally, notice that the constants $a$ in Lemma \ref{lem:volume lemma S^1} and \ref{lem:volume T2} are non-zero, as in the asymptotic cone $\varrho_\infty$ is a multiple of the distance function and the growth of the renormalized limit measure on the asymptotic cone is cubic.

\begin{remark}
Recall our notation for ALE asymptotic models in and below Definition \ref{def:ale}.
It was proved in our previous paper \cite{me} that there is no Hermitian non-K\"ahler ALE gravitational instanton with structure group $\Upsilon$ in $SU(2)$, except the Eguchi-Hanson metric with reversed orientation.  
A crucial part of the proof is that the K\"ahler metric $g$ associated to a Hermitian non-K\"ahler ALE gravitational instanton can be naturally compactified to a K\"ahler orbifold $(\widehat{M},\widehat{g})$ by adding one orbifold point $q$.
In the ALE setting, there is no collapsing at infinity, and we can similarly consider the limit rescaled conformal factor $\lambda^{1/3}_\infty$. The limit metric $h_\infty$ is just the standard flat metric on the asymptotic cone $\mathbb{R}^4/\Gamma$ (it is proved that the structure group $\Upsilon\subset U(2)$) and $\lambda^{1/3}_\infty$ is globally defined on $\mathbb{R}^4/\Gamma$ outside the origin. 
In this situation the limit rescaled conformal factor $\lambda^{1/3}_\infty$ is a constant multiple of $1/\rho^2$. The limit K\"ahler metric ${g_\infty}=\lambda^{2/3}_\infty h_\infty=1/\rho^4 h_\infty$, which is flat and incomplete on the asymptotic cone.  The quotient $\mathbb{C}^2/\Upsilon$ with the flat K\"ahler metric $g_\infty$ is exactly the asymptotic cone of the compactified K\"ahler orbifold $(\widehat{M},\widehat{g})$ at the orbifold point $q$.
\end{remark}

We end this subsection with a remark on the local geometry of the limit K\"ahler metric $\widetilde{g_\infty}$. The limit rescaled conformal factor $\varrho_\infty^{-1}$ is always a constant multiple of $\rho_\infty^{-1}$ and $V_\infty\equiv constant$, therefore after the conformal change, we obtain
\begin{align}\label{eq:split kahler metric}
\widetilde{g_\infty}&=\frac{k^2}{\rho_\infty^2}(d\rho_\infty^2+\rho_\infty^2g_{S^2})+\frac{k^2}{\rho_\infty^2}\eta_\infty^2\\
&=\frac{k^2}{\rho_\infty^2}(d\rho_\infty^2+\eta_\infty^2)+k^2g_{S^2},\notag
\end{align}
which splits as the product of a sphere with a cusp.

\subsection{Classification of ends}
\label{subsec:classification of ends}

We proceed to improve the previous classification of asymptotic cones to a classification of ends. By saying that a gravitational instanton has a specific kind of asymptotic geometry, we mean it is asymptotic to the corresponding asymptotic model with rate $O(\rho^{-\delta_0})$ for some $\delta_0>0$.
\begin{theorem}[Classification of ends]\label{thm:classification of ends}
  For a Hermitian non-K\"ahler gravitational instanton $(M,h)$ it must have one of the following asymptotic geometry
  \begin{itemize}
    \item ALE with asymptotic cone $\mathbb{C}^2/\Gamma$ where $\Gamma\subset U(2)$.
    \item ALF-$A$ with asymptotic cone $\mathbb{R}^3$.
    \item Rational AF with asymptotic cone $\mathbb{R}^3/\mathbb{Z}_k$ or irrational AF with asymptotic cone right-half-plane $\mathbb{H}$. 
  \end{itemize}
\end{theorem}

The rest part of this subsection is devoted to proving this theorem. It suffices to study the case of non-maximal volume growth by our previous work on the ALE case \cite{me}. Pick a $\mathcal{K}$-invariant end $(M\setminus K,h)$. The following reduction is helpful.
\begin{enumerate}
  \item[(1)] If the asymptotic cone is $\mathbb{R}^3$, then we are in the case that the $S^1$-action is free and we can take the quotient of $M\setminus K$ by $S^1$ directly. Set $(M_0,h_0)$ to be the end $(M\setminus K,h)$. Denote the quotient by $S^1$ of the end as $(M_\flat,h_\flat)$.
  \item[(2)] If the asymptotic cone is $\mathbb{R}^3/\mathbb{Z}_k$ with non-trivial $\mathbb{Z}_k$ or $\mathbb{H}$, then the end diffeomorphically is $(N_0,\infty)\times S^2\times S^1$ where the action induced by $\mathcal{K}$ is a skewed rotation. Take the universal cover of $M\setminus K$ followed by a $\mathbb{Z}$-quotient with $\mathbb{Z}$ contained in the $\mathbb{R}$-action generated by $\mathcal{K}$, and denote the resulting end as $(M_0,h_0)$. Asymptotic cone of $(M_0,h_0)$ is $\mathbb{R}^3$, as the previous unwrapping operation turns the collapsing direction $\mathcal{K}$ into an integral $S^1$. Action induced by $\mathcal{K}$ in $(M_0,h_0)$ is a free $S^1$-action. The further quotient of $(M_0,h_0)$ by $S^1$ is denoted by $(M_\flat,h_\flat)$.
\end{enumerate}
Thus to prove Theorem \ref{thm:classification of ends} we only need to show $(M_0,h_0)$ above is ALF. We work with $(M_0,h_0)$ without mentioning it in the following of this subsection.

\begin{lemma}\label{lem:constant k0}
  There is a constant $k_0>0$ such that $V\to k_0^2$ as we approach infinity.
\end{lemma}
\begin{proof}
  Consider a specific asymptotic cone $(M,h_j,p)$ with $h_j=2^{-2j}h$. From the fact that $V_\infty\equiv constant$ we know along with the non-collapsed convergence $\widetilde{B_j}\to\widetilde{B_\infty}$ there are constants $C_j$ such that $C_jV\to1$. Now observe by the integral formulas \eqref{eq:integral e^u}-\eqref{eq:integral Ve^u} the ratio 
  $$\frac{\int_{\Sigma}Ve^u dxdy}{\int_{\Sigma}e^u dxdy}=\frac{-6a\varrho^2-12b\varrho}{4\pi\varrho^2+a\varrho+b}.$$
  Hence it is clear from the above formula that $V\to-\frac{3a}{2\pi}$ as we approach infinity. Note that $a<0$.
\end{proof}

The equation $V=-12\varrho+6\varrho^2u_{\varrho}$ now gives $u_{\varrho}=\frac{2}{\varrho}+\frac{k_0^2}{6\varrho^2}+o(\varrho^{-2})$. Thus $u=2\log\varrho+\varpi(x,y)-\frac{k_0^2}{6\varrho}+o(\varrho^{-1})$. Here again the metric $e^{\varpi(x,y)}(dx^2+dy^2)$ is the round metric $g_{S^2}$. Let $\varepsilon$ denote the error term 
$$\varepsilon:=u-2\log\varrho-\varpi(x,y)+\frac{k_0^2}{6\varrho}.$$ 
It is direct to see $\partial_{\varrho}^k\varepsilon=o(\varrho^{-k-1})$ by considering higher derivatives of $u$. Replacing $u$ by $2\log\varrho+\varpi-\frac{k_0^2}{6\varrho}+\varepsilon$ in the $SU(\infty)$ Toda equation we get
\begin{equation}\label{eq:expand spherical harmonics error}
\Delta_{S^2}\varepsilon+\partial_\varrho^2(\varrho^2\varepsilon)+O(\varrho^{-2})=0.
\end{equation}
Write $\varepsilon=\sum\varepsilon_\lambda f_{\lambda}$ with $f_\lambda$ being an eigenfunction of $S^2$ with eigenvalue $\lambda$ such that $\{f_\lambda\}$ forms an orthogonal basis of $L^2(S^2)$ and $O(\varrho^{-2})=\sum a_\lambda f_\lambda$. Then \eqref{eq:expand spherical harmonics error} becomes
$$\sum(\partial_\varrho^2(\varrho^2\varepsilon_\lambda)+\lambda\varepsilon_\lambda+a_\lambda)f_\lambda=0$$ 
and hence 
$$\partial_\varrho^2(\varrho^2\varepsilon_\lambda)+\lambda\varepsilon_\lambda+a_\lambda=0.$$
Note the eigenvalues of $\Delta_{S^2}$ are given by $-k(k+1)$ with $k\geq0$.
It is now a standard ODE analysis to see that $\varepsilon=O(\varrho^{-1-\delta_0})$ for any $\delta_0<1$. In conclusion $u=2\log\varrho+\varpi-\frac{k_0^2}{6\varrho}+O(\varrho^{-1-\delta_0})$ and $V=k_0^2+O(\varrho^{-\delta_0})$ for any $\delta_0<1$.

Together with our previous result \cite{me}, this confirms Theorem \ref{thm:classification of ends}. Particularly, for the choice of ALF/AF coordinate in Definition \ref{def:alf} and \ref{def:af}, by our discussion we can arrange $\mathcal{K}=-\partial_t$ and $\eta$ asymptotic to the connection form $\gamma$. We choose to put a negative sign before $\partial_t$ since it turns out to be more natural later. By scaling the metric suitably we can take the constant $k_0=1$ in Lemma \ref{lem:constant k0}. The asymptotic rate $\delta_0$ to the ALF/AF model can be chosen as any number that is close to but remains less than one. So for the rest of the paper, we can work under the assumption that the decay order $\delta_0$ in the ALF/AF assumption is a positive number that is very close to one.

Finally we have
\begin{theorem}\label{thm:expansion}
For a Hermitian non-K\"ahler ALF/AF gravitational instanton $(M,h)$ with a decay rate of $\delta_0$, the conformal factor $\lambda^{1/3}$ has an asymptotic expansion 
$$\lambda^{1/3}=\frac{1}{\rho}+O'(\rho^{-1-\delta_0}).$$ 
\end{theorem}

\begin{proof}
Recall by definition $\varrho=\lambda^{-1/3}$. From the expansion $V=1+O(\varrho^{-1-\delta_0})$ and $h=V(d\varrho^2+e^u(dx^2+dy^2))+V^{-1}\eta^2$ this clearly follows.
\end{proof}

\begin{remark}\label{rm:complete}
From the expansion, we in particular know that the real holomorphic vector field $J\mathcal{K}$ is complete, in the sense its flow exists on the entire $(-\infty,+\infty)$ beginning at any point, since $J\mathcal{K}=\nabla_h\lambda^{-1/3}$.
\end{remark}

\subsection{Asymptotic models that cannot be filled in}
\label{subsec:not realizable models}

Despite Theorem \ref{thm:classification of ends}, there are additional asymptotic models that cannot be filled in as Hermitian non-K\"ahler gravitational instantons, as discussed in Section \ref{subsec:define asymptotic models}. We now take a detour to classify all such asymptotic models.

\begin{theorem}\label{thm:asymptotic nonrealizable in detail}
  Given a Hermitian non-K\"ahler Ricci-flat metric on an end $(M,h)$ that is complete at infinity, though possibly incomplete inside, and has finite $\int |Rm|^2$, it must be asymptotic to one of the following asymptotic models with a polynomial rate at infinity: ALE, ALF-$A$, AF, skewed special Kasner, $\text{ALH}^*$.
\end{theorem}

We prove this theorem in this subsection. For an end incomplete inside, the scalar curvature $s_g$ can be either positive or negative. In Section \ref{subsec:analysis tangent cone}-\ref{subsec:classification of ends}, we only used the condition that the scalar curvature $s_g$ is positive, rather than requiring that the metric be complete inside (which would imply $s_g>0$). Therefore when $s_g>0$ our conclusion is the same as Theorem \ref{thm:classification of ends}, meaning that $h$ is asymptotic to an ALE, ALF-$A$, AF model. Therefore we turn to the case where $s_g<0$, which will be assumed in the rest of this subsection.

There are two situations again:
\begin{enumerate}
  \item[($\clubsuit$)] $\mathcal{K}$ induces an $S^1$-action.
  \item[($\spadesuit$)] $\mathcal{K}$ only induces an $\mathbb{R}$-action.  
\end{enumerate}

\subsubsection{$\mathcal{K}$ induces an $S^1$-action}
\label{subsubsec:nonrealizable S^1 case}

We briefly recall our argument in Section \ref{subsubsec:induce S^1-action}. Apply ansatz Theorem \ref{thm:local ansatz version two} to write the metric on the end $(-\infty,-N)\times L$ as 
\begin{equation}\label{eq:ansatz nonrealizable h}
h=\varrho^2g=V(d\varrho^2+e^u(dx^2+dy^2))+V^{-1}\eta^2.
\end{equation}
Label the surface arises from symplectic reduction as $\Sigma$. Here, $x+iy$ is a holomorphic coordinate over an affine open subset $U\subset\Sigma$. Notice that this time $\varrho=\xi^{-1}=-\lambda^{-1/3}<0$. There are again constants $a,b$ such that
\begin{align}
  \int_\Sigma e^u dxdy=2\pi\chi(\Sigma)\varrho^2+a\varrho+b,\label{eq:integral u not realizable}\\
  \int_\Sigma Ve^u dxdy=-6a\varrho^2-12b\varrho.\label{eq:integral Vu not realizable}
\end{align}
For an arbitrary asymptotic cone $(M,h_j,p)\to(M_\infty,d_\infty,p_\infty)$ and $q_j\to q_\infty\neq p_\infty$, set $\lambda^{1/3}_j:=\lambda^{-1/3}(q_j)\lambda^{1/3}$, $\mathcal{K}_j=-J\nabla_{h_j}\lambda^{-1/3}_j$, and $V_j:=|\nabla_{h_j}\lambda_j^{-1/3}|_{h_j}^{-2}=\lambda^{-2/3}(q_j)r_j^{-2}V$. Hence the rescaled metric $h_j$ is rewritten as
\begin{equation}
  h_j=r_j^{-2}V(d\varrho^2+e^u(dx^2+dy^2))+r_j^{-2}V^{-1}\eta^2=V_j(d\varrho_j^2+\lambda^{2/3}(q_j)e^u(dx^2+dy^2))+r_j^{-4}\lambda^{-2/3}(q_j)V_j^{-1}\eta^2.
\end{equation}
The rescaled conformal factor $\lambda_j^{1/3}$ converges smoothly to a limit rescaled conformal factor $\lambda_\infty^{1/3}$ up to subsequences, so $\varrho_j:=-\lambda_j^{-1/3}$ converges smoothly to $\varrho_\infty:=-\lambda_\infty^{-1/3}$ as well. By definition, $V_j$ converges smoothly to a limit $V_\infty$. Lemma \ref{lem:volume lemma S^1}-\ref{lem:rescaled conformal cannot be trivial} still hold by the exactly same proof so $\lambda_\infty^{1/3}$ is not identically a constant.

If $\chi(\Sigma)>0$, then Lemma \ref{lem:tangent cone 3-dim S^1}-\ref{lem:structure of tangent cone for S^1} also remain valid with the completely same proof. One can then continue with the same argument in Section \ref{subsec:classification of ends} to conclude the asymptotic model is ALF-$A$ or AF. Notice that $\chi(\Sigma)$ cannot be negative by \eqref{eq:integral u not realizable}.

Thus we focus on the case $\chi(\Sigma)=0$. If $\chi(\Sigma)=0$, then $\Sigma$ is a torus $T^2$. For simplicity suppose $x+iy$ parametrizes $\mathbb{C}$ and identify $T^2$ as $\mathbb{C}/\langle1,\alpha+i\beta\rangle$ in \eqref{eq:ansatz nonrealizable h}.  One needs to replace Lemma \ref{lem:tangent cone 3-dim S^1} by the following lemma. An important difference lies in now the symplectic reduction $T^2$ is collapsed, while in Lemma \ref{lem:tangent cone 3-dim S^1}, $\Sigma$ is non-collapsed.

\begin{lemma}\label{lem:nonrealizable u convergence}
  The function $u+\log\lambda^{2/3}(q_j)\to-\infty$. Therefore the symplectic reduction $T^2$ is collapsed in the asymptotic cone in the sense that for $(M,h_j,p)$, the diameter of $T^2$ is going to zero.
\end{lemma}
\begin{proof}
  If not, then due to the smooth convergence of the symplectic reduction along with the smooth non-collapsed convergence $\widetilde{B_j}\to\widetilde{B_\infty}$, the function $u+\log\lambda^{2/3}(q_j)$ would converge smoothly to a limit and the proof of Lemma \ref{lem:tangent cone 3-dim S^1}-\ref{lem:structure of tangent cone for S^1} goes through. As a consequence, the metric $\lambda^{2/3}(q_j)e^{u-2\log\varrho_j}(dx^2+dy^2)$ on the symplectic reduction $T^2$ would converge to a metric with curvature one, which is a contradiction.
\end{proof}

\begin{lemma}\label{lem:nonrealizable K collapse}
  The number $\lambda^{-1}(q_j)r_j^{-2}$ converges to a non-zero constant and local geometry of asymptotic cone is always Type II, which means the Ricci-flat metric $\widetilde{h_\infty}$ is Type II. The $S^1$-action induced by $\mathcal{K}$ is also collapsed in the asymptotic cone. That is, $r_j^{-4}\lambda^{-2/3}(q_j)V_j^{-1}\to0$. 
\end{lemma}
\begin{proof}
  If the local geometry of the asymptotic cone were Type I, then the K\"ahler metric $\widetilde{g_\infty}$ is K\"ahler scalar-flat and conformal to the Ricci-flat metric $\widetilde{h_\infty}$, and argument in Lemma \ref{lem:structure of tangent cone for S^1} would go through, which gives a contradiction. Hence $W^+_{\widetilde{h_\infty}}$ never vanishes, and $\lambda^{-1}(q_j)r_j^{-2}$ converges to a non-zero constant, since by definition $\lambda=2\sqrt{6}|W^+|_{h}$. 
\end{proof}

Summarizing Lemma \ref{lem:nonrealizable u convergence}-\ref{lem:nonrealizable K collapse}, any asymptotic cone must be 1-dimensional. By Cheeger-Fukaya-Gromov's theory on collapsing with bounded curvature, it is clear to see here that collapsing happens along the link $L$ that is an $S^1$-bundle over $T^2$, which is a nilpotent manifold. When the bundle is trivial, the group action by $G_\infty$ on $(\widetilde{B_\infty},\widetilde{h_\infty})$ is abstractly an action by the abelian $\mathbb{R}^3$. When the bundle is non-trivial, the group action is abstractly an action by the Heisenberg group $\mathcal{H}$. As we have seen collapsing happens along an $S^1$-bundle over $T^2$, the collapsed fibers cannot be infranilmanifolds.

Because of the smooth convergence of symplectic reduction there are constants $C_j$ such that $u+C_j$ converges smoothly to a limit $u_\infty$, with $C_j\gg \log\lambda^{2/3}(q_j)$. Set $u_j:=u+C_j$ and $x_j:=\lambda^{1/3}(q_j)e^{-\frac{1}{2}C_j}x,y_j:=\lambda^{1/3}(q_j)e^{-\frac{1}{2}C_j}y$. Then along with $\widetilde{B_j}\to\widetilde{B_\infty}$, $x_j,y_j$ converge to limits $x_\infty,y_\infty$ that parametrize symplectic reductions in $\widetilde{B_\infty}$. Non-collapsed metric $\widetilde{h_\infty}$ over $\widetilde{B_\infty}$ takes the form
\begin{equation}\label{eq:limit noncollapsed metric nonrealizable}
  \widetilde{h_\infty}=V_\infty(d\varrho_\infty^2+e^{u_\infty}(dx_\infty^2+dy_\infty^2))+V_\infty^{-1}\eta_\infty^2.
\end{equation}
The $SU(\infty)$ Toda equation in the limit now becomes 
\begin{equation}\label{eq:limit Toda nonrealizable}
  \partial_{\varrho_\infty}^2(e^{u_\infty})+\partial^2_{x_\infty}u_{\infty}+\partial^2_{y_\infty}u_{\infty}=0.  
\end{equation}
asymptotic cone is $M_\infty=(-\infty,0]$ equipped with the metric $h_\infty=V_\infty d\varrho_\infty^2$.

\begin{lemma}\label{lem:nonrealizable invariant V u}
  Functions $V_\infty$ and $u_\infty$ are invariant under $G_\infty$, therefore they are functions of $\varrho_\infty$ only.
\end{lemma}
\begin{proof}
  By definition $V_\infty=|\partial_{\varrho_\infty}|^2_{\widetilde{h_\infty}}$ clearly is $G_\infty$-invariant. As symplectic reductions $T^2$ are collapsed, by our choice of $x$ and $y$, $u_\infty$ is also $G_\infty$-invariant.
\end{proof}

Therefore \eqref{eq:limit Toda nonrealizable} reduces to $\partial_{\varrho_\infty}^2(e^{u_\infty})=0$, which implies there are constants $A,B$ such that $e^{u_\infty}=A\varrho_\infty+B$. It follows $V_\infty=-12\varrho_\infty+6\varrho_\infty^2\partial_{\varrho_\infty}(u_\infty)=-\frac{6A\varrho_\infty^2+12B\varrho_\infty}{A\varrho_\infty+B}$.

\begin{lemma}\label{lem:nonrealizable either kasner or ALH*}
  Either $A>0$ and $B=0$, or $A=0$ and $B>0$.
\end{lemma}
\begin{proof}
  Recall \eqref{eq:integral u not realizable}-\eqref{eq:integral Vu not realizable}. As $V_\infty$ is purely a function of $\varrho_\infty$, it is clear that if $a<0$ in \eqref{eq:integral u not realizable}-\eqref{eq:integral Vu not realizable}, then $V_\infty=-6\varrho_\infty$ and $A>0,B=0$.  If $a=0$ in \eqref{eq:integral u not realizable}-\eqref{eq:integral Vu not realizable}, then $V_\infty=-12\varrho_\infty$ and $A=0,B>0$. 
\end{proof}

One can determine $\eta_\infty$ via \eqref{eq:deta}.
\begin{itemize}
  \item When $a<0$, we have $A>0,B=0$, and \eqref{eq:deta} gives $d\eta_\infty=0$, and in a suitable gauge 
  $$\widetilde{h_\infty}=-6\varrho_\infty\left(d\varrho_\infty^2-\varrho_\infty(dx_\infty^2+dy_\infty^2)\right)-\frac{1}{6\varrho_\infty}dt^2.$$
  Hence we recover the special Kasner metric \eqref{eq:Kasner0},
  \item When $a=0$, we have $A=0,B>0$, and \eqref{eq:deta} gives $d\eta_\infty=-12dx_\infty dy_\infty$, and in a suitable gauge 
  $$\widetilde{h_\infty}=-12\varrho_\infty\left(d\varrho_\infty^2+dx_\infty^2+dy_\infty^2\right)-\frac{1}{12\varrho_\infty}(dt-12x_\infty dy_\infty)^2.$$
  Hence we recover the $\text{ALH}^*$ metric \eqref{eq:ALH*}.
\end{itemize}
A particular consequence of Lemma \ref{lem:nonrealizable either kasner or ALH*} is the uniqueness of asymptotic cone as well as the local geometry of it in this situation. In summary we have

\begin{proposition}
  For a Hermitian non-K\"ahler Ricci-flat end $(M,h)$ that has finite $\int|Rm_h|_h^2dvol_h$, when $s_g<0$ and $\mathcal{K}$ induces an $S^1$-action, it has unique asymptotic cone. Moreover, the local geometry of the asymptotic cone is determined by the value of $a$ in \eqref{eq:integral u not realizable}-\eqref{eq:integral Vu not realizable}.
\end{proposition}

Next we improve the above to the existence of asymptotic models. 
\begin{lemma}
  $-\frac{A\varrho+B}{6A\varrho^2+12B\varrho}\cdot V\to1$ as we approach infinity.
\end{lemma}
\begin{proof}
  This is because we know after renormalizing $V$ to $V_j=\lambda^{-2/3}(q_j)r_j^{-2}V$ it converges to $V_\infty=-\frac{6A\varrho_\infty^2+12B\varrho_\infty}{A\varrho_\infty+B}$ and we have integral formulas \eqref{eq:integral u not realizable}-\eqref{eq:integral Vu not realizable}. The ratio $\int_\Sigma Ve^udxdy/\int_\Sigma e^udxdy=-\frac{6a\varrho^2+12b\varrho}{a\varrho+b}$.
\end{proof}

Since $V=-12\varrho+6\varrho^2u_\varrho$, we have $u_\varrho=\frac{1}{\varrho}+o(\varrho^{-1})$ when $a<0$, or $u_\varrho=o(\varrho^{-1})$ when $a=0$.

\begin{proposition}\label{prop:kasner asymptotic}
  When $a<0$, $(M,h)$ is asymptotic at a polynomial rate to a special Kasner metric where $\partial_t$ induces an $S^1$-action. That is, there is a lattice $\Lambda\subset\mathbb{R}^3$, such that $\partial_t$ induces an integral $S^1$ in $\mathbb{R}^3/\Lambda$ and $h=h_{\text{Kasner},\Lambda}+O(\rho^{-\delta_0})$ for some $\delta_0>0$.
\end{proposition} 
Note that for special Kasner metrics the distance function is comparable to $|\varrho|^{3/2}$.
\begin{proof}
  By considering the convergence of higher order derivatives of $u_j$ to $u_\infty$ one has $\partial_{\varrho}^{k}(\partial_\varrho u-\frac{1}{\varrho})=o(\varrho^{-k-1})$ for $k\geq0$. Denote the average of $u$ on symplectic reductions $T^2$ as $\overline{u_{T^2}}$, which is a function only depends on $\varrho$. The $SU(\infty)$ Toda equation gives 
  $$(u-\overline{u_{T^2}})_{xx}+(u-\overline{u_{T^2}})_{yy}=-(e^u)_{\varrho\varrho}.$$
  Since $\|u-\overline{u_{T^2}}\|_{H^{k+2}(T^2)}\leq C\|\Delta_{T^2}(u-\overline{u_{T^2}})\|_{H^k(T^2)}$, over each $T^2$ we obtain $\|u-\overline{u_{T^2}}\|_{H^k(T^2)}\leq C\frac{1}{\varrho}$ for any $k\geq0$, and Sobolev inequality then gives $\|u-\overline{u_{T^2}}\|_{C^k(T^2)}\leq C\frac{1}{\varrho}$. Now look at $\int_{T^2}e^udxdy=a\varrho+b$. Replace $u$ by $\overline{u_{T^2}}+O(\varrho^{-1})$ we get $\int_{T^2} e^{\overline{u_{T^2}}}dxdy=(a\varrho+b)e^{O(\varrho^{-1})}$. So we see $u=\log\varrho+a'+O(\varrho^{-1})$ for some constant $a'$. From this it is direct to derive the conclusion.
\end{proof}

\begin{proposition}\label{prop:ALH* asymptotic}
  When $a=0$, $(M,h)$ is asymptotic at a rate that is faster than any polynomial rate to a nilpotent $\text{ALH}^*$ metric. That is, there is a cocompact lattice $\Gamma\subset\mathcal{H}$, such that $h=h_{\text{ALH}^*-\Gamma}+O(\varrho^{-\delta_0})$ for any $\delta_0>0$.
\end{proposition}
For $\text{ALH}^*-\Gamma$ metrics the distance function again is comparable to $|\varrho|^{3/2}$.
\begin{proof}
  This time we have $\partial_\varrho^{k+1}u=o(\varrho^{-k-1})$ for $k\geq0$. Again let $\overline{u_{T^2}}$ be the average of $u$ over symplectic reductions $T^2$. Repeating the proof of Proposition \ref{prop:kasner asymptotic} one has $\|u-\overline{u_{T^2}}\|_{C^k(T^2)}=o(\varrho^{-1})$. Beginning with $s=1$, we do the following iteration steps. By $\|u-\overline{u_{T^2}}\|_{C^k(T^2)}=o(\varrho^{-s})$ and $\int_{T^2}e^udxdy=b$, there is a constant $a'$ such that $u=a'+o(\varrho^{-s})$. Substitute into the $SU(\infty)$ Toda equation we have $\Delta_{T^2}(u-\overline{u_{T^2}})=o(\varrho^{-s-2})$. It follows $\|u-\overline{u_{T^2}}\|_{C^k(T^2)}=o(\varrho^{-s-2})$, so the initial decay rate is improved by two. Using $\int_{T^2}e^u dxdy=b$ one can see $u=a'+o(\varrho^{-s-2})$. This finishes the proof.
\end{proof}

Summarizing Proposition \ref{prop:kasner asymptotic}-\ref{prop:ALH* asymptotic} we have finished the proof of Theorem \ref{thm:asymptotic nonrealizable in detail} in the situation that $\mathcal{K}$ induces an $S^1$-action.

\subsubsection{$\mathcal{K}$ only induces an $\mathbb{R}$-action}

As in Section \ref{subsubsec:induce S^1-action}, $\mathcal{K}$ must induce a $T^k$-action. Diffeomorphically, the end $M$ is $(-\infty, -N) \times L$, where the link $L$ admits a $T^k$-action. The link $L$ can only be a lens space, $S^2 \times S^1$, or $T^3$. In the first two cases, the action is a $T^2$-action, while for $L \simeq T^3$, the action can be either a $T^2$-action or a $T^3$-action. 

When $L$ is a lens space or $S^2\times S^1$, the proof in Section \ref{subsubsec:induce S^1-action} directly applies, where the key point is the topology prevents the link to be entirely collapsed. The conclusion is $(M,h)$ is asymptotic to an AF model. 

When $L$ is $T^3$, one can consider its cover $\mathbb{R}^3$ and take a different quotient so that $\mathcal{K}$ induces an $S^1$-action in the new link $L' \simeq T^3$. This results in a new end $(M', h')$, which is a Hermitian non-K\"ahler Ricci-flat space with finite $\int |Rm|^2$ and where the extremal vector field induces an $S^1$-action. This case reduces to the situation discussed in Section \ref{subsubsec:induce S^1-action}. Since the link of $(M', h')$ is $T^3$, it is asymptotic to a special Kasner model. Therefore, the original end $(M, h)$ is asymptotic to a skewed special Kasner model. This completes the proof of Theorem \ref{thm:asymptotic nonrealizable in detail} for the $\mathbb{R}$-action case.

\section{Compactification}
\label{sec:compactification}

As we mentioned lots of times, we have the following two situations.
\begin{enumerate}
\item[($\clubsuit$)] $\mathcal{K}$ induces an $S^1$-action, in which case the gravitational instanton is ALF or rational AF, depending on whether the $S^1$-action is free. Asymptotic cone is $\mathbb{R}^3$ or $\mathbb{R}^3/\mathbb{Z}_k$ with non-trivial $\mathbb{Z}_k$.
\item[($\spadesuit$)] $\mathcal{K}$ only induces an $\mathbb{R}$-action. In this case, the closure is a $T^2$-action. The gravitational instanton $(M,h)$ is irrational AF, and is holomorphically isometrically toric automatically. Asymptotic cone is $\mathbb{H}$. We shall explain later that \cite{biquardgauduchon} can be applied directly, enabling the classification under this circumstance. 
\end{enumerate}
We therefore concentrate on Hermitian non-K\"ahler ALF/AF gravitational instantons where $\mathcal{K}$ induces $S^1$-action. The vector fields $\mathcal{K}$ and $J\mathcal{K}$ together induce a holomorphic $\mathbb{C^*}$-action.

\subsection{Local ansatz revisited} 
\label{subsec:localansatztogravitational}

We work under the assumption that the $S^1$-action induced by $\mathcal{K}=\partial_t$ is free on the end of the gravitational instanton. The semi-free case is similar and one only needs to replace the quotient by orbifolds.
We apply the ansatz Theorem \ref{thm:localansatz} this time. With the expansion for $\lambda^{1/3}$, the moment map $\xi$ asymptotically is 
\begin{equation}\label{eq:expansionxi}
\xi=\lambda^{1/3}=\frac{1}{\rho}+O'(\rho^{-1-\delta_0}).
\end{equation}
Employing symplectic reduction, we can put the metric $g$ on $M\setminus K$ into the form
\begin{equation}\label{eq:metricglocal}
Wd\xi^2+W^{-1}\eta^2+g_{\xi}.
\end{equation}
Here $g_\xi$ is the metric on $\Sigma=\mathbb{P}^1$ arising from the symplectic reduction. The Hermitian Ricci-flat metric $h$ is related to $g$ via $g=\lambda^{2/3}h$, so the metric $h$ is
\begin{equation}
h=\lambda^{-2/3}(Wd\xi^2+W^{-1}\eta^2+g_{\xi}).
\end{equation}
If we pick a holomorphic coordinate function $z=x+iy$ on $\Sigma$, then the complex structure $J$ is given by $J:d\xi\to -W^{-1}\eta$ and $J:dx\to-dy$. The metric $g$ then further takes the form 
$$g=Wd\xi^2+W^{-1}\eta^2+We^v(dx^2+dy^2)$$ 
with some function $v$. This is exactly ansatz Theorem \ref{thm:localansatz}. Let $p$ be the holomorphic map $p:M\setminus K\to \Sigma$, which is the complex quotient map. Suppose the moment map takes value $(0,\epsilon]$ on the end $M\setminus K$. Since we derived the asymptotic expansion for the conformal factor $\lambda^{1/3}$ and consequently the moment map $\xi$, the asymptotic behavior of functions $W$ and $v$ in the ansatz can be analyzed as follows, which are just some straightforward computations.

\begin{lemma}
The function $W^{-1}=|\partial_t|^2_g$ has the following asymptotic expansion:
\begin{equation}\label{eq:expansionw}
W^{-1}=\frac{1}{\rho^2}+O'(\rho^{-2-\delta_0}).
\end{equation}
\end{lemma}
Notice that $W$ does not depend on the choice of the holomorphic coordinate $z=x+iy$.

\begin{lemma}
If we pick coordinate function $z=x+iy$ on an affine open subset $U\subset\mathbb{P}^1$, then on the open set $p^{-1}(U)\subset M\setminus K$, the function $v$ has the asymptotic expansion
\begin{equation}\label{eq:asymptoticv}
v(\rho,x,y)=\varpi(x,y)-2\log\rho+O'(\rho^{-\delta_0})
\end{equation}
with some function $\varpi$ that only depends on $x$ and $y$. The function $\varpi$ satisfies the equation $\varpi_{xx}+\varpi_{yy}=-2e^\varpi$.
\end{lemma}
Notice that the function $\varpi$ depends on the choice of the coordinate function $z=x+iy$. However, $e^{\varpi}(dx^2+dy^2)$ is a well-defined metric and has curvature one.

\subsection{Compactification of Hermitian non-K\"ahler ALF gravitational instantons}

\begin{theorem}\label{thm:compactifyalf}
For a Hermitian non-K\"ahler ALF gravitational instanton $(M,h)$,  the complex surface $M$ can be naturally compactified by adding a holomorphic $\mathbb{P}^1$ at infinity to a smooth compact complex surface. 
\end{theorem}

\begin{proof}

We consider the complex structure on the open set $p^{-1}(U)$, where the $S^1$-bundle is trivialized and we use the coordinate system $(\xi,{t},x,y)$. 
On $p^{-1}(U)$, the connection one-form $\eta$ can be written as $d{t}+\mathcal{X}dx+\mathcal{Y}dy+\mathcal{Z}d\xi$, and the complex structure is given by
$$J:d\xi\to -W^{-1}\eta,\ dx\to-dy.$$
As $\eta$ is $\partial_{t}$-invariant, $\mathcal{X}(\xi,x,y),\mathcal{Y}(\xi,x,y),$ and $\mathcal{Z}(\xi,x,y)$ are functions of $\xi,x,y$ only.
The following two differential forms are of type $(1,0)$, which determine the complex structure on $p^{-1}(U)$ completely
$$d\xi+iW^{-1}\eta,\ dx+idy.$$
A straightforward calculation shows that in terms of the basis of vector fields  $\partial_\xi,\partial_{t},\partial_x,\partial_y$ under the system $(\xi,{t},x,y)$, the complex structure is given by
\begin{align}
J{\partial_\xi}&=W{\partial_{t}}-W^{-1}\mathcal{Z}\partial_\xi+W^{-1}\mathcal{Z}^2\partial_{t},\\
J{\partial_x}&={\partial_y}-W^{-1}\mathcal{X}{\partial_\xi}+(-\mathcal{Y}+W^{-1}\mathcal{X}\mathcal{Z}){\partial_{t}}.
\end{align}

\paragraph{\textbf{\underline{\underline{Step 1}}}}

In this step we choose a suitable gauge.

Recall $\xi\in(0,\epsilon]$. In the coordinate system $(\xi,{t},x,y)$, we introduce 
\begin{align}\label{eq:coorchangezeta}
\zeta&:=\int^\epsilon_{\xi}Wd\xi,\\
\varphi&:=\left(-{t}+\int^\epsilon_\xi\mathcal{Z}d\xi\right)\mod{2\pi}.\label{eq:coorchangephi}   
\end{align}
The new parameter $\zeta$ takes value in $[0,\infty)$ because of the asymptotic expansion for $W$. Next we consider the coordinate change $(\xi,{t},x,y)\to(\zeta,\varphi,x,y)$. It is very crucial to notice that $\partial_x,\partial_y$ are different vectors in the system $(\xi,{t},x,y)$ and $(\zeta,\varphi,x,y)$. In the following we will always specify which coordinate system we are using when we write $\partial_x,\partial_y$. After performing a computation, we have the following
\begin{equation}\label{eq:vectorrelation}
\partial_\xi=-W\partial_\zeta-\mathcal{Z}\partial_\varphi,\ \partial_{t}=-\partial_\varphi.
\end{equation}
In the revised gauge $(\zeta,\varphi,x,y)$, conducting calculations based on the coordinate transformation and utilizing equation (\ref{eq:deta}), we deduce that the complex structure in the revised gauge is
\begin{align}\label{eq:Jxy}
J\partial_x&=\partial_y+\mathcal{X}(\epsilon,x,y)\partial_\zeta+\mathcal{Y}(\epsilon,x,y)\partial_\varphi,\\
J\partial_\zeta&=\partial_\varphi.
\end{align}
Here, the functions $\mathcal{X}(\epsilon,x,y)$ and $\mathcal{Y}(\epsilon,x,y)$ denote the respective values of the functions $\mathcal{X}(\xi,x,y)$ and $\mathcal{Y}(\xi,x,y)$ evaluated at $\xi=\epsilon$. Equation (\ref{eq:Jxy}) does not include the variables $\zeta$ and $\varphi$.  This coincides with the fact that the complex structure $J$ remains invariant under the real holomorphic vector fields $\partial_\varphi$ and $\partial_\zeta=-J\partial_\varphi$.  The significance of adopting this gauge is that the complex structure now becomes standard along the $\partial_\zeta$ and $\partial_\varphi$ directions.

\begin{comment}

\paragraph{\textbf{\underline{\underline{Step 2}}}}
In this step we use the invariance of the complex structure $J$ under the holomorphic $\mathbb{C}^*$-action induced by the real holomorphic vector fields $\partial_{t}$ and $J\partial_{t}$. 

Since $\partial_\varphi$ is real holomorphic, the vector field $-J\partial_\varphi=\partial_{\zeta}$ is also real holomorphic.  These vector fields preserve the complex structure $J$, hence in the equation $J\partial_x=\partial_y+\mathcal{Z}\partial_\zeta+\Phi\partial_\varphi$, the functions $\mathcal{Z}$ and $\Phi$ are independent of the variables $\zeta$ and $\varphi$. So they are only functions of $x$ and $y$ in our gauge:
$$\mathcal{Z}=\mathcal{Z}(x,y),\  \Phi=\Phi(x,y).$$

\end{comment}

\paragraph{\textbf{\underline{\underline{Step 2}}}} With this step, we complete the process of compactifying the complex structure. 

Consider $w=\exp(-\zeta-i\varphi)$. From the perspective of the smooth structure, we compactify $p^{-1}(U)$ as $\overline{p^{-1}(U)}=p^{-1}(U)\sqcup U$, by claiming that $\mathrm{Re}\,w$, $\mathrm{Im}\,w$, $x$, and $y$ serve as smooth coordinate functions on $\overline{p^{-1}(U)}$. The equation $w=0$ defines the added $U$ at infinity. 
In the coordinate given by $(\mathrm{Re}\,w,\mathrm{Im}\,w,x,y)$, the complex structure reads as
\begin{align}
&J\partial_x=\partial_y+\mathcal{X}(\epsilon,x,y)(-\mathrm{Re}\, w\partial_{\mathrm{Re}\, w}-\mathrm{Im}\, w\partial_{\mathrm{Im}\, w})+\mathcal{Y}(\epsilon,x,y)(\mathrm{Im}\, w\partial_{\mathrm{Re}\, w}-\mathrm{Re}\, w\partial_{\mathrm{Im}\, w}),\label{eq:complexstructure2}\\
&J\partial_{\mathrm{Re}\, w}=\partial_{\mathrm{Im}\, w}\label{eq:complexstructure1}.
\end{align}
Now the complex structure $J$ on $p^{-1}(U)$ clearly can be smoothly extended to $\overline{p^{-1}(U)}$ with our chosen smooth structure on $\overline{p^{-1}(U)}$, since the coefficients of the complex structure in this gauge are all smooth functions of $(\mathrm{Re}\,w,\mathrm{Im}\,w,x,y)$.

Equations (\ref{eq:complexstructure2}) and (\ref{eq:complexstructure1}) show that along the added $\Sigma$, the complex structure is:
$$J\left|_\Sigma(\partial_x)\right.=\partial_y,\ J\left|_\Sigma(\partial_{\mathrm{Re}\,w})\right.=\partial_{\mathrm{Im}\,w}.$$
Consequently, the added real surface constitutes a holomorphic curve in $\overline{M}$. The holomorphic curve must be a $\mathbb{P}^1$ given that topologically it is $\Sigma=S^2$. The theorem is proved.
\end{proof}

\subsection{Positivity of the compactified surfaces for ALF gravitational instantons}\label{subsec:alfstructure}
We label the compactified surface as $\overline{M}$ and denote the added divisor $\mathbb{P}^1$ as $D$. 
Denote the hermitian metric induced by the K\"ahler metric $g$ on the anti-canonical bundle $-K_M$ as $g_{-K}$. Subsequently, the metric $e^{-2\log s_g}g_{-K}$ is also a hermitian metric on the line bundle $-K_M$. This particular hermitian metric holds significant importance due to its positive curvature, as initially observed by LeBrun \cite{lebrun95}.

\begin{proposition}\label{prop:metricbasicestimate}
The hermitian metric $e^{-2\log s_g}g_{-K}$ on the line bundle $-K_M$ extends to a singular hermitian metric on the line bundle $-(K_{\overline{M}}+D)$ over $\overline{M}$. As a singular hermitian metric, it has the following regularity
\begin{itemize}
\item It is smooth outiside $D$ and continuous across the divisor $D$, in the sense that if there is a local holomorphic trivialisation $\tau:-(K_{\overline{M}}+D)|_{V}\to V\times\mathbb{C}$ of $-(K_{\overline{M}}+D)$ on an open set $V$ near the divisor $D$, then $e^{-2\log s_g}g_{-K}(s,s)=|\tau(s)|^2e^{-2\phi}$ for any smooth section $s$, with a continuous potential function $\phi$.

\item The Lelong number $\nu(\phi,d)\coloneqq\liminf_{l\to d}\frac{\phi(l)}{\log|l-d|}=0$ for any $d\in D$.
\end{itemize}
\end{proposition}

\begin{proof}
We consider locally on $\overline{p^{-1}(U)}$. 
Equations (\ref{eq:complexstructure2}) and (\ref{eq:complexstructure1}) show that
\begin{align}
Jd\mathrm{Re}\,w=&-d\mathrm{Im}\,w+(-\mathcal{X}(\epsilon,x,y)\mathrm{Re}\,w+\mathcal{Y}(\epsilon,x,y)\mathrm{Im}\,w)dx\label{eq:complexstructure3}\\
&+(-\mathcal{X}(\epsilon,x,y)\mathrm{Im}\,w-\mathcal{Y}(\epsilon,x,y)\mathrm{Re}\,w)dy,\notag\\
Jdx=&-dy.\label{eq:complexstructure4}
\end{align}
A computation yields $2\bar\partial w=w((-i(\mathcal{X}(\epsilon,x,y)+i\mathcal{Y}(\epsilon,x,y))dx-(\mathcal{X}(\epsilon,x,y)+i\mathcal{Y}(\epsilon,x,y))dy)$.
The differential form $-i(\mathcal{X}(\epsilon,x,y)+i\mathcal{Y}(\epsilon,x,y))dx-(\mathcal{X}(\epsilon,x,y)+i\mathcal{Y}(\epsilon,x,y))dy$ is of type $(0,1)$ on $U\subset \Sigma$. Consequently, we can find a smooth function $F(x,y)$ defined on $U$, such that $2\bar\partial F=-i(\mathcal{X}(\epsilon,x,y)+i\mathcal{Y}(\epsilon,x,y))dx-(\mathcal{X}(\epsilon,x,y)+i\mathcal{Y}(\epsilon,x,y))dy$ over $U$. The modified function $w e^{-F(x,y)}$ has now become a holomorphic function over $\overline{p^{-1}(U)}$, serving as a holomorphic defining function for the divisor $D$ within $\overline{p^{-1}(U)}$.
The pair of holomorphic functions $we^{-F}$ and $x+iy$ give a holomorphic coordinate system for $\overline{p^{-1}(U)}$. 
Notably, the two-form $s\coloneqq d(we^{-F})\wedge(dx+idy)$ is of type $(2,0)$ and is a holomorphic section of $K_{\overline{M}}$ over $\overline{p^{-1}(U)}$. The two-form $\frac{1}{we^{-F}}s$ constitutes a holomorphic basis of $K_{\overline{M}}+D$ since $we^{-F}$ is a defining function of the divisor $D$.

It suffices to establish the extension of the hermitian metric $e^{2\log s_g}g_{K}$ on $K_{M}$ to a singular hermitian metric on $K_{\overline{M}}+D$, with the corresponding regularity properties.  
Therefore, our task reduces to computing the norm of the holomorphic section $\frac{1}{we^{-F}}s$ under the $e^{2\log s_g}g_{K}$ metric within $\overline{p^{-1}(U)}$.
Let us proceed with the calculations
\begin{align*}
e^{2\log s_g}g_{K}\left(\frac{1}{we^{-F}}s,\frac{1}{we^{-F}}s\right)
=&e^{2F}\frac{s_g^2}{|w|^2}g_K(s,s)=e^{2F}\frac{s_g^2}{|w|^2}\frac{s\wedge\overline{s}}{dvol_g}\\
=&\frac{s_g^2}{|w|^2}\frac{4d\mathrm{Re}\,w\wedge d\mathrm{Im}\,w\wedge dx\wedge dy}{We^vd\xi\wedge\eta\wedge dx\wedge dy}\\
=&\frac{s_g^2}{|w|^2}\frac{4|w|^2d\zeta\wedge d\varphi\wedge dx\wedge dy}{e^vd\zeta\wedge d\varphi\wedge dx\wedge dy}\\
=&4s_g^2e^{-v}.
\end{align*}
We have $s_g^2e^{-v}=e^{-\varpi+O'(\rho^{-\delta_0})}(1+O'(\rho^{-\delta_0}))$. Clearly $s_g^2e^{-v}$ can be continuously extended from $p^{-1}(U)$ to $\overline{p^{-1}(U)}$, which equals $e^{-\varpi}$ on the added $U$. The Lelong number vanishes as the potential function is continuous.
\end{proof}

The subsequent crucial insight stems from \cite{lebrun95}.
Given the fact that the hermitian metric $h$ is Ricci-flat and conformally related to $g$ through $g=s_g^2h$, a straightforward computation shows that the symmetric $(1,1)$-form $\Omega(\cdot,J\cdot)$ associated with the curvature form $\Omega$ of $e^{-2\log s_g}g_{-K}$ is
\begin{equation}\label{eq:curvatureform}
\Omega(\cdot,J\cdot)_{ab}=\frac{s_g}{6}g_{ab}+s_g^{-2}\left(|ds_g|^2g_{ab}-(ds_g)_a(ds_g)_b-(Jds_g)_a(Jds_g)_b\right).
\end{equation}
It follows that $\Omega(\cdot,J\cdot)$ is a positive symmetric $(1,1)$-form on $M$, owing to the evident positivity of 
\begin{align}
&\quad|ds_g|^2g_{ab}-(ds_g)_a(ds_g)_b-(Jds_g)_a(Jds_g)_b\\
=&\quad W^{-1}g-d\xi\otimes d\xi-(Jd\xi)\otimes(Jd\xi)\notag\\
=&\quad d\xi^2+W^{-2}\eta^2+e^v(dx^2+dy^2)-d\xi^2-W^{-2}\eta^2\notag\\
=&\quad e^v(dx^2+dy^2)\notag
\end{align}
as a symmetric $(1,1)$-form.
The positivity of $\Omega$ can be extended across the added divisor $D$ as follows.

\begin{proposition}\label{prop:positivecurrent}
On the compactified complex surface $\overline{M}$, there exists a smooth strictly positive $(1,1)$-form $\vartheta$ such that $\Omega\geq\vartheta$ in the sense of $(1,1)$-currents.
\end{proposition}

\begin{proof}
Due to the positivity of $\Omega$ on $M$, it is sufficient to verify this property near the divisor $D$. Consequently, we only need to consider $\Omega$ locally on each $\overline{p^{-1}(U)}$.

Using (\ref{eq:curvatureform}) along with the expansions for $W$ and $v$, we have that on $p^{-1}(U)$:
\begin{align*}
\Omega(\cdot,J\cdot)&=\frac{\xi}{6}g+\xi^{-2}\left(W^{-1}g-d\xi\otimes d\xi-(Jd\xi)\otimes(Jd\xi)\right)\\
&=\frac{\xi}{6}g+\xi^{-2}e^v(dx^2+dy^2)\\
%&\geq\epsilon\frac{1}{\zeta^3}(d\zeta^2+\eta^2)+\xi^{-2}e^v(dx^2+dy^2)\\
%&\geq\epsilon\frac{1}{\zeta^3}(d\zeta^2+d\varphi^2)+\epsilon\xi^{-2}e^v(dx^2+dy^2)\\
&\geq\epsilon\frac{1}{|w|^2|\log|w||^3}\left(d\mathrm{Re}\,w^2+d\mathrm{Im}\,w^2\right)+\epsilon(dx^2+dy^2).
\end{align*}
From this, it is clear that there is a smooth strictly positive $(1,1)$-form $\vartheta$ on $\overline{M}$ such that $\Omega\geq\vartheta$ as smooth $(1,1)$-forms on $\overline{M}\setminus D$. The following claim finishes the proof.
\begin{claim}
As $(1,1)$-currents, $\Omega\geq\vartheta$ on $\overline{M}$.
\end{claim}
\begin{proof}
Locally the potential function of $\Omega-\vartheta$ is a bounded quasi-plurisubharmonic function, which therefore can be extended across the divisor $D$ (see Theorem 5.24 in \cite{demaillycomplex}). This implies $\Omega\geq\vartheta$ as $(1,1)$-currents. 
\end{proof}
\end{proof}

\begin{theorem}\label{thm:ampleness}
For a Hermitian non-K\"ahler ALF gravitational instanton, the line bundle $-(K_{\overline{M}}+D)$ on the compactified surface $\overline{M}$ is ample. In particular, $\overline{M}$ is an algebraic surface.
\end{theorem}

\begin{proof}
Direct application of Demailly's regularization theorem (Theorem 1.1 in \cite{demailly92}). 
\end{proof}

\begin{proposition}\label{prop:extendaction}
The holomorphic $\mathbb{C}^*$-action induced by $\partial_{t}$ and $J\partial_{t}$ on $M$ extends to $\overline{M}$, which fixes points in $D$.
\end{proposition}
\begin{proof}
Within our compactification coordinate $(\mathrm{Re}\,w,\mathrm{Im}\,w,x,y)$, the action is rotating $\varphi$ and shifting $\zeta$, and where $w=\exp(-\zeta-i\varphi)$. 
\end{proof}

\begin{remark}
The compactification is natural in the sense that we are adding a holomorphic sphere $D$ that is fixed by the extended $\mathbb{C}^*$-action above.
\end{remark}

\subsection{Rational $\text{AF}$ gravitational instantons}
\label{subsec:compactificationrational}

In this subsection, we study Hermitian non-K\"ahler rational $\mathrm{AF}_{\mathfrak{a},\mathfrak{b}}$ gravitational instantons. 
For simplicity reason, we are going to take $\mathfrak{b}$ as $2\pi$ in this subsection.
We represent the equivalence class associated with $(\rho,\alpha,\beta,t)$ within the quotient by $[\rho,\alpha,\beta,t]$.
When $\mathfrak{a}/2\pi={m}/{n}$ is rational, there is the apparent finite covering map
\begin{align}
\Pi:\mathfrak{M}_{0,2\pi n}\to \mathfrak{M}_{\mathfrak{a},2\pi},&&[\rho,\alpha,\beta,t]\to[\rho,\alpha,\beta,t].
\end{align}
This is nothing but the unwrapping operation in (2) of Section \ref{subsec:classification of ends}. We can pull back the end $\mathfrak{M}_{\mathfrak{a},2\pi}\setminus\overline{B_N(0)}$, alongside the Hermitian Ricci-flat metric $h$, the K\"ahler metric $g$, and the complex structure $J$, onto the cover $\mathfrak{M}_{0,2\pi n}\setminus\overline{B_N(0)}$. Our previous treatment for Hermitian non-K\"ahler ALF gravitational instantons in Section \ref{subsec:localansatztogravitational}-\ref{subsec:alfstructure}  remains wholly applicable to the end $\mathfrak{M}_{0,2\pi n}\setminus\overline{B_N(0)}$, with the pulled-back complex structure $\Pi^*J$ and K\"ahler metric $\Pi^*g$. Thus

\begin{proposition}\label{prop:rationala2picompactification}
The end $\mathfrak{M}_{0,2\pi n}\setminus\overline{B_N(0)}$, along with the complex structure $\Pi^*J$, can be compactified at infinity, by introducing a holomorphic sphere $D'$ at infinity. This added curve $D'$ has trivial self-intersection.
\end{proposition}

Our focus is solely on the behavior of the covering space $\mathfrak{M}_{0,2\pi n}\setminus\overline{B_N(0)}$ at infinity. Denote the compactification of the end $\mathfrak{M}_{0,2\pi n}\setminus\overline{B_N(0)}$ as $M'$. 
We now at least have a good understanding about the complex structure of the gravitational instanton on the end. 

\begin{proposition}\label{prop:biholomorphic}
For a Hermitian non-K\"ahler rational $AF_{\mathfrak{a},2\pi}$ gravitational instanton with $\mathfrak{a}/2\pi=m/n$, a suitable end $M\setminus K$ of the gravitational instanton is naturally biholomorphic to the quotient $\mathbb{C}^*\times\mathbb{P}^1/_\sim$ by the following relation
$$\sim:z\times[x:y]\mapsto e^{2\pi i\frac{1}{n}}z\times[e^{2\pi i\frac{m}{n}}x:y].$$
The holomorphic $S^1$-action induced by the extremal vector field $\mathcal{K}$ corresponds to the $S^1$-action on the quotient $\mathbb{C}^*\times\mathbb{P}^1/_\sim$ descended from the rotation on $\mathbb{C}^*\times\mathbb{P}^1$
$$e^{i{t}}\cdot(z\times[x:y])=e^{-i{t}}z\times[x:y].$$
\end{proposition}
The action is chosen to be $e^{i{t}}\cdot(z\times[x:y])=e^{-i{t}}z\times[x:y]$ in order to be compatible with $\mathcal{K}=-\partial_t$.
\begin{proof}
The cover of a suitable end of the gravitational instanton is biholomorphic to $\mathbb{C}^*\times\mathbb{P}^1$. 
Because of Proposition \ref{prop:extendaction}, 
the way we compactify the cover at infinity is natural, in the sense that the added holomorphic sphere $D'$ is fixed by the extended $\mathbb{C}^*$-action. The pulled-back extremal vector field $\Pi^*\mathcal{K}$ corresponds to the rotation on the $\mathbb{C}^*$ factor
$$e^{i{t}}\cdot(z\times[x:y])=e^{-i{t}}z\times[x:y].$$
Using the covering map $\Pi$, we can now  recover the end holomorphically, and the end emerges as the quotient of $\mathbb{C}^*\times\mathbb{P}^1$ by the action
$$z\times[x:y]\mapsto e^{2\pi i\frac{1}{n}}z\times[e^{2\pi i\frac{m}{n}}x:y].$$ 
\end{proof}

\begin{proposition}\label{prop:rationala2pipositivity}
The hermitian metric $e^{-2\log s_g}g_{-K}$ on $-K_M$ restricted to the end $\mathfrak{M}_{\mathfrak{a},2\pi}\setminus\overline{B_N(0)}$ pulled back to the cover $\mathfrak{M}_{0,2\pi n}\setminus\overline{B_N(0)}$ can be extended to a singular hermitian metric on $-(K_{M'}+D')$, with the regularity properties described in Proposition \ref{prop:metricbasicestimate}.
\end{proposition}
\begin{proof}
This is just Proposition \ref{prop:metricbasicestimate} applied to the end $\mathfrak{M}_{0,2\pi n}\setminus\overline{B_N(0)}$.
\end{proof}

\begin{theorem}\label{thm:finalcompactificationrational}
For a Hermitian non-K\"ahler rational ${AF}_{\mathfrak{a},2\pi}$ gravitational instanton $(M,h)$ with $\mathfrak{a}/2\pi=m/n$, the complex surface $M$ can be naturally compactified by adding a holomorphic $D=\mathbb{P}^1$ at infinity to a compact complex algebraic surface $\overline{M}$ with two orbifold points in the added divisor $D$. The $\mathbb{Q}$-line bundle $-(K_{\overline{M}}+D)$ is ample on the compactified surface $\overline{M}$.
\end{theorem}

\begin{proof}
The positivity is clear from how we prove Proposition \ref{prop:positivecurrent} and Theorem \ref{thm:ampleness}.
The compactification $\overline{M}$ clearly has two orbifold points, explicitly given by $\infty\times[0:1]$ and $\infty\times[1:0]$ located in the divisor $D$, identifying the end as the quotient $\mathbb{C}^*\times\mathbb{P}^1/_\sim$ by the action $z\times[x:y]\mapsto e^{2\pi\frac{1}{n}}z\times[e^{2\pi i\frac{m}{n}}x:y]$.
\end{proof}

\begin{remark}
Similarly, the holomorphic $\mathbb{C}^*$-action induced by $\mathcal{K}$ and $J\mathcal{K}$ extends to the compactification $\overline{M}$, fixing the holomorphic sphere $D$. This explains the naturality of the compactification.
\end{remark}

\subsection{Structure of the compactified surfaces} 
\label{subsec:structure}
In the following, the pair $(\overline{M},D)$ always denotes the naturally compactified algebraic surface with the added divisor, arising from a Hermitian non-K\"ahler ALF gravitational instanton or a Hermitian non-K\"ahler rational AF gravitational instanton.

\subsubsection{Hermitian non-K\"ahler ALF gravitational instantons}

The main conclusions of this part are Corollary \ref{cor:alfpositive} and Corollary \ref{cor:ruled}.

\begin{proposition}
The complex surface $\overline{M}$ is rational.
\end{proposition}
\begin{proof}
Generic orbits under the holomorphic $\mathbb{C}^*$-action form a family of holomorphic spheres with nonnegative self-intersections. By performing appropriate blow-ups of points within $\overline{M}\setminus D$, $\overline{M}$ is birational to a ruled surface over $D=\mathbb{P}^1$, thereby affirming the rationality of $\overline{M}$.
\end{proof}

Recall that a rational surface must be either $\mathbb{P}^1\times\mathbb{P}^1$, or have $\mathbb{P}^2$ as a minimal model, or have the Hirzebruch surface $H_k$ with $k\geq2$ as a minimal model.

\begin{corollary}\label{cor:intersection}
The surface $\overline{M}$ contains no rational curve with self-intersection $\leq-2$, except possibly the divisor $D$. Rational curves intersecting $D$ but different from $D$ have nonnegative self-intersections. 
\end{corollary}
\begin{proof}
In the ALF case, the line bundle $-(K_{\overline{M}}+D)$ is ample. A straightforward application of the adjunction formula demonstrates that for a rational curve $C$ not intersecting $D$
$$C^2=-2-K_{\overline{M}}\cdot C=-2-(K_{\overline{M}}+D)\cdot C\geq-1.$$ 
For a rational curve $C$ intersecting $D$ yet distinct from $D$
$$C^2=-2-K_{\overline{M}}\cdot C\geq-1-(K_{\overline{M}}+D)\cdot C\geq0.$$
\end{proof}

\begin{corollary}\label{cor:alfpositive}
Suppose $(M,h)$ is a Hermitian non-K\"ahler ALF gravitational instanton.
If $D$ has nonnegative self-intersection $D^2\geq0$, then $\overline{M}$ is a del Pezzo surface. Moreover, the del Pezzo surface along with the divisor $(\overline{M},D)$, up to biholomorphisms, must be one of the following
\begin{itemize}
\item $(\mathbb{P}^2,D)$, where $D=\mathbb{P}^1$ is a hyperplane with self-intersection 1.
\item $(Bl_p\mathbb{P}^2,D)$, where $D=\mathbb{P}^1$ is a curve with self-intersection 1.
\item $(\mathbb{P}^1\times\mathbb{P}^1,D)$, where $D=0\times\mathbb{P}^1$ is a curve with self-intersection 0.
\end{itemize}
\end{corollary}
\begin{proof}
Given the positivity, we deduce $-K_{\overline{M}} > D$. Intersecting with curves in $\overline{M}$ reveals that $-K_{\overline{M}}$ is ample (Nakai-Moishezon). As a result, $\overline{M}$ is a del Pezzo surface. Del Pezzo surface either has $\mathbb{P}^2$ as a minimal model or is exactly $\mathbb{P}^1\times\mathbb{P}^1$.

In the situation where the minimal model is $\mathbb{P}^2$, the holomorphic $\mathbb{C}^*$-action must originate from an action on $\mathbb{P}^2$. As the added divisor $D$ has nonnegative self-intersection, it cannot be contracted to a point within $\mathbb{P}^2$. Since $D$ is fixed by the $\mathbb{C}^*$-action, its image in $\mathbb{P}^2$ is a fixed curve. 
This leads to that the pairs $(\overline{M},D)$ in this case, up to biholomorphisms, must be $(\mathbb{P}^2,\mathbb{P}^1)$ with a hyperplane $\mathbb{P}^1$, or $(Bl_p\mathbb{P}^2,\mathbb{P}^1)$ with a $\mathbb{P}^1$ of self-intersection 1.

In the other case where $\overline{M}=\mathbb{P}^1\times\mathbb{P}^1$, since $D$ must be a fixed curve under the $\mathbb{C}^*$-action, the only plausible pair is $(\mathbb{P}^1\times\mathbb{P}^1,0\times\mathbb{P}^1)$.
\end{proof}

\begin{corollary}\label{cor:ruled}
Suppose $(M,h)$ is a Hermitian non-K\"ahler ALF gravitational instanton. 
If $D$ has negative self-intersection number $D^2=-k$, then $\overline{M}$ up to biholomorphisms is the Hirzebruch surface $H_{k}=\mathbb{P}(\mathcal{O}\oplus\mathcal{O}(k))$ and $D$ is the curve at infinity $C_\infty$.
\end{corollary}
\begin{proof}
According to our Corollary \ref{cor:intersection}, curves intersecting $D$ must have nonnegative self-intersections. Therefore, when we proceed to contract $(-1)$-curves (excluding $D$ when $D^2=-1$), the rational curve $D$ will remain unaffected during this process. In conclusion
\begin{itemize}
\item  When $k\geq2$, this process results in a minimal model that contains a curve with self-intersection $-k$, which is necessarily the Hirzebruch surface $H_k$, and the divisor $D$ corresponds to the proper transform of $C_\infty$.
\item When $k=1$, performing one more contraction by contracting the $(-1)$-curve $D$ results in a minimal model, which in this case is $\mathbb{P}^2$. Thus, for $k=1$, $\overline{M}$ originates from the blow-up of $H_1$, and the divisor $D$ corresponds to the proper transform of $C_\infty$.
\end{itemize}
Blow-ups are performed exclusively at points away from $C_\infty$ to derive $\overline{M}$ from $H_k$. The holomorphic $\mathbb{C}^*$-action on $\overline{M}$ originates from the standard ruling on $H_k$. Consequently, it can be deduced that $\overline{M}$ is indeed $H_k$, given that any blow-up within $H_k$ that preserves the $\mathbb{C}^*$-action would lead to curves with self-intersections $\leq-1$, which would intersect $D$ in contradiction to Corollary \ref{cor:intersection}.
\end{proof}

\subsubsection{Hermitian non-K\"ahler rational AF gravitational instantons}

The main conclusion of this part is Corollary \ref{cor:af}.

\begin{proposition}\label{prop:complexstructureinfinity}
For a Hermitian non-K\"ahler rational $AF_{\mathfrak{a},2\pi}$ gravitational instanton $(M,h)$ with $\mathfrak{a}/2\pi=m/n$, a suitable end $M\setminus K$ together with the holomorphic $S^1$-action induced by $\mathcal{K}$ is biholomorphic to:
\begin{equation}\label{eq:aftertransformation}
\text{$\mathbb{C}^*\times\mathbb{P}^1$, together with the $S^1$-action $e^{i{t}}\cdot (w\times[u:v])\mapsto e^{-in{t}}w\times[e^{im{t}}u:v]$}.
\end{equation}
Here, we use $w\times[u:v]$ to parametrize $\mathbb{C}^*\times\mathbb{P}^1$. 
\end{proposition}
\begin{proof}
Recall that holomorphically, a suitable end of such a gravitational instanton appears as the quotient  $\mathbb{C}^*\times\mathbb{P}^1/_\sim$ by the relation $\sim:z\times[x:y]\mapsto e^{2\pi i\frac{1}{n}}z\times[e^{2\pi i\frac{m}{n}}x:y]$. 
The $S^1$-action induced by $\mathcal{K}$ is given by the rotation on $z\in\mathbb{C}^*$ on the cover of the end $\mathbb{C}^*\times\mathbb{P}^1$
$$e^{i{t}}\cdot(z\times[x:y])= e^{-i{t}}z\times[x:y].$$
Now, the end of the gravitational instanton ${\mathbb{C}^*\times\mathbb{P}^1}/_{\sim}$ can be identified to the standard $\mathbb{C}^*\times\mathbb{P}^1$ by the following biholomorphism
\begin{align}
\mathbb{C}^*\times\mathbb{P}^1/_\sim \to \mathbb{C}^*\times\mathbb{P}^1 && z\times[x:y]/_\sim\mapsto w\times[u:v]\coloneqq z^n\times [z^{-m}x:y].
\end{align}
Under this identification, the end of the gravitational instanton holomorphically becomes $\mathbb{C}^*\times\mathbb{P}^1$, and the $S^1$-action induced by $\mathcal{K}$ becomes
$$e^{i{t}}\cdot (w\times[u:v])=e^{-in{t}}w\times[e^{im{t}}u:v].$$
\end{proof}

\begin{proposition}
Suppose $(M,h)$ is a Hermitian non-K\"ahler rational ${AF}_{\mathfrak{a},2\pi}$ gravitational instantons with $\mathfrak{a}/2\pi=m/n$. Then $M$ can be smoothly compactified by adding a holomorphic sphere $D^\sharp$ to a smooth compact complex surface $\overline{M}^\sharp$.
\end{proposition}
\begin{proof}
Add $\infty\times\mathbb{P}^1$ to the end $\mathbb{C}^*\times\mathbb{P}^1$ that appears in Proposition \ref{prop:complexstructureinfinity}.
\end{proof}

\begin{remark}
This compactification is unnatural in the sense that the holomorphic sphere $D^\sharp$ at infinity is not fixed by the extended holomorphic $\mathbb{C}^*$-action induced by $\mathcal{K}$ and $J\mathcal{K}$.
\end{remark}

\begin{proposition}\label{prop:unnaturalcompactification}
The unnatural compactification $(\overline{M}^\sharp,D^\sharp)$ up to biholomorphisms can only be
\begin{itemize}
\item $(H_k,\mathbb{P}^1)$, where $H_k=\mathbb{P}(\mathcal{O}\oplus\mathcal{O}(k))$ is the Hirzebruch surface with $k\geq1$ and $\mathbb{P}^1$ is a ruling with trivial self-intersection.
\item $(Bl_pH_k,\mathbb{P}^1)$, where $p\in C_0$ or $C_\infty$ and $\mathbb{P}^1$ is a ruling with trivial self-intersection.
\end{itemize}
\end{proposition}
\begin{proof}
We begin with a claim.
\begin{claim}
The surface $\overline{M}^\sharp$ can be blown up from a Hirzebruch surface $H_k$.
\end{claim}
\begin{proof}
Because of the existence of the rational curve $D^\sharp$ with trivial self-intersection in $\overline{M}^\sharp$, by contracting certain curves inside $M$ if necessary, we will end with a ruled surface. The ruled surface has to be Hirzebruch, as $\overline{M}^\sharp$ carries the holomorphic $S^1$-action induced by $\mathcal{K}$, which obviously does not preserve general rational curves with trivial self-intersection in the class $[D^\sharp]$, as indicated by (\ref{eq:aftertransformation}).
\end{proof}
From the proof of the claim, we can also see that $D^\sharp$ is exactly a ruling $\mathbb{P}^1$ of the Hirzebruch surface $H_k$, which is not affected by blow-ups. 
The holomorphic vector fields $\mathcal{K}$ and $J\mathcal{K}$ together induce a holomorphic $\mathbb{C}^*$-action on $\overline{M}^\sharp$, which descends to the Hirzebruch surface. 

\begin{claim}
The descended holomorphic $\mathbb{C}^*$-action on $H_k$ either:
\begin{itemize}
\item The holomorphic $\mathbb{C}^*$-action has one fixed curve in $H_k\setminus D^\sharp$, which is a ruling $\mathbb{P}^1$.
\item The holomorphic $\mathbb{C}^*$-action has  two fixed points in $H_k\setminus D^\sharp$, lying in a same ruling $\mathbb{P}^1$.
\end{itemize}
\end{claim}
\begin{proof}
This suffices to understand the structure of holomorphic $\mathbb{C}^*$-actions on Hirzebruch surfaces, which can be done easily by contracting the $(-k)$-curve in $H_k$ to the weighted projective space $\mathbb{P}(1,1,k)$. From Proposition \ref{prop:complexstructureinfinity}, the descended holomorphic $\mathbb{C}^*$-action already has two fixed points in the ruling $D^\sharp$. Therefore, there  either is one fixed curve, or two fixed points lying in a same ruling $\mathbb{P}^1$, in $H_k\setminus D^\sharp$. 
\end{proof}

The surface $\overline{M}^\sharp$ is realized as an iterative blow-up of  $H_k$.
To make the $\mathbb{C}^*$-action and $D^\sharp$ survive, blow-ups can only be taken at the fixed points that are not in $D^\sharp$. In both of the cases in the above claim, only one blow-up is allowed, since the positivity of $-(K_{\overline{M}}+D)$ guarantees that there are no $(-2)$-curves inside $M$. The proposition is proved.
\end{proof}

Concluding Proposition \ref{prop:complexstructureinfinity}-\ref{prop:unnaturalcompactification}, we can see that the natural compactification $\overline{M}$ is a toric surface with two orbifold points in $D$. Since $(\overline{M}^\sharp,D^\sharp)$ is either $(H_k,\mathbb{P}^1)$ or $(Bl_pH_k,\mathbb{P}^1)$, the polytope $\mathbf{\Delta}_{\overline{M}}$ of $\overline{M}$ has four or five edges, containing two parallel edges. As $D$ has a neighborhood in $\overline{M}$, holomorphically given by $(\mathbb{C}^*\cup\{\infty\})\times\mathbb{P}^1/_\sim$, where the finite quotient is $\sim:z\times[x:y]\mapsto e^{-2\pi i\frac{1}{n}}z\times[e^{2\pi i\frac{m}{n}}x:y]$, it is straightforward to conclude that in terms of fans $\mathbf{F}_{\overline{M}}$, the toric surface $\overline{M}$ must be described by Figure \ref{possible1} or \ref{possible2}.

\begin{figure}[ht]
\centering
\begin{minipage}{.5\textwidth}
\centering
\begin{tikzpicture}
\path[draw](0,0)--(0,1.5) node[anchor=east]{$(0,1)$};
\path[draw](0,0)--(0,-1.5) node[anchor=east]{$(0,-1)$};
\path[draw](0,0)--(3/4,3/2) node[anchor=west]{$(1,k)$};
\path[draw](0,0)--(-3/2,-0.3) node[anchor=north east]{$(-n,-m)$};
\end{tikzpicture}
\caption{The fan $\mathbf{F}_{\overline{M}}$ if four edges.}
\label{possible1}
\end{minipage}%
\begin{minipage}{.5\textwidth}
\centering
\begin{tikzpicture}
\path[draw](0,0)--(0,3/2) node[anchor=east]{$(0,1)$};
\path[draw](0,0)--(0,-3/2) node[anchor=east]{$(0,-1)$};
\path[draw](0,0)--(3/4,3/2) node[anchor=west]{$(1,k)$};
\path[draw](0,0)--(3/4,3/4) node[anchor=west]{$(1,k-1)$};
\path[draw](0,0)--(-3/2,-0.3) node[anchor=north east]{$(-n,-m)$};
\end{tikzpicture}
\caption{The fan $\mathbf{F}_{\overline{M}}$ if five edges.}
\label{possible2}
\end{minipage}
\end{figure}
In the case that $\mathbf{\Delta}_{\overline{M}}$ has four edges, the ray with the direction $(-n,-m)$ corresponds to the divisor $D$. And the number $k$ corresponds to the index of the Hirzebruch surface $H_k$.
In the case that $\mathbf{\Delta}_{\overline{M}}$ has five edges, the situation is similar. 
With our explicit understanding about the possible structure of the toric surface $\overline{M}$ and the positivity of $-(K_{\overline{M}}+D)$, we can now give a more precise classification of the natural compactification $(\overline{M},D)$.

\begin{corollary}\label{cor:af}
Suppose $(M,h)$ is a rational Hermitian non-K\"ahler ${AF}_{\mathfrak{a},2\pi}$ gravitational instanton with $\mathfrak{a}/2\pi=m/n$. Then the natural compactification $(\overline{M},D)$ can only be
\begin{itemize}
\item $(H_{m,n},D)$, where $H_{m,n}$ is the toric complex surface described by the fan $\mathbf{F}_{m,n}$, and $D$ is the holomorphic sphere described by the ray with the direction $(-n,-m)$ in $\mathbf{F}_{m,n}$. 

\item $(Bl_pH_{m,n},D)$, where $Bl_pH_{m,n}$ is the toric complex surface described by the fan $Bl_p\mathbf{F}_{m,n}$, and $D$ is the holomorphic sphere described by the ray with the direction $(-n,-m)$ in $Bl_p\mathbf{F}_{m,n}$. The surface $Bl_pH_{m,n}$ is simply the blow-up of $H_{n,m}$ at a specific fixed point.
\end{itemize} 
\end{corollary}

\begin{figure}[ht]
\centering
\begin{minipage}{.5\textwidth}
\centering
\begin{tikzpicture}
\path[draw](0,0)--(0,3/2) node[anchor=west]{$(0,1)$};
\path[draw](0,0)--(0,-3/2) node[anchor=west]{$(0,-1)$};
\path[draw](0,0)--(3/2,0) node[anchor=south]{$(1,0)$};
\path[draw](0,0)--(-3/2,-0.3) node[anchor=north east]{$(-n,-m)$};
\end{tikzpicture}
\caption{The fan $\mathbf{F}_{m,n}$ of $H_{m,n}$.}
\label{fanmn}
\end{minipage}%
\begin{minipage}{.5\textwidth}
%%%%%%%%%%%%%%%%%%%%%%%%%%%%%%%%%%%%%%%%%%%%%%%%%%%%%%%%%%%%%%%%%%%%%%%%%%%%%%%%
\begin{comment}
\centering
\begin{tikzpicture}
\path[draw](0,0)--(0,3/2);
\path[draw](0,3/2)--(3/2,3/2);
\path[draw](3/2,3/2)--(1.8,0);
\path[draw](0,0)--(1.8,0);
\node at (2.55,3/4) {slope $-\frac{n}{m}$};
\end{tikzpicture}
\caption{The polytope $\mathbf{\Delta}_{m,n}$ of $H_{m,n}$.}
\label{polytopemn}
\end{comment}
%%%%%%%%%%%%%%%%%%%%%%%%%%%%%%%%%%%%%%%%%%%%%%%%%%%%%%%%%%%%%%%%%%%%%%%%%%%%%%%
\centering
\begin{tikzpicture}
\path[draw](0,0)--(0,3/2) node[anchor=east]{$(0,1)$};
\path[draw](0,0)--(1.414*3/4,1.414*3/4) node[anchor=south west]{$(1,1)$};
\path[draw](0,0)--(0,-3/2)  node[anchor=west]{$(0,-1)$};
\path[draw](0,0)--(3/2,0) node[anchor=north]{$(1,0)$};
\path[draw](0,0)--(-3/2,-0.3) node[anchor=north east]{$(-n,-m)$};
\end{tikzpicture}
\caption{The fan $Bl_p\mathbf{F}_{m,n}$ of $Bl_pH_{m,n}$.}
\label{blfanmn}
\end{minipage}
\end{figure}

%%%%%%%%%%%%%%%%%%%%%%%%%%%%%%%%%%%%%%%%%%%%%%%%%%%%%%%%%%%%%%%%%%%%%%%%%%%%%%%%%
\begin{comment}

  \begin{figure}[h]
\centering
\begin{minipage}{.5\textwidth}
\centering
\begin{tikzpicture}
\path[draw](0,0)--(0,3/2) node[anchor=east]{$(0,1)$};
\path[draw](0,0)--(1.414*3/4,1.414*3/4) node[anchor=south west]{$(1,1)$};
\path[draw](0,0)--(0,-3/2)  node[anchor=west]{$(0,-1)$};
\path[draw](0,0)--(3/2,0) node[anchor=north]{$(1,0)$};
\path[draw](0,0)--(-3/2,-0.3) node[anchor=north east]{$(-n,-m)$};
\end{tikzpicture}
\caption{The fan $Bl_p\mathbf{F}_{m,n}$ of $Bl_pH_{m,n}$.}
\label{blfanmn}
\end{minipage}%
\begin{minipage}{.5\textwidth}
\centering
\begin{tikzpicture}
\path[draw](0,0.7*3/4)--(0,3/2);
\path[draw](0,3/2)--(3/2,3/2);
\path[draw](3/2,3/2)--(1.8,0);
\path[draw](0.7*3/4,0)--(1.8,0);
\path[draw](0,0.7*3/4)--(0.7*3/4,0);
\node at (2.55,3/4) {slope $-\frac{n}{m}$};
\end{tikzpicture}
\caption{The polytope $Bl_p\mathbf{\Delta}_{m,n}$ of $Bl_pH_{m,n}$.}
\label{blpolytopemn}
\end{minipage}
\end{figure}

\end{comment}
%%%%%%%%%%%%%%%%%%%%%%%%%%%%%%%%%%%%%%%%%%%%%%%%%%%%%%%%%%%%%%%%%%%%%%%%%%%%%

\begin{proof}
It suffices to rule out all other possibilities in Figure \ref{possible1} and \ref{possible2}. It will be a direct calculation based on the positivity of $-(K_{\overline{M}}+D)$, and we will focus on the case that $\mathbf{F}_{\overline{M}}$ has four rays, since the case that $\mathbf{F}_{\overline{M}}$ has five rays is similar.

For the fan given by Figure \ref{possible1}, its polytope is
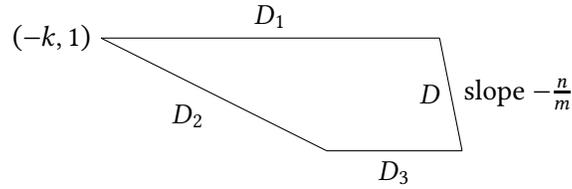
$$\centering
\begin{tikzpicture}
\path[draw](0,0)--node[anchor=north east]{$D_2$}(-4*3/4,2*3/4)node[anchor=east]{$(-k,1)$};
\path[draw](-4*3/4,2*3/4)--node[anchor=south]{$D_1$}(2*3/4,2*3/4);
\path[draw](2*3/4,2*3/4)--node[anchor=east]{$D$}(2.4*3/4,0);
\path[draw](0,0)--node[anchor=north]{$D_3$}(2.4*3/4,0);
\node at (2.55,1*3/4) {slope $-\frac{n}{m}$};
\end{tikzpicture}$$
Direct consideration shows that the divisors $D_1,D_3,D$ are merely $\mathbb{Q}$-Cartier. The intersection numbers are given by 
\begin{align*}
(D_1)^2=\frac{-m+nk}{n},&&(D_3)^2=\frac{m-nk}{n}, && D_1\cdot D=D_3\cdot D=\frac{1}{n}.
\end{align*}
For the toric surface $\overline{M}$, its canonical divisor can be written as $K_{\overline{M}}=-D_1-D_2-D_3-D$.
The positivity of $-(K_{\overline{M}}+D)$ therefore implies
$$-(K_{\overline{M}}+D)\cdot D_1=(D_1+D_2+D_3)\cdot D_1=\frac{-m+kn}{n}+1>0,$$
$$-(K_{\overline{M}}+D)\cdot D_3=(D_1+D_2+D_3)\cdot D_3=1+\frac{m-kn}{n}>0.$$
These force $(-1+k)n<m<(1+k)n$. However, we strictly have $0<m<n$. Consequently, we must have $k=0$ or $1$. If $k=0$, we exactly get the polytope of $\mathbf{F}_{m,n}$. If $k=1$, after performing a linear transform in $SL(2,\mathbb{Z})$ followed by a reflection, it follows that the polytope in this case is equivalent to the polytope of $\mathbf{F}_{n-m,n}$. Therefore, we conclude the case that $\mathbf{F}_{\overline{M}}$ has four edges.
\end{proof}

Finally, we conclude this section with the following theorem, which serves as a corollary of the results presented herein.

\begin{theorem}
Any Hermitian non-K\"ahler ALF/AF gravitational instanton must be holomorphically toric.
\end{theorem}
\begin{proof}
If $\mathcal{K}$ has non-closed orbits, then the gravitational instanton $(M,h)$ automatically is toric.

If $\mathcal{K}$ has closed orbits, then results from this section can be applied. 
The surfaces appear in Corollary \ref{cor:alfpositive}-\ref{cor:af} are all holomorphically toric, so we conclude the closed orbits case.
\end{proof}

By combining Corollary \ref{cor:alfpositive}-\ref{cor:af} and noticing that we always have $-(K_{\overline{M}}+D)>0$, we have completed the proof of Theorem \ref{thm:partialcompactify}. A simple consideration based on the topology of the reversed Taub-NUT, Taub-bolt, rational Kerr, rational Chen-Teo spaces gives their natural compactifications
\begin{itemize}
  \item The complex structure of the reversed Taub-NUT space can be compactified to $(\mathbb{P}^2,\mathbb{P}^1)$, where $D=\mathbb{P}^1$ is a standard hyperplane in $\mathbb{P}^2$.
  \item The complex structures of the Taub-bolt space can be compactified to $(H_1,C_0)$ or $(H_1,C_\infty)$ depending on the choice of orientation.
  \item The complex structure of the rational Kerr spaces can be compactified to $(H_{m,n},D)$.
  \item The complex structure of the rational Chen-Teo spaces can be compactified to $(Bl_pH_{m,n},D)$.
  \end{itemize}

\section{The Calabi-type theorem}
\label{sec:matsushima}

In this section, we prove a Calabi-type theorem for Hermitian non-K\"ahler ALF/AF gravitational instantons. 
Our focus is once again narrowed to gravitational instantons with closed $\mathcal{K}$-orbits. 
Drawing from the conclusions in Section \ref{sec:compactification}, we already know that the natural compactifications are holomorphically toric, in the sense that they possess a holomorphic $T^2$-action that preserves the divisor $D$ (maps points in $D$ into $D$).

Suppose $\bar\partial^{\#}$ is the operator acting on the space of smooth complex-valued functions, which maps each $f$ to the $(1,0)$-vector field $(\bar\partial f)^{\#}$ dual to $\bar\partial f$. Then the holomorphic extremal vector field is just $\bar\partial^{\#}s_g$. The vector field $\bar\partial^{\#}s_g$ extends to the compactified surface $\overline{M}$ as a holomorphic vector field, which fixes every point in $D$, because of Proposition \ref{prop:extendaction}.  Suppose that $\mathfrak{aut}_{\mathcal{K}}(M)$ denotes the vector space of real holomorphic vector fields that commute with $\mathcal{K}$, and $\mathfrak{hamiltoniso}_{\mathcal{K}}(M)$ denotes the vector space of Hamiltonian Killing fields that commute with $\mathcal{K}$ under the K\"ahler metric $g$. 
Now, we are going to prove the following Calabi-type theorem. We will refer to the theorem as a Calabi-type theorem, as part of the proof is indebted to Calabi's contributions in \cite{calabi}.
The statement and proof of this theorem are also partially inspired by \cite{ronan}, where a version of this theorem is proved for complete shrinking gradient K\"ahler-Ricci solitons.

\begin{theorem}[Calabi-type theorem]\label{thm:calabi}
For Hermitian non-K\"ahler ALF gravitational instantons or Hermitian non-K\"ahler rational AF gravitational instantons, there is a direct sum decomposition
\begin{equation}\label{eq:decomposition}
\mathfrak{aut}_{\mathcal{K}}(M)=\mathfrak{hamiltoniso}_{\mathcal{K}}(M)\oplus J\mathfrak{hamiltoniso}_{\mathcal{K}}(M).
\end{equation}
\end{theorem}

For all the possible $M$, the theorem implies that $\mathfrak{hamiltoniso}_{\mathcal{K}}(M)$ contains more than just the 1-dimensional vector space spanned by the Hamiltonian Killing field $\mathcal{K}$.  Hence, this theorem in particular implies that any Hermitian non-K\"ahler ALF/AF gravitational instanton $(M,h)$ with closed $\mathcal{K}$-orbits is holomorphically isometrically toric.

The remainder of this section is dedicated to proving this theorem.
We begin the proof of the Calabi-type theorem.
The directness of the decomposition (\ref{eq:decomposition}) follows in a similar manner as showed in Section 5 of \cite{ronan}. If the decomposition were not a direct sum, it would imply the existence of a Hamiltonian Killing field denoted as $\mathfrak{k}$, where $J\mathfrak{k}$ is a Hamiltonian Killing field simultaneously. Detailed computation shows that under the metric $g$, $\mathfrak{k}$ and $J\mathfrak{k}$ would be parallel (as described on page 314 in \cite{ronan}). 
Importantly, these parallel vector fields have no fixed points, and the vector fields $\mathfrak{k}$ and $J\mathfrak{k}$ cannot both possess closed orbits simultaneously, as they are holomorphic. Therefore, at least one of these two vector fields induces an $\mathbb{R}$-symmetry and flow points on the end of the gravitational instanton to infinity.
However, this leads to a contradiction, as the metric $g$ only has finite volume $\mathrm{Vol}_g(M)$ (attributed to $g=Wd\xi^2+W^{-1}\eta^2+We^v(dx^2+dy^2)$).
This finite volume contrasts with the $\mathbb{R}$-symmetry.

In the following, we will demonstrate that for any real holomorphic vector field $Y_{\text{real}}\in\mathfrak{aut}_{\mathcal{K}}(M)$, there exists a decomposition $Y_{\text{real}}=\mathfrak{k}+J\mathfrak{w}$, where both $\mathfrak{k}$ and $\mathfrak{w}$ belong to $\mathfrak{hamiltoniso}_{\mathcal{K}}(M)$, combining Hamiltonian and Killing properties.
For a given $Y_{\text{real}}\in\mathfrak{aut}_{\mathcal{K}}(M)$, straightforward calculations show that $Y_{\text{real}}$ commutes with $J\mathcal{K}$, and $JY_{\text{real}}$ commutes with $\mathcal{K}$ and $J\mathcal{K}$.
We work under the K\"ahler metric $$g=Wd\xi^2+W^{-1}\eta^2+g_\xi.$$ Introduce $Y\coloneqq Y_{\text{real}}-iJY_{\text{real}}$ as the holomorphic vector field corresponding to the real holomorphic vector field $Y_{\text{real}}$. Our earlier discussions establish that $Y$ commutes with $\bar\partial^{\#}s_g$.

\subsection{Construction of the potential function}\label{subsec:potentialfunction}
In this subsection, we construct a complex-valued potential function $\chi$ for the holomorphic vector field $Y$, such that $\bar\partial\chi=g(Y,\cdot)$.

\begin{lemma}
The holomorphic vector field $Y$ induces a holomorphic vector field on $\Sigma=\mathbb{P}^1$, denoted by $Y_\Sigma$, which also induces a holomorphic $\mathbb{C}^*$-action on $\mathbb{P}^1$.
\end{lemma}

\begin{proof}
This is because $Y$ commutes with $\bar\partial^{\#}s_g$. 
\end{proof}

The $\bar\partial$ operator on $\Sigma=\mathbb{P}^1$ will be denoted by $\bar\partial_{\Sigma}$ in the following.
The $\bar\partial_{\Sigma}$-equation can be solved easily on the Riemann sphere $\mathbb{P}^1$.
\begin{lemma}
Given a $(0,1)$-form $\nu$ on $\Sigma=\mathbb{P}^1$, then
\begin{equation}
f(z)=-\frac{1}{2\pi i}\int_{\mathbb{P}^1}\nu\wedge\frac{dw}{w-z}
\end{equation}
solves the $\bar\partial_{\Sigma}$-equation $\bar\partial_{\Sigma} f=\nu$.
\end{lemma}

Recall that the metric $g_\xi$ on $\Sigma$ is the metric induced by the symplectic reduction at level $\xi$. Locally it is just $g_\xi=We^v(dx^2+dy^2)$.

\begin{lemma}
On each $(\Sigma,g_\xi)$, we can find a potential function $f_\xi$ for the holomorphic vector field $Y_\Sigma$, in the sense that
\begin{equation}
\bar\partial_\Sigma f_{\xi}=g_\xi(Y_\Sigma,\cdot).
\end{equation}
\end{lemma}

\begin{proof}
By picking coordinate $z=x+iy$ suitably on $\mathbb{P}^1$, we can assume that the holomorphic vector field $Y_\Sigma=z\partial_z$.
So $g_\xi(Y_\Sigma,\cdot)=\frac12We^vzd\bar z$. It suffices to solve the $\bar\partial_{\Sigma}$-equation $\bar\partial_\Sigma f_\xi=\frac12We^vzd\bar z$ on $\mathbb{P}^1$.
The above lemma then implies that we can simply take
\begin{equation}\label{eq:fxi}
f_\xi(z)=\frac{1}{4\pi i}\int_{\mathbb{P}^1}We^vw\frac{1}{w-z}dwd\bar w.
\end{equation}
\end{proof}

Given that $Y$ is a holomorphic vector field, the $(0,1)$-form dual to $g(Y,\cdot)$ on the gravitational instanton $M$ is $\bar\partial$-closed. The hermitian metric $e^{-2\log s_g}g_{-K}$ on $-K_M$ has strictly positive curvature. Our calculations from the proof of Proposition \ref{prop:positivecurrent} establish that, on the K\"ahler manifold $(M,g)$, the sum of curvature eigenvalues associated with $e^{-2\log s_g}g_{-K}$ is greater than a small positive value $\epsilon$. Treat the $\bar\partial$-closed $(0,1)$-form $g(Y,\cdot)$ as a section of $\Lambda^{2,1}T_M^*\otimes K^{-1}_M$. We shall apply the standard $L^2$-existence theorem (see Theorem 5.1 of \cite{demailly}) to the modified form $g(Y,\cdot)-\bar\partial f_\xi$.

\begin{lemma}\label{lem:L2}
The $\bar\partial$-closed $(0,1)$-form $g(Y,\cdot)-\bar\partial f_\xi$ is $L^2$ with respect to the hermitian metric on $\Lambda^{2,1}T_M^*\otimes K^{-1}_M$, which is induced by the K\"ahler metric on $\Lambda^{2,1}T_M^*$ and the hermitian metric $e^{-2\log s_g}g_{-K}$ on $-K_M$.
\end{lemma}

\begin{proof}
Notice that the approximating potential function $f_\xi$ is $\partial_{t}$-invariant. Recall that the complex structure is given by $J:d\xi\to -W^{-1}\eta$ and $J:dx\to -dy$. 
The complex structure and the exterior differential on  the base $\Sigma=\mathbb{P}^1$ will be denoted by $J_\Sigma$ and $d_\Sigma$ respectively.
So we have
\begin{align*}
\bar\partial f_\xi&=\frac{1}{2}\left(df_\xi+iJdf_\xi\right)=\frac12\left(d_\Sigma f_\xi+{\partial_\xi}f_\xi d\xi\right)+\frac12iJ\left(d_\Sigma f_\xi+{\partial_\xi}f_\xi d\xi\right)\\
&=\bar\partial_\Sigma f_\xi+\frac12{\partial_\xi}f_\xi d\xi+i\frac12{\partial_\xi}f_\xi Jd\xi.
\end{align*}
Hence,
\begin{align*}
g(Y,\cdot)-\bar\partial f_\xi&=(Wd\xi^2+W^{-1}\eta^2)(Y,\cdot)+g_\xi(Y_\Sigma,\cdot)-\bar\partial f_\xi\\
&=(Wd\xi^2+W^{-1}\eta^2)(Y,\cdot)-\frac12{\partial_\xi}f_\xi d\xi-i\frac12{\partial_\xi}f_\xi Jd\xi\\
&=Wd\xi(Y)d\xi+W^{-1}\eta(Y)\eta-\frac12{\partial_\xi}f_\xi d\xi-i\frac12{\partial_\xi}f_\xi Jd\xi.
\end{align*}
Now notice that $\eta(Y)=-WJd\xi(Y)=-Wd\xi(JY)=-iWd\xi(Y)$, so we have
\begin{align}
g(Y,\cdot)-\bar\partial f_\xi=\underbrace{\eta(Y)(Jd\xi+id\xi)}_{I}-\frac12\underbrace{\left({\partial_\xi}f_\xi d\xi+i{\partial_\xi}f_\xi Jd\xi\right)}_{II}.
\end{align}
We need to handle the two terms separately.

\paragraph{\textbf{\underline{\underline{The $I$-term}}}}
We will consider both the $(\xi,{t},x,y)$ and $(\zeta,\varphi,x,y)$ coordinates. This dual perspective is essential due to the metric is explicitly defined in the $(\xi,{t},x,y)$ coordinate system, while the complex structure behaves favorably in the $(\zeta,\varphi,x,y)$ coordinate.
In the $(\zeta,\varphi,x,y)$ coordinate, the holomorphic vector field $Y$ can be written as $Y=a\partial_\zeta+b\partial_\varphi+c\partial_x+d\partial_y$, where $a,b,c,d$ are complex-valued functions of $\zeta,\varphi,x,y$. Because $Y$ commutes with the real holomorphic vector fields $\mathcal{K}=-\partial_\varphi$ and $\partial_\zeta=J\partial_\varphi$, the coefficients $a$, $b$, $c$, and $d$ are independent of $\zeta$ and $\varphi$. Now recall that we write $\eta=d{t}+\mathcal{X}dx+\mathcal{Y}dy+\mathcal{Z}d\xi$ in the $(\xi,{t},x,y)$ coordinate. 
Through the explicit coordinate transformations (\ref{eq:coorchangezeta}) and (\ref{eq:coorchangephi}), it is straightforward to verify that
\begin{align}\label{eq:etay}
|\eta(Y)|&=|\eta(a\partial_\zeta+b\partial_\varphi+c\partial_x+d\partial_y)|\\
&=\left|-b+c\left(\int^\epsilon_\xi\mathcal{Z}_xd\xi+\mathcal{X}\right)+d\left(\int^\epsilon_\xi\mathcal{Z}_yd\xi+\mathcal{Y}\right)\right|\notag\\
&=\left|-b+c\left(\int^\epsilon_\xi(\mathcal{X}_\xi-W_y)d\xi+\mathcal{X}\right)+d\left(\int^\epsilon_\xi(\mathcal{Y}_\xi+W_x)d\xi+\mathcal{Y}\right)\right|\notag\\
&=\left|-b+c\left(\int^\epsilon_\xi-W_yd\xi+\mathcal{X}(\epsilon,x,y)\right)+d\left(\int^\epsilon_\xi W_xd\xi+\mathcal{Y}(\epsilon,x,y)\right)\right|\notag\\
&\leq C+C\int^\epsilon_\xi|W_y|d\xi+C\int^\epsilon_\xi |W_x|d\xi \notag.
\end{align}
Here, all the derivatives are taken in the $(\xi,{t},x,y)$ coordinate, and for the third equality we used equation (\ref{eq:deta}). Exploiting the asymptotic expansions $W=\rho^2+O'(\rho^{2-\delta_0})$ and $\xi=1/\rho+O'(\rho^{-1-\delta_0})$, we deduce in the $(\xi,{t},x,y)$ coordinate system
\begin{equation}\label{eq:estimatew}
|\partial_\xi^k\partial_x^j\partial_y^lW|\leq C\xi^{-2-k+\delta_0}
\end{equation}
when at least one of $j$ and $l$ is non-zero.
Particularly, the integrals $\int^\epsilon_\xi |W_x|d\xi,\int^\epsilon_\xi|W_y|d\xi$ can be estimated as
\begin{align}\label{eq:integraletay}
\int^\epsilon_\xi |W_x|d\xi,\int^\epsilon_\xi|W_y|d\xi
&\leq C\int^\epsilon_\xi\xi^{-2+\delta_0}d\xi\\
&\leq C\xi^{-1+\delta_0}\notag
\end{align}
on the end of the gravitational instanton.
Equations (\ref{eq:etay}) and (\ref{eq:integraletay}) together imply 
\begin{align}
|\eta(Y)|\leq C\xi^{-1+\delta_0}
\end{align}
on the end of the gravitational instanton.

We can now proceed to estimate the $L^2$-norm of the $I$-term as follows:
\begin{align*}
\int_M|I|^2_{g}e^{-2\log s_g}dvol_g&\leq C\int_{M\setminus K}|d\xi|^2_g|\eta(Y)|^2\xi^{-2} dvol_g+C\\
&\leq C\int_{M\setminus K} W^{-1} \xi^{-4+2\delta_0}dvol_g+C\\
&\leq C\int_{M\setminus K} \xi^{-2+2\delta_0}We^vd\xi\wedge\eta\wedge dx\wedge dy+C\\
&\leq C\int_{M\setminus K} \xi^{-2+2\delta_0}d\xi\wedge\eta\wedge dx\wedge dy+C.
\end{align*}
Given that we have chosen the decay order $\delta_0$ to be close to 1, the inequality $-2+2\delta_0>-1$ holds. Finally, 
\begin{align*}
\int_{M\setminus K} \xi^{-2+2\delta_0}d\xi\wedge\eta\wedge dx\wedge dy\leq C\int^\epsilon_0\xi^{-2+2\delta_0}d\xi\leq C.
\end{align*} 
This concludes the estimation of the $I$-term.

\paragraph{\textbf{\underline{\underline{The $II$-term}}}}
We once again work in the coordinate system $(\xi,{t},x,y)$. Recall that we have equation (\ref{eq:fxi}) $f_\xi(z)=\frac{1}{4\pi i}\int_{\mathbb{P}^1}We^vw\frac{1}{w-z}dwd\bar w$. Because it clearly holds that
$$\left|{\partial_\xi}f_\xi\right|\leq C\int_{\mathbb{P}^1}\left|{\partial_\xi}\left(We^v\right)\right|\frac{|w|}{|w-z|}idwd\bar w,$$
our focus narrows to estimating the term $|\partial_\xi(We^v)|$. 
Again, with the expansions $W=\rho^2+O'(\rho^{2-\delta_0})$, $v=\varpi(x,y)-2\log\rho+O'(\rho^{-\delta_0})$, and $\xi=1/\rho+O'(\rho^{-1-\delta_0})$, it is direct to check $|\partial_\xi(We^v)|\leq C\rho^{1-\delta_0}$.
The $L^2$-norm of the $II$-term can be estimated as
\begin{align*}
\int_M|II|^2_ge^{-2\log s_g}dvol_g&\leq C\int_{M\setminus K}|\partial_\xi f_\xi|^2|d\xi|_g^2\xi^{-2} dvol_g+C\\
&\leq C\int_{M\setminus K}|\partial_\xi f_\xi|^2W^{-1}\xi^{-2} dvol_g+C\\
&\leq C\int_{M\setminus K}\rho^{2-2\delta_0}dvol_g+C\\
&\leq C\int_{M\setminus K} \xi^{-2+2\delta_0}We^vd\xi\wedge\eta\wedge dx\wedge dy+C\\
&\leq C.
\end{align*}
Here, $We^v$ is bounded and because we have chosen $\delta_0$ to be close to 1, $\int_{M\setminus K} \xi^{-2+2\delta_0}d\xi\wedge\eta\wedge dx\wedge dy$ is finite.

Therefore the lemma is proved.
\end{proof}

\begin{proposition}
There exists a potential function $\chi$ for $Y$.
\end{proposition}
\begin{proof}
Direct application of Theorem 5.1 of \cite{demailly}. 
\end{proof}

\subsection{Estimates on the potential function} 
In this subsection we derive some estimates on the potential function $\chi$. 
\begin{lemma}\label{lem:invarianceofchi}
The potential function $\chi$ can be taken as invariant under the Killing field $\partial_{t}$.
\end{lemma}
\begin{proof}
This is clear since the K\"ahler metric $g$, the approximating potential function $f_\xi$, and the holomorphic vector field $Y$ are all $\partial_{t}$-invariant.
\end{proof}

Recalling that in the coordinate system $(\zeta,\varphi,x,y)$, we have $J\partial_\zeta=\partial_\varphi$, and $Y=a\partial_\zeta+b\partial_\varphi+c\partial_x+d\partial_y$, where $a$, $b$, $c$, and $d$ are only functions of $x$ and $y$ and hence bounded, as established in the proof of Lemma \ref{lem:L2}. The metric $g=Wd\xi^2+W^{-1}\eta^2+We^v(dx^2+dy^2)$ with $\eta=d{t}+\mathcal{X}dx+\mathcal{Y}dy+\mathcal{Z}d\xi$.
Applying the explicit coordinate transformations (\ref{eq:coorchangezeta}) and (\ref{eq:coorchangephi}), along with (\ref{eq:vectorrelation}), we arrive at
\begin{align*}
\bar{\partial}{\chi}(\partial_\varphi)=\frac12\left(d\chi+iJd\chi\right)(\partial_\varphi)=-\frac12id\chi(\partial_{\zeta})=\frac12iW^{-1}\partial_\xi\chi.
\end{align*}
Meanwhile, we have
\begin{align*}
\bar{\partial}{\chi}(\partial_\varphi)&=g(Y,\partial_\varphi)\\
&=g(a\partial_\zeta+b\partial_\varphi+c\partial_x+d\partial_y,\partial_\varphi)\\
&=-W^{-1}\left(-b+c\left(\mathcal{X}+\int^\epsilon_\xi \mathcal{Z}_xd\xi\right)+d\left(\mathcal{Y}+\int^\epsilon_\xi \mathcal{Z}_yd\xi\right)\right)\\
&=W^{-1}\left(b-c\left(\mathcal{X}(\epsilon,x,y)-\int^\epsilon_\xi W_yd\xi\right)-d\left(\mathcal{Y}(\epsilon,x,y)+\int^\epsilon_\xi W_xd\xi\right)\right).
\end{align*}
Here, in the third line we applied $\mathcal{Z}_x=\mathcal{X}_\xi-W_y$ and $\mathcal{Z}_y=\mathcal{Y}_\xi+W_x$ again, which comes from (\ref{eq:deta}).
Consequently,
\begin{equation}\label{eq:chiderivative}
\partial_\xi\chi=-2i\left(b+c\left(-\mathcal{X}(\epsilon,x,y)+\int^\epsilon_\xi W_yd\xi\right)+d\left(-\mathcal{Y}(\epsilon,x,y)-\int^\epsilon_\xi W_xd\xi\right)\right).
\end{equation}
We shall proceed to estimate the potential function $\chi$ based on this equation as follows.
\begin{lemma}
The estimate $|\nabla_g^k\chi|_g\leq C$ holds.
\end{lemma}
\begin{proof}
In the coordinate system $(\xi,{t},x,y)$, recalling (\ref{eq:estimatew}), we have established that
$|\partial_\xi^{k}\partial_x^j\partial_y^lW|\leq C\xi^{-2-k+\delta_0}$
when at least one of $j$ or $l$ is non-zero. This yields
$$\int^\epsilon_\xi |\partial_\xi^{k}\partial_x^j\partial_y^lW|d\xi\leq C\xi^{-1-k+\delta_0}$$
when at least one of $j$ or $l$ is non-zero. Moving on from equation (\ref{eq:chiderivative}), we obtain
$$\left|\partial_\xi^k\partial_x^j\partial_y^l\chi\right|\leq C\xi^{-k+\delta_0}$$
for $k\geq1$.
Finally to get the statement in the lemma, we observe that the metric $g=Wd\xi^2+W^{-1}\eta^2+We^v(dx^2+dy^2)$. Thus, $|d\xi|_g=\sqrt{W^{-1}}\sim\xi$ and $|dx|_g=|dy|_g=\sqrt{W^{-1}e^{-v}}\sim C$. Alongside the fact that the potential function $\chi$ is $\partial_{t}$-invariant, it is straightforward to check
$$|\nabla_g^k\chi|_g\leq C\xi^\delta_0$$
for $k\geq1$. This specifically implies that $|\nabla_g^k\chi|_g\leq C$ for all $k\geq0$.
\end{proof}

\subsection{Calabi's argument}

In this subsection, we employ an argument attributed to Calabi \cite{calabi}, involving various integration by parts that require verification. We define the operator $\mathfrak{L}\coloneqq \bar\partial\bar\partial^{\#}$.
A function $f$ is annihilated by $\mathfrak{L}$ if and only if the $(1,0)$-vector field $\bar\partial^{\#}f$ is holomorphic. In particular, for the potential function $\chi$ associated with $Y$, we have $\mathfrak{L}\chi=0$.
The dual operator $\mathfrak{L}^*$ of $\mathfrak{L}$ can be defined using the formula
\begin{equation}
\int_{M}\psi\overline{\mathfrak{L}^*\phi}dvol_g\coloneqq\int_{M}\mathfrak{L}\psi\overline{\phi}dvol_g
\end{equation}
for any smooth $\psi$ with compact support and any smooth $\phi$.
This leads us to the introduction of a fourth-order differential operator $\mathfrak{D}\coloneqq\mathfrak{L}^*\mathfrak{L}$. We then define the complex conjugate of the operator $\overline{\mathfrak{D}}$ as $\overline{\mathfrak{D}}f\coloneqq \overline{{\mathfrak{D}}\overline{f}}$.
\begin{proposition}\label{prop:commute}
The operators $\mathfrak{D}$ and $\overline{\mathfrak{D}}$ commute.
\end{proposition}
\begin{proof}
This is essentially Proposition 3.1 and its corollary in \cite{calabi}. 
\end{proof}

Next, we proceed with some local calculations.

\begin{lemma}\label{lem:localcomputation}
In  local holomorphic coordinate, we have
$\left(\overline{\mathfrak{D}}-\mathfrak{D}\right)f={s_g}^{j}f_j-{s_{g}}_jf^j= g^{j\bar{k}}\partial_{\bar{k}}s_g\partial_{j}f-g^{j\bar{k}}\partial_{\bar{k}}f\partial_{j}s_g$, and $\mathfrak{L}f=f^{jk}g_{k\bar{l}}\partial_{j}\otimes d\bar{z}^l$.
\end{lemma}
\begin{proof}
Refer to equation (2.3') in \cite{calabi} for the derivation of the first equality. The second equality stems from the definition $\mathfrak{L}=\bar\partial\bar\partial^{\#}$.
\end{proof}

\begin{proposition}\label{prop:decomposition}
If $\mathfrak{D}\chi=\overline{\mathfrak{D}}\chi=0$, then there is a decomposition for $Y_{real}=\mathrm{Re}\,\bar\partial^{\#}\chi$
\begin{equation}
Y_{real}=\mathfrak{k}+J\mathfrak{w},
\end{equation}
with $\mathfrak{k},\mathfrak{w}\in\mathfrak{hamiltoniso}_{\mathcal{K}}(M)$.
\end{proposition}
\begin{proof}
The two equations $\mathfrak{D}\chi=\overline{\mathfrak{D}}\chi=0$ together imply that $\mathfrak{D}(\mathrm{Re}\,\chi)+i\mathfrak{D}(\mathrm{Im}\,\chi)=0$ and $\mathfrak{D}(\mathrm{Re}\,\chi)-i\mathfrak{D}(\mathrm{Im}\,\chi)=0$. This yields $\mathfrak{D}(\mathrm{Re}\,\chi)=\mathfrak{D}(\mathrm{Im}\,\chi)=0$. For brevity, let us denote $\mathrm{Re}\,\chi$ as $\chi_R$ and $\mathrm{Im}\,\chi$ as $\chi_I$. We assert that $\chi_R$ and $\chi_I$ also serve as potential functions for certain real holomorphic vector fields, and importantly, they are real-valued potential functions.

\begin{claim}
$\bar\partial^{\#}(\chi_R)$ and $\bar\partial^{\#}(\chi_I)$ are holomorphic as  $(1,0)$-vector fields.
\end{claim}
\begin{proof}
Given the similarity with the proof for the holomorphicity of $\bar\partial^{\#}(\chi_I)$, we will focus on proving the holomorphicity of $\bar\partial^{\#}(\chi_R)$. Our goal is to show that $\int_M\left|\mathfrak{L}(\chi_R)\right|^2dvol_g=0$.

Choose a cut-off function $\varphi_\varepsilon(\xi)$ as follows:
let $\varphi_\varepsilon(x)=1$ when $x>\varepsilon$ and $\varphi_\varepsilon(x)=0$ when $x<\frac12\varepsilon$, then smoothly extend $\varphi_\varepsilon$ to $\frac12\varepsilon\leq x\leq\varepsilon$ such that $\left|\frac{d}{dx}\varphi_\varepsilon\right|<\frac{100}{\varepsilon}$ and $\left|\frac{d^2}{dx^2}\varphi_\varepsilon\right|<\frac{100}{\varepsilon^2}$. 
The cut-off function $\varphi_\varepsilon(\xi)$ then can be defined on the end of the gravitational instanton, which  satisfies $|\nabla_g\varphi_\varepsilon|_g\leq C$ and $|\nabla^2_g\varphi_\varepsilon|_g\leq C$, for some $C$ does not depend on $\varepsilon$.
In general for any smooth function $f$, we have the estimate $|\bar\partial^{\#}f|_g=|\bar\partial f|_g\leq|\nabla_g f|_g$ and $|\mathfrak{L}(f)|_g=|\bar\partial\bar\partial^{\#}f|_g\leq |\nabla_g^2f|_g$. Hence,
\begin{align*}
& \left|\int_M\mathfrak{L}(\varphi_{\varepsilon}\chi_R)\overline{\mathfrak{L}(\chi_R)}dvol_g-\int_M\varphi_{\varepsilon}\mathfrak{L}(\chi_R)\overline{\mathfrak{L}(\chi_R)}dvol_g\right|\\
\leq& C\left(\int_M|\nabla_g\varphi_\varepsilon|_g|\nabla_g\chi_R|_g|\mathfrak{L}(\chi_R)|_gdvol_g+\int_M|\nabla_g^2\varphi_\varepsilon|_g|\chi_R||\mathfrak{L}(\chi_R)|_gdvol_g\right)\\
\leq& C\left(\int_{\frac12\varepsilon<\xi<\varepsilon}|\nabla_g\varphi_\varepsilon|_g|\nabla_g\chi_R|_g\left|\nabla^2_g\chi_R\right|_gdvol_g+\int_{\frac12\varepsilon<\xi<\varepsilon}|\nabla_g^2\varphi_\varepsilon|_g|\chi_R|\left|\nabla^2_g\chi_R\right|_gdvol_g\right)\\
\leq& C\int_{\frac12\varepsilon<\xi<\varepsilon}dvol_g\\
\leq & C\varepsilon.
\end{align*}
Therefore, 
$$\int_M\left|\mathfrak{L}(\chi_R)\right|^2dvol_g=\lim_{\varepsilon\to0}\int_M\mathfrak{L}(\varphi_{\varepsilon}\chi_R)\overline{\mathfrak{L}(\chi_R)}dvol_g=\lim_{\varepsilon\to0}\int_M\varphi_{\varepsilon}\chi_R\overline{\left(\mathfrak{L}^*\mathfrak{L}(\chi_R)\right)}dvol_g=0.$$
The claim is proved.
\end{proof}

Recall that $\chi$ is chosen as $\partial_{t}$-invariant. As $\chi=\chi_R+i\chi_I$, the real part $\chi_R$ and imaginary part $\chi_I$ are also $\partial_{t}$-invariant. So both the holomorphic vector fields $\bar\partial^{\#}\chi_R$ and $\bar\partial^{\#}\chi_I$ commute with $\partial_{t}$.
By writing $\bar\partial^{\#}(\chi_R)$ as $X-iJX$, we find that $\bar\partial(\chi_R)=g(X-iJX,\cdot)=\frac12\left(d\chi_R-iJd\chi_R\right)$. By comparing the real and imaginary parts, we get $g(X,\cdot)=\frac12d\chi_R$. Thus, the real holomorphic vector field $JX$ is Hamiltonian.
Since $JX$ is both a real holomorphic vector field and a Hamiltonian vector field, it is also a Killing vector field. This shows that $\mathrm{Im}\,\bar\partial^{\#}\chi_R$ is a Hamiltonian Killing vector field.

Similarly, starting with that $\bar\partial^{\#}\chi_I$ is holomorphic, replacing every $\chi_R$ by $\chi_I$ in the previous paragraph, we can conclude that $\mathrm{Im}\,\bar\partial^{\#}\chi_I$ is Hamiltonian Killing too. 
Since both $\bar\partial^{\#}\chi_R$ and $\bar\partial^{\#}\chi_I$ are holomorphic vector fields, we have: 
\begin{align*}
J\mathrm{Im}\,\bar\partial^{\#}\chi_R&=\mathrm{Re}\,\bar\partial^{\#}\chi_R,\\
J\mathrm{Im}\,\bar\partial^{\#}\chi_I&=\mathrm{Re}\,\bar\partial^{\#}\chi_I.
\end{align*}
Therefore, 
\begin{align*}
Y=\bar\partial^{\#}\chi=\bar\partial^{\#}\chi_R+i\bar\partial^{\#}\chi_I&=\left(\mathrm{Re}\,\bar\partial^{\#}\chi_R-\mathrm{Im}\,\bar\partial^{\#}\chi_I\right)+i\left(\mathrm{Im}\,\bar\partial^{\#}\chi_R+\mathrm{Re}\,\bar\partial^{\#}\chi_I\right)\\
&=\left(J\mathrm{Im}\,\bar\partial^{\#}\chi_R-\mathrm{Im}\,\bar\partial^{\#}\chi_I\right)+i\left(\mathrm{Im}\,\bar\partial^{\#}\chi_R+J\mathrm{Im}\,\bar\partial^{\#}\chi_I\right).
\end{align*}
Hence, with $\mathfrak{k}\coloneqq-\mathrm{Im}\,\bar\partial^{\#}(\chi_I)$ and $\mathfrak{w}\coloneqq\mathrm{Im}\,\bar\partial^{\#}(\chi_R)$, we have $Y_{real}=J\mathrm{Im}\,\bar\partial^{\#}\chi_R-\mathrm{Im}\,\bar\partial^{\#}\chi_I=\mathfrak{k}+J\mathfrak{w}$. The two Hamiltonian Killing fields $\mathfrak{k},\mathfrak{w}\in\mathfrak{hamiltoniso}_{\mathcal{K}}(M)$ because $\bar\partial^{\#}\chi_R$ and $\bar\partial^{\#}\chi_I$ commute with $\mathcal{K}=\partial_{t}$.
\end{proof}

Since we already have $\mathfrak{D}\chi=0$, to complete the proof of the Calabi-type theorem, it remains to show that $\overline{\mathfrak{D}}\chi=0$.
From the commutativity between the operators $\mathfrak{D}$ and $\overline{\mathfrak{D}}$, we have ${\mathfrak{D}}\overline{\mathfrak{D}}\chi=0$. As a $(1,0)$-vector field, $\bar\partial^{\#}(\overline{\mathfrak{D}}\chi)$ is holomorphic. %There must be complex constants $a$ and $b$ such that $\bar\partial^{\#}(\overline{\mathfrak{D}}\chi)=a\bar\partial^{\#}s_g+b\bar\partial^{\#}\chi$ since the holomorphic automorphism group is $(\mathbb{C}^*)^2$. Hence, $\bar\partial\left(\overline{\mathfrak{D}}\chi-as_g-b\chi\right)=0$ and $u\coloneqq \overline{\mathfrak{D}}\chi-as_g-b\chi$ is a holomorphic function. 

\begin{lemma}
$\bar\partial(\overline{\mathfrak{D}}\chi)=0$, therefore $\overline{\mathfrak{D}}\chi$ is a holomorphic function.
\end{lemma}
\begin{proof}
From the equation $\mathfrak{D}\chi=0$, the following equality holds:
\begin{align*}
\overline{\mathfrak{D}}\chi&=\overline{\mathfrak{D}}\chi-{\mathfrak{D}}\chi=(\overline{\mathfrak{D}}-{\mathfrak{D}})\chi={s_g}^{j}\chi_j-{s_{g}}_j\chi^j.
\end{align*}
Here the third equality follows from Lemma \ref{lem:localcomputation}.
Moreover, we have ${s_g}^{jk}=\chi^{jk}=0$ since $\mathfrak{L}s_g=\mathfrak{L}\chi=0$.
Hence $g^{j\bar{k}}\partial_{\bar{k}}(\overline{\mathfrak{D}}\chi)=g^{j\bar{k}}\partial_{\bar{k}}\left({s_g}^{l}\chi_l-{s_{g}}_l\chi^l\right)={s_g}^{l}\chi_l^j-{s_g}^{j}_{l}\chi^l$. 
Now notice that the vector field $\bar\partial^{\#}(\overline{\mathfrak{D}}\chi)$ is exactly given by $g^{j\bar{k}}\partial_{\bar{k}}(\overline{\mathfrak{D}}\chi)\partial_{j}$.
It follows that 
$$\bar\partial^{\#}(\overline{\mathfrak{D}}\chi)=g^{j\bar{k}}\partial_{\bar{k}}\left({s_g}^{l}\chi_l-{s_g}_{l}\chi^l\right)\partial_j=\left({s_g}^{l}\chi_l^j-{s_g}^{j}_{l}\chi^l\right)\partial_{j}=\left[\bar\partial^{\#}s_g,\bar\partial^{\#}\chi\right]=0.$$
The second to last equality follows from that $\bar\partial^{\#}s_g={s_g}^{j}\partial_j$ and $\bar\partial^{\#}\chi=\chi^j\partial_j$.
The last equality follows from that $\bar\partial^{\#}s_g$ and $\bar\partial^{\#}\chi$ are commutative since $\chi$ is $\partial_{t}$-invariant.
Therefore $\bar\partial(\overline{\mathfrak{D}}\chi)=0$ and $\overline{\mathfrak{D}}\chi$ is a holomorphic function.
\end{proof}

\begin{proposition}
$\overline{\mathfrak{D}}\chi=0$.
\end{proposition}
\begin{proof}
We need to prove the holomorphic function $\overline{\mathfrak{D}}\chi=0$.
Writing out the operators $\mathfrak{L}$ and $\mathfrak{D}$ locally in a holomorphic coordinate, for any function $f$, we have
\begin{align*}
\mathfrak{L}(f)&=\partial_{\bar{j}}\left(g^{k\bar{l}}\partial_{\bar{l}}f\right)\partial_{k}\otimes dz^{\bar{j}}={f^k}_{\bar{j}}\partial_{k}\otimes dz^{\bar{j}},\\
\mathfrak{D}(f)&=\mathfrak{L}^*\left({f^k}_{\bar{j}}\partial_{k}\otimes dz^{\bar{j}}\right)=g^{p\bar{q}}\partial_{p}\left(g^{r\bar{j}}\partial_{r}\left(g_{k\bar{q}}{f^k}_{\bar{j}}\right)\right)={f_{\bar{q}\bar{j}}}^{\bar{j}\bar{q}}.
\end{align*}
See the calculation in \cite{calabi} for the derivation of these formulae. In particular, $|\overline{\mathfrak{D}}\chi|=|{\overline{\chi}_{\bar{q}\bar{j}}}^{\bar{j}\bar{q}}|\leq |\nabla_g^4\chi|_g\leq C$. The holomorphic function $\overline{\mathfrak{D}}\chi$ now is bounded. Thus $\overline{\mathfrak{D}}\chi\equiv c$ for some constant $c$.

Next we prove that $c=0$. This is because constant functions are orthogonal to the image of $\overline{\mathfrak{D}}$: for any constant $c$,
$$\int_Mc\overline{\mathfrak{D}}\chi dvol_g=\int_M\mathfrak{L}(c)\overline{\mathfrak{L}(\overline\chi)}dvol_g=0.$$
The integration by parts is justified by the fact that $M$ has finite volume $\mathrm{Vol}_g(M)$ under $g$ and our previous estimate $|\nabla^k_g\chi|_g\leq C$. This forces $c=0$.
\end{proof}

As we have shown that $\overline{\mathfrak{D}}\chi=0$, with Proposition \ref{prop:decomposition}, this completes the proof of Theorem \ref{thm:calabi}.
Now finally we have proved the main theorem for ALF/AF gravitational instantons of our paper:
\begin{theorem}
Any Hermitian non-K\"ahler ALF/AF gravitational instanton must be holomorphically isometrically toric.
\end{theorem}

To conclude the classification for Hermitian non-K\"ahler ALF/AF gravitational instantons $(M,h)$,  the classification results in the toric setting proved by Biquard-Gauduchon \cite{biquardgauduchon} can be applied (with slight modifications). 
Notice that the holomorphic isometric $T^2$-action we produced preserves the moment map $\xi$ and $\eta$ in $g=Wd\xi^2+W^{-1}\eta^2+g_\xi$. %We can take $\mathcal{K}$ as $\partial_t$ in the definition of ALF/AF, as we explained in the end of Section \ref{subsec:localansatztogravitational}. 
This guarantees that the toric assumption in \cite{biquardgauduchon} is satisfied.
Although the decay rate $\delta_0$ a prior is less than 1, it can be verified that the asymptotic expansion $\lambda^{1/3}=1/\rho+O'(\rho^{-1-\delta_0})$ and the toricness assumption allow us to apply the proof used by Biquard-Gauduchon with minor modifications.  The main difference lies in the need to change the decay of the error terms in the expansions in the Section 4.3 of \cite{biquardgauduchon}.
The final classification result we get is
\begin{theorem}\label{thm:classificationresult}
Any Hermitian non-K\"ahler ALF or AF gravitational instanton must be one of the following:
\begin{itemize}
\item the Kerr metrics;
\item the Chen-Teo metrics;
\item the reversed Taub-NUT metric, i.e. the Taub-NUT metric with orientation opposite to the hyperk\"ahler orientation;
\item the Taub-bolt metric with both orientations.
\end{itemize}
\end{theorem}

Combining with the classification of ends Theorem \ref{thm:classification of ends}, we conclude the following classification of Hermitian non-K\"ahler gravitational instantons.
\begin{theorem}\label{thm:classification of faster than quadratic}
A Hermitian non-K\"ahler gravitational instanton can only be one of the following:
\begin{itemize}
  \item ALE with $\Gamma\subset U(2)$, which corresponds to a special Bach-flat K\"ahler orbifold in the terminology of \cite{me}.
  \item ALF, in which case it must be the reversed Taub-NUT metric or the Taub-bolt metric with both orientations.
  \item AF, in which case it must be the Kerr metrics or the Chen-Teo metrics. 
\end{itemize}  
\end{theorem}

\bibliographystyle{plain}
\bibliography{reference}

\begin{thebibliography}{10}

\bibitem{conjecture}
S.~Aksteiner and L.~Andersson.
\newblock {Gravitational instantons and special geometry}.
\newblock {\em arXiv:2112.11863}, 2021.
\newblock To appear in Journal of Differential Geometry.

\bibitem{and23}
S.~Aksteiner, L.~Andersson, M.~Dahl, G.~Nilsson, and W.~Simon.
\newblock {Gravitational instantons with $S^1$ symmetry}.
\newblock {\em arXiv:2306.14567}, 2023.

\bibitem{ah88}
M.~Atiyah and N.~Hitchin.
\newblock {\em {The Geometry and Dynamics of Magnetic Monopoles}}.
\newblock Princeton University Press, Princeton, 1988.

\bibitem{besse}
A.~Besse.
\newblock {\em {Einstein manifolds}}.
\newblock Classics in Mathematics. Springer-Verlag Berlin, Heidelberg, 1987.

\bibitem{biquardgauduchonb}
O.~Biquard and P.~Gauduchon.
\newblock {About a Family of ALF Instantons with Conical Singularities}.
\newblock {\em Symmetry, Integrability and Geometry: Methods and Applications (SIGMA)}, 19(e79), 2023.

\bibitem{biquardgauduchon}
O.~Biquard and P.~Gauduchon.
\newblock {On Toric Hermitian ALF Gravitational Instantons}.
\newblock {\em Communications in Mathematical Physics}, 339(1):389--442, 2023.

\bibitem{claude}
O.~Biquard, P.~Gauduchon, and C.~LeBrun.
\newblock {Gravitational Instantons, Weyl Curvature, and Conformally Kaehler Geometry}.
\newblock {\em arXiv:2310.14387}, 2023.

\bibitem{biquardminerbe}
O.~Biquard and V.~Minerbe.
\newblock {A Kummer Construction for Gravitational Instantons}.
\newblock {\em Communications in Mathematical Physics}, 308(3):773--794, 2011.

\bibitem{biquardozuch}
O.~Biquard and T.~Ozuch.
\newblock {Instability of conformally K\"ahler, Einstein metrics}.
\newblock {\em arXiv:2310.10109}, 2023.

\bibitem{calabi}
E.~Calabi.
\newblock {\em {Extremal K\"ahler Metrics II}}.
\newblock Differential Geometry and Complex Analysis. Springer-Verlag, Heidelberg, 1985.

\bibitem{cfg}
J.~Cheeger, K.~Fukaya, and M.~Gromov.
\newblock {Nilpotent structures and invariant metrics on collapsed manifolds}.
\newblock {\em Journal of American Mathematical Society}, 5(2):327--372, 1992.

\bibitem{cheegertian}
J.~Cheeger and G.~Tian.
\newblock {Curvature and injectivity radius estimates for Einstein 4-manifolds}.
\newblock {\em Journal of the American Mathematical Society}, 19(2):487--525, 2006.

\bibitem{cc17}
G.~Chen and X.~Chen.
\newblock {Gravitational instantons with faster than quadratic curvature decay (II)}.
\newblock {\em Journal f\"{u}r die reine und angewandte Mathematik}, 2019(756):259--284, 2017.

\bibitem{cc}
G.~Chen and X.~Chen.
\newblock {Gravitational instantons with faster than quadratic curvature decay (I)}.
\newblock {\em Acta Mathematica}, 227(2):263--307, 2021.

\bibitem{cc21}
G.~Chen and X.~Chen.
\newblock {Gravitational instantons with faster than quadratic curvature decay (III)}.
\newblock {\em Mathematische Annalen}, 380(1-2):687--717, 2021.

\bibitem{cv21}
G.~Chen and J.~Viaclovsky.
\newblock {Gravitational instantons with quadratic volume growth}.
\newblock {\em arXiv:2110.06498}, 2021.
\newblock To appear in Journal of the London Mathematical Society.

\bibitem{CVZ}
G.~Chen, J.~Viaclovsky, and R.~Zhang.
\newblock {Torelli-type theorems for gravitational instantons with quadratic volume growth}.
\newblock {\em Duke Mathematical Journal}, 173(2):227--275, 2024.

\bibitem{cxx}
X.~Chen, C.~LeBrun, and B.~Weber.
\newblock {On conformally K\"ahler Einstein manifolds}.
\newblock {\em Journal of the American Mathematical Society}, 21(4):1137--1168, 2008.

\bibitem{liyu}
X.~Chen and Y.~Li.
\newblock {On the geometry of asymptotically flat manifolds}.
\newblock {\em Geometry \& Topology}, 25(5):2469--2572, 2021.

\bibitem{chenteo}
Y.~Chen and E.~Teo.
\newblock {A new AF gravitational instanton}.
\newblock {\em Physics Letters B}, 703(3):359--362, 2011.

\bibitem{chengyau}
S.~Cheng and S.~Yau.
\newblock {Differential equations on Riemannian manifolds and their geometric applications}.
\newblock {\em Communications on Pure and Applied Mathematics}, 28(3):333--354, 1975.

\bibitem{ch05}
S.~Cherkis and N.~Hitchin.
\newblock {Gravitational Instantons of Type $D_k$}.
\newblock {\em Communications in Mathematical Physics}, 260(2):299--317, 2005.

\bibitem{ck99}
S.~Cherkis and A.~Kapustin.
\newblock {Singular Monopoles and Gravitational Instantons}.
\newblock {\em Communications in Mathematical Physics}, 203(3):713--728, 1999.

\bibitem{ck02}
S.~Cherkis and A.~Kapustin.
\newblock {Hyper-K\"ahler metrics from periodic monopoles}.
\newblock {\em Physical Review D}, 65(8):084015, 2002.

\bibitem{cjl22}
T.~Collins, A.~Jacob, and Y.~Lin.
\newblock {The Torelli theorem for $ALH^*$ gravitational instantons}.
\newblock {\em Forum of Mathematics, Sigma}, 10(79), 2022.

\bibitem{ronan}
R.~Conlon, A.~Deruelle, and S.~Sun.
\newblock {Classification results for expanding and shrinking gradient K\"ahler-Ricci solitons}.
\newblock {\em Geometry \& Topology}, 28(1):267--351, 2024.

\bibitem{demailly92}
J.~Demailly.
\newblock {Regularization of closed positive currents and intersection theory}.
\newblock {\em Journal of Algebraic Geometry}, 1(3):361--409, 1992.

\bibitem{demailly}
J.~Demailly.
\newblock {\em {Analytic Methods in Algebraic Geometry}}.
\newblock International Press, Boston, 2012.

\bibitem{demaillycomplex}
J.~Demailly.
\newblock {\em {Complex Analytic and Differential Geometry}}.
\newblock Online book, 2012.
\newblock Version of Thursday June 21, 2012.

\bibitem{derdzinski}
A.~Derdzi\'{n}ski.
\newblock {Self-dual K\"ahler manifolds and Einstein manifolds of dimension four}.
\newblock {\em Compositio Mathematica}, 49(3):405--433, 1983.

\bibitem{fukaya}
K.~Fukaya.
\newblock {Collapsing of Riemannian manifolds and eigenvalues of Laplace operator}.
\newblock {\em Inventiones Mathematicae}, 87(3):517--547, 1987.

\bibitem{gibbons}
G.~Gibbons and M.~Perry.
\newblock {New gravitational instantons and their interactions}.
\newblock {\em Physical Review D}, 22(2):313--321, 1980.

\bibitem{gromov}
M.~Gromov.
\newblock {\em {Metric Structures for Riemannian and Non-Riemannian Spaces}}.
\newblock Modern Birkh\"{a}user Classics. Birkh\"{a}user, Boston, 2007.

\bibitem{hei12}
H.~Hein.
\newblock {Gravitational instantons from rational elliptic surfaces}.
\newblock {\em Journal of the American Mathematical Society}, 25(2):355--393, 2012.

\bibitem{hsvz22}
H.~Hein, S.~Sun, J.~Viaclovsky, and R.~Zhang.
\newblock {Nilpotent structures and collapsing Ricci-flat metrics on the K3 surface}.
\newblock {\em Journal of the American Mathematical Society}, 35(1):123--209, 2022.

\bibitem{yamada}
M.~Khuri, M.~Reiris, G.~Weinstein, and S.~Yamada.
\newblock {Gravitational Solitons and Complete Ricci Flat Riemannian Manifolds of Infinite Topological Type}.
\newblock {\em arXiv:2204.08048}, 2022.

\bibitem{kro89b}
P.~Kroheimer.
\newblock {A Torelli-type theorem for gravitational instantons}.
\newblock {\em Journal of Differential Geometry}, 29(3):685--697, 1989.

\bibitem{kro89a}
P.~Kroheimer.
\newblock {The construction of ALE spaces as hyper-K\"ahler quotients}.
\newblock {\em Journal of Differential Geometry}, 29(3):665--683, 1989.

\bibitem{lebrun}
C.~LeBrun.
\newblock {Explicit self-dual metrics on $\mathbb{CP}_2\#\ldots\#\mathbb{CP}_2$}.
\newblock {\em Journal of Differential Geometry}, 34(1):223--253, 1991.

\bibitem{lebrun95}
C.~LeBrun.
\newblock {Einstein Metrics on Complex Surfaces}.
\newblock In {\em Geometry and Physics}, pages 167--176. CRC Press, 1995.

\bibitem{lebrun12}
C.~LeBrun.
\newblock {On Einstein, Hermitian 4-manifolds}.
\newblock {\em Journal of Differential Geometry}, 90(2):277--302, 2012.

\bibitem{lebrun20}
C.~LeBrun.
\newblock {Bach-Flat K\"ahler Surfaces}.
\newblock {\em Journal of Geometric Analysis}, 30(3):2491--2514, 2020.

\bibitem{leelin}
T.~Lee and Y.~Lin.
\newblock {Period domains for gravitational instantons}.
\newblock {\em arXiv:2208.13083}, 2022.

\bibitem{me}
M.~Li.
\newblock {On 4-dimensional Ricci-flat ALE manifolds}.
\newblock {\em arXiv:2304.01609}, 2023.

\bibitem{ls}
M.~Li and S.~Sun.
\newblock {On asymptotically flat 4-manifolds}.
\newblock {\em arXiv:2410.11168}, 2024.

\bibitem{minerbe2010}
V.~Minerbe.
\newblock {On the asymptotic geometry of gravitational instantons}.
\newblock {\em Annales scientifiques de l'\'Ecole Normale Sup\'erieure}, 43(6):883--924, 2010.

\bibitem{minerbe}
V.~Minerbe.
\newblock {Rigidity for Multi-Taub-NUT metrics}.
\newblock {\em Journal f\"{u}r die reine und angewandte Mathematik}, 2011(656):47--58, 2011.

\bibitem{naber}
A.~Naber and G.~Tian.
\newblock {Geometric Structures of Collapsing Riemannian Manifolds I}.
\newblock {\em arXiv:0804.2275}, 2008.

\bibitem{page}
D.~Page.
\newblock {A compact rotating gravitational instanton}.
\newblock {\em Physics Letters B}, 79(3):253--238, 1978.

\bibitem{petrunin}
A.~Petrunin and W.~Tuschemann.
\newblock {Asymptotical flatness and cone structure at infinity}.
\newblock {\em Mathematische Annalen}, 321(4):775--788, 2001.

\bibitem{rong}
X.~Rong.
\newblock {Convergence and collapsing theorems in Riemannian geometry}.
\newblock In {\em Handbook of geometric analysis}, pages 193--299. International Press, Somerville, 2010.

\bibitem{sz}
S.~Sun and R.~Zhang.
\newblock {Collapsing geometry of hyperk\"ahler 4-manifolds and applications}.
\newblock {\em arXiv:2108.12991}, 2021.
\newblock To appear in Acta Mathematica.

\bibitem{problemyau}
S-T. Yau.
\newblock {Problem Section}.
\newblock In {\em Seminar on Differential Geometry}, Annals of Mathematics Studies, pages 669--706. Princeton University Press, Princeton, 1982.

\end{thebibliography}

\Addresses

\end{document}